\documentclass[hidelinks,onefignum,onetabnum]{siamart251104}

\usepackage{hyperref}

\usepackage{amsmath,amsbsy,amssymb,bm,amscd,mathrsfs,graphicx}
\graphicspath{{./}{../}{figure/}{figures/}{../figure/}{../figures/}}
\usepackage[font=small,labelfont=md,textfont=it]{caption}
\usepackage[titletoc, title]{appendix}
\usepackage[hyperpageref]{backref}
\usepackage{longtable}
\usepackage{pgfplots}
\usepackage{diagbox}
\usepackage{booktabs,makecell,multirow}
\usepackage{cases, color}
\usepackage{enumitem}
\usepackage{comment}
\usepackage{empheq}

\usepackage{subcaption}
\newcommand{\nm}[1]{\left\lVert {#1} \right\rVert}
\newcommand{\dom}[0]{ {\bf dom\,}}
\newcommand{\dd}{\,{\rm d}}
\newcommand{\bs}{\boldsymbol}
\newcommand{\dual}[1]{\left\langle {#1} \right\rangle}
\newcommand{\inprd}[1]{\langle#1\rangle} 
\newcommand{\argmin}[0]{ {\mathrm {argmin}\,}}

\DeclareMathOperator{\sym}{sym}

\renewcommand{\div}{\operatorname{div}}

\newcommand{\fb}{f_{-\mu}}

\newcounter{mnote}
\let\oldmarginpar\marginpar
\renewcommand\marginpar[1]{\-\oldmarginpar[\raggedleft\footnotesize #1]%
	{\raggedright\footnotesize #1}}

\definecolor{algorange}{RGB}{230,159,0}      
\definecolor{algskyblue}{RGB}{86,180,233}    
\definecolor{alggreen}{RGB}{0,158,115}       
\definecolor{algvermillion}{RGB}{213,94,0}   
\definecolor{algpurple}{RGB}{204,121,167}    

\definecolor{algDarkRed}{rgb}{0.6350, 0.0780, 0.1840}      
\definecolor{algDarkBlue}{rgb}{0.0000, 0.4470, 0.6980}      

\usepackage{lipsum}
\usepackage{amsfonts}
\usepackage{graphicx}
\usepackage{epstopdf}
\usepackage{algorithmic}
\ifpdf
  \DeclareGraphicsExtensions{.eps,.pdf,.png,.jpg}
\else
  \DeclareGraphicsExtensions{.eps}
\fi

\newsiamremark{remark}{Remark}
\newsiamremark{hypothesis}{Hypothesis}
\crefname{hypothesis}{Hypothesis}{Hypotheses}
\newsiamthm{claim}{Claim}
\newsiamremark{fact}{Fact}
\crefname{fact}{Fact}{Facts}

\headers{Accelerated Methods Through
Variable and Operator Splitting}{Long Chen, Hao Luo, and Jingrong Wei}

\title{Accelerated Gradient Methods Through \\
Variable and Operator Splitting\thanks{Submitted to the editors October 8, 2026.
\funding{The first author was supported in part by NSF grants DMS-2309777 and DMS-2309785. The second author was supported by the National Natural Science Foundation of China (Grant No. 12401402).}}}

\author{Long Chen
	\thanks{Department of Mathematics, University of California at Irvine, Irvine, CA 92697, USA (\email{chenlong@math.uci.edu})}
	\and
	Hao Luo
	\thanks{National Center for Applied Mathematics in Chongqing, Chongqing Normal University, Chongqing, 401331, China (\email{luohao@cqnu.edu.cn}); Chongqing Research Institute of 
 Big Data, Peking University, Chongqing, 401121, China (\email{luohao@cqbdri.pku.edu.cn})}
\and Jingrong Wei
\thanks{Department of Applied Mathematics, The Hong Kong Polytechnic University, Hung Hom, Kowloon, Hong Kong, China (\email{jingronw@uci.edu})}
}
\usepackage{amsopn}

\ifpdf
\hypersetup{
  pdftitle={Accelerated Gradient Methods Through Variable and Operator Splitting},
  pdfauthor={L. Chen, H. Luo, and J. Wei}
}
\fi

\makeatletter
\renewcommand{\tableofcontents}{\par\bigskip\noindent{\large\bfseries\contentsname}\par\medskip\@starttoc{toc}\par}
\newcommand*{\l@section}{\@dottedtocline{1}{0em}{2.3em}}
\newcommand*{\l@subsection}{\@dottedtocline{2}{2.3em}{3.2em}}
\newcommand*{\l@subsubsection}{\@dottedtocline{3}{5.5em}{4.1em}}
\makeatother

\begin{document}

\maketitle

\begin{abstract}
This survey develops a variable and operator splitting (VOS) framework for accelerated first-order methods. Operator splitting separates the components of the optimality system, while variable splitting introduces auxiliary dynamics that generate acceleration. The strong Lyapunov property guides the stability analysis, and practical schemes follow from Accelerated Over-Relaxation (AOR) and Extrapolation by Predictor--Corrector (EPC) discretizations. For merely convex problems, a perturbed VOS flow combined with homotopy in the auxiliary scale yields accelerated rates. The framework gives a common derivation and analysis of Nesterov-type, FISTA-type, composite, mirror-descent-type, and bilinear saddle-point accelerated methods. Its main contribution is structural and expository: it organizes these methods within one Lyapunov and splitting template and explains how acceleration survives different splittings and discretizations. 
\end{abstract}

\begin{keywords}
Unconstrained convex optimization, First-order methods, Dynamical system, Lyapunov function, Gradient flow, Proximal gradient methods, Acceleration, Saddle-point problems
\end{keywords}

\begin{MSCcodes}
65K10, 65B99, 90C25, 90C47
\end{MSCcodes}

\tableofcontents

\vspace{8pt}


\section{Introduction}

Optimization methods are central to data science, machine learning, inverse problems, and computational mathematics and physics~\cite{FercoqRichta2016Optimization,bottou2018optimization,PehersWillcoGunzbu2018Survey}. First-order methods are simple and scalable, but gradient descent is slow on ill-conditioned problems. Nesterov's accelerated gradient (NAG) methods improve the dependence on the condition number~\cite{nesterov1983method,Nesterov:2013Introductory} and have motivated several approaches to acceleration for convex minimization and structured operator problems \cite{lin2020accelerated,daspremont2021acceleration,wilson2021lyapunov}. This survey presents a variable and operator splitting framework that links accelerated flows, stable discretizations, and Lyapunov analysis. It covers convex minimization, saddle-point problems, and structured monotone equations and inclusions.

Let $V$ be a Hilbert space, with its dual $V^*$ identified with $V$ through the Riesz map. We begin with the operator equation
$$
\mathcal L(x)=\nabla F(x)+N(x)=0,
$$
where $F$ is smooth and convex, and $N:V\to V^*$ is continuous and $\mu$-strongly monotone:
$$
\langle N(x)-N(y),x-y\rangle
\ge
\mu\|x-y\|^2,
\qquad
\mu>0.
$$
Then $\mathcal L$ is continuous, coercive, and $\mu$-strongly monotone, so it has a unique zero $x^\star$; see~\cite{BauschkeCombettes2011}. Nonsmooth and set-valued extensions are treated separately.

A basic continuous model is the generalized gradient flow
\begin{equation}\label{eq:introode}
	x'=-(\nabla F(x)+N(x)).
\end{equation}
For the appropriate operator choices, gradient descent, proximal point methods~\cite{rockafellar1970maximality,rockafellar1976augmented}, and splitting methods~\cite{eckstein1992douglas,chen1997convergence,tseng2000modified,davis2017three,malitsky2020forward,condat2023proximal} can be derived from explicit, implicit, and implicit--explicit Euler steps for \eqref{eq:introode}. For nonsmooth or set-valued $N$, the flow becomes a differential inclusion, while the discrete schemes retain the same structure.

\subsection{Main idea}
Operator splitting separates $\nabla F$ from $N$, while variable splitting introduces an auxiliary variable $y$ that generates acceleration. This gives the {\it variable and operator splitting} (VOS) flow
\begin{equation}\label{eq:introsplitflow}
\begin{cases}
x'=y-x,\\
y'=-\dfrac{1}{\mu}\bigl(\nabla F(x)+N(y)\bigr).
\end{cases}
\end{equation}
Let $\mathcal G(x,y)$ be the vector field of \eqref{eq:introsplitflow}. Its equilibrium is $(x^\star,x^\star)$, where $\mathcal L(x^\star)=0$.

Write $D_F(x,x^\star)=F(x)-F(x^\star)-\langle\nabla F(x^\star),x-x^\star\rangle$ for the Bregman divergence of $F$. Consider the Lyapunov energy
$$
E(x,y)=D_F(x,x^\star)+\frac{\mu}{2}\|y-x^\star\|^2.
$$
A direct calculation gives the strong Lyapunov estimate~\cite{chen2021unified}
\begin{equation}
-\langle\nabla E(x,y),\mathcal G(x,y)\rangle
\ge
E(x,y)+D_F(x^\star,x)+\frac{\mu}{2}\|y-x^\star\|^2,
\end{equation}
which implies exponential energy decay and energy contraction under implicit Euler. The additional nonnegative terms help control discretization errors in explicit and semi-implicit schemes.

To control both variables, set
$$
f=F+\frac{\mu}{2}\|\cdot\|^2,
\qquad
K=N-\mu I.
$$
Then $f$ is $\mu$-strongly convex, $K$ is monotone, and \eqref{eq:introsplitflow} becomes
\begin{equation}\label{eq:introsplitflowfK}
\begin{cases}
x'=y-x,\\
y'=x-y-\dfrac{1}{\mu}\bigl(\nabla f(x)+K(y)\bigr).
\end{cases}
\end{equation}
The energy
$$
E_f(x,y)
=
D_f(x,x^\star)+\frac{\mu}{2}\|y-x^\star\|^2
$$
controls both variables and satisfies
\begin{equation}
-\langle\nabla E_f(x,y),\mathcal G_f(x,y)\rangle
\ge
E_f(x,y)+\frac{\mu}{2}\|y-x\|^2,
\end{equation}
where $\mathcal G_f$ is the vector field of \eqref{eq:introsplitflowfK}. We use $E$ to show the basic cancellation and $E_f$ to prove strong convergence of both variables.

The discrete schemes follow from stability-preserving discretizations of \eqref{eq:introsplitflow} or \eqref{eq:introsplitflowfK}. The fully implicit Euler method inherits the Lyapunov decay but is usually too costly. Practical methods keep the stable part implicit and treat the remaining terms explicitly, so the analysis reduces to controlling the deviation from implicit Euler. We use two tools: Accelerated Over-Relaxation (AOR), which symmetrizes discrete error terms, and Extrapolation by Predictor--Corrector (EPC), which cancels the cross term $\langle \nabla f(x_{k+1}), y_k - y_{k+1}\rangle$ in the discrete energy estimate.

The strong Lyapunov viewpoint~\cite{chen2021unified,chen2023accelerated}, EPC discretizations~\cite{luoDifferentialEquationSolvers2021a}, AGSS methods~\cite{chen2023accelerated}, and AOR-HB~\cite{wei2024accelerated} were developed in our earlier work. The main contribution here is a common Lyapunov framework for these flows, algorithms, and proofs.


\subsection{Acceleration via spectral complexification}

We first illustrate acceleration with a linear example. Let $A$ be a symmetric positive definite (SPD) matrix and consider
$$
x'=-Ax/\mu,
\qquad
L=\lambda_{\max}(A),\quad
\mu=\lambda_{\min}(A),\quad
\kappa(A)=L/\mu\gg1.
$$
The generator $-A/\mu$ has eigenvalues in $[-\kappa,-1]$. Forward Euler converges for $0<\alpha<2/\kappa$. With a suitable step of order $\kappa^{-1}$, it requires $\mathcal O(\kappa)$ iterations to reduce the error by a fixed factor.

\begin{figure}[htp]
\begin{center}
\includegraphics[width=3in]{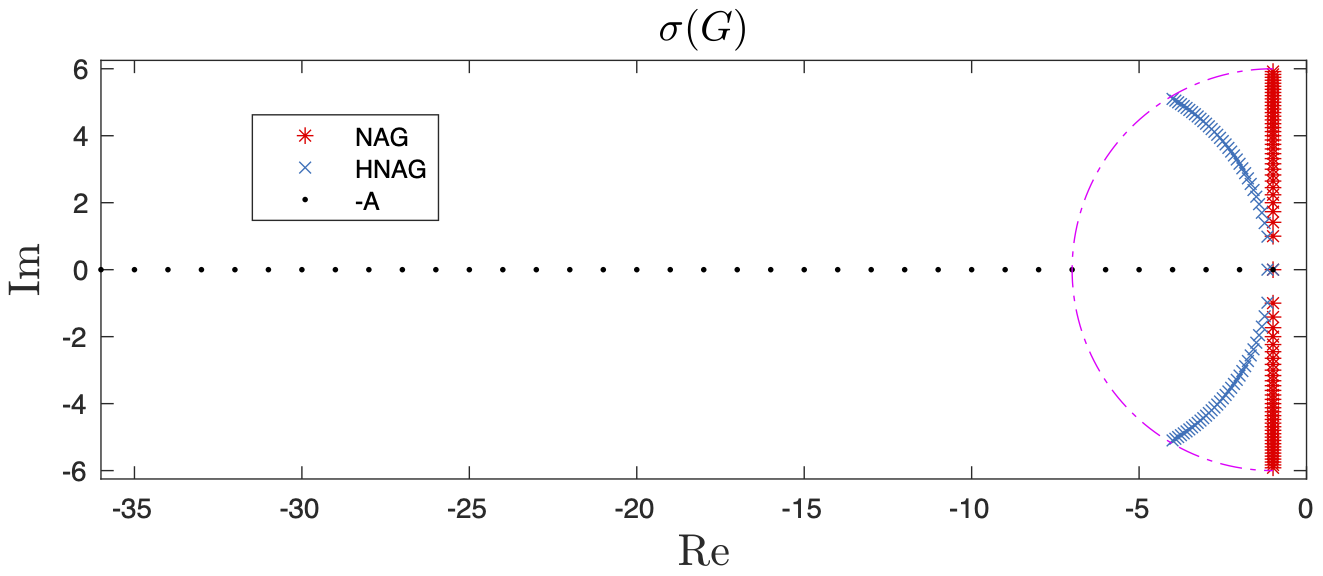}
\caption{Spectra of the gradient-flow generator $-A/\mu$ and accelerated linear flows, with $\mu=1$. The accelerated spectra lie in the disk centered at $(-1,0)$ with radius $\sqrt{\kappa(A)}$.}
\label{fig:spectrum}
\end{center}
\end{figure}

The VOS flow gives a different spectral picture. For $f(x)=\frac12(Ax,x)$ and $K=0$, the generator of \eqref{eq:introsplitflowfK} and its block for the largest eigenvalue are
$$
G_{_{\rm NAG}}=
\begin{pmatrix}
-I&I\\
I-A/\mu&-I
\end{pmatrix},
\qquad
R=
\begin{pmatrix}
-1&1\\
1-\kappa&-1
\end{pmatrix},
\qquad 
\kappa \geq 1.
$$
Here $R$ is the restriction of $G_{_{\rm NAG}}$ to the eigenmode of $A/\mu$ associated with its largest eigenvalue $\kappa$. The eigenvalues of $R$ satisfy
$$
\lambda^2-(\operatorname{tr}R)\lambda+\det R=0.
$$
Since $(\operatorname{tr}R)^2<4\det R$ for $\kappa>1$, they form a complex conjugate pair and satisfy
$$
\lambda\bar{\lambda} = |\lambda|^{2} = \det R = \kappa, \qquad \lambda + \bar{\lambda} = 2\Re(\lambda) = \operatorname{tr}R = -2,
$$
which means
$$
|\lambda| = \sqrt{\kappa}, \qquad \Re(\lambda) = -1 < 0.
$$
Thus the spectral scale changes from $\kappa$ to $\sqrt{\kappa}$ while the real part remains negative. Figure~\ref{fig:spectrum} illustrates this spectrum change; more such transformations appear in the appendix of~\cite{luoDifferentialEquationSolvers2021a}.

This spectral argument guides the choice of discretization but does not establish convergence for nonlinear problems. Forward Euler still requires $0<\alpha<2/\kappa$ for convergence. The block structure permits stable triangular implicit discretizations that need $\mathcal O(\sqrt{\kappa})$ iterations for a fixed error reduction. We establish convergence through Lyapunov estimates.

\subsection{Strong Lyapunov property}

Our convergence analysis is based on the strong Lyapunov property introduced in~\cite{chen2021unified,chen2023accelerated}. Let
\begin{equation}\label{eq:introzode}
\bm z'(t)=\mathcal G(\bm z(t)),\qquad t>0,\qquad \bm z(0)=\bm z_0,
\end{equation}
and let $\bm z^\star$ be an equilibrium point of $\mathcal G$. A differentiable function $E$ is called a Lyapunov function for \eqref{eq:introzode} if
$$
E(\bm z)\ge0,
\qquad
E(\bm z^\star)=0,
\qquad
-\langle\nabla E(\bm z),\mathcal G(\bm z)\rangle\ge0.
$$
If, in addition, $E(\bm z)=0$ implies $\bm z=\bm z^\star$, then $E$ is positive definite with respect to $\bm z^\star$. When this property is unavailable, we call $E$ a Lyapunov energy. A Lyapunov function or energy has the strong Lyapunov property if, for some $c_L>0$,
\begin{equation}\label{eq:introLyp-cond}
-\langle\nabla E(\bm z),\mathcal G(\bm z)\rangle
\ge
c_LE(\bm z).
\end{equation}
The constant $c_L$ is analogous to the coercivity constant in PDE theory. Along solutions of \eqref{eq:introzode},
$$
\frac{\dd}{\dd t}E(\bm z(t))
=
\langle\nabla E(\bm z(t)),\mathcal G(\bm z(t))\rangle
\le
-c_LE(\bm z(t)),
\qquad
E(\bm z(t))
\le
E(\bm z_0)e^{-c_Lt}.
$$
Thus the Lyapunov energy decays exponentially. Convergence of the state follows when the energy also controls the distance to the equilibrium.

The usefulness of \eqref{eq:introLyp-cond} extends beyond continuous time. Applying implicit Euler to \eqref{eq:introzode} gives
$$
\bm z_{k+1}-\bm z_k=\alpha\mathcal G(\bm z_{k+1}).
$$
If $E$ is convex, then
$$
\begin{aligned}
E(\bm z_{k+1})-E(\bm z_k)
&\le \langle\nabla E(\bm z_{k+1}),\bm z_{k+1}-\bm z_k\rangle
&&\text{(convexity of $E$)}\\
&=\alpha\langle\nabla E(\bm z_{k+1}),\mathcal G(\bm z_{k+1})\rangle
&&\text{(implicit Euler)}\\
&\le-c_L\alpha E(\bm z_{k+1})
&&\text{(strong Lyapunov property)}.
\end{aligned}
$$
Hence, for any $\alpha>0$, the implicit Euler method satisfies the linear contraction
\begin{equation}
E(\bm z_{k+1})
\le
\frac{1}{1+c_L\alpha}E(\bm z_k).
\end{equation}
This inheritance principle underlies the discrete analysis. 

\subsection{Stability-preserving discretizations}

A fully implicit Euler step is stable but often impractical. The VOS form permits sequential updates that keep part of the coupling implicit. The analysis follows one principle: start from an unconditionally stable implicit Euler step and control the error introduced. Specifically, we use two discretization tools.

The first is {\it Accelerated Over-Relaxation} (AOR), originally developed for linear algebraic systems~\cite{hadjidimos1978accelerated}. In primal--dual splitting~\cite{chambolle2011first,he2012convergence}, the choice $\theta=1$ gives the extrapolation $\bar x_{k+1}=2x_{k+1}-x_k$. The extrapolated gradient $2\nabla F(x_{k+1})-\nabla F(x_k)$ used below has the same form as Popov's optimistic step~\cite{popov1980modification} for saddle-point problems. The VOS analysis interprets this extrapolation as a symmetrization device rather than a prediction of the next gradient. This structure allows cross terms to be expanded by difference-of-squares identities and yields sharp Lyapunov estimates.

The second is {\it Extrapolation by Predictor--Corrector} (EPC), a standard idea in numerical ODEs~\cite{Gautschi:2011Numerical} and recently used for optimization flows~\cite{luoDifferentialEquationSolvers2021a}. The predictor--corrector scheme follows the continuous flow more closely, while the correction is chosen to preserve its Lyapunov estimate.

\subsection{Applications and scope}\label{sec:application}

The framework applies to several structured first-order problems.
\begin{enumerate}[leftmargin=15pt,itemsep=3pt]

\item {\it Strongly convex optimization.} For a smooth $\mu$-strongly convex objective $f$,
$$
\min_{x\in V} f(x),
\qquad
\nabla F(x)=\nabla f(x)-\mu x,
\qquad
N(y)=\mu y,
$$
the VOS framework recovers Nesterov-type schemes~\cite{Nesterov:2013Introductory}.

\item {\it Composite convex optimization.} For a smooth $\mu$-strongly convex function $f$ and a convex nonsmooth function $g$,
$$
\min_{x\in V} f(x)+g(x),
\qquad
\nabla F(x)=\nabla f(x)-\mu x,
\qquad
N(y)=\partial g(y)+\mu y,
$$
the nonsmooth part is treated implicitly, giving accelerated proximal-type schemes related to FISTA and APG~\cite{beck2009fast,li2015accelerated}.

\item {\it Saddle-point problems with bilinear coupling.} For
$$
\min_{u\in\mathbb R^m}\max_{p\in\mathbb R^n}\mathcal H(u,p)
=
f(u)-g(p)+(Bu,p),
$$
where $f$ and $g$ are smooth and strongly convex with parameters $\mu_f$ and $\mu_g$, the splitting
$$
\underbrace{\begin{pmatrix}\nabla f(u)-\mu_f u\\ \nabla g(p)-\mu_g p\end{pmatrix}}_{\nabla F}
+
\underbrace{
\begin{pmatrix}0&B^{\intercal}\\ -B&0\end{pmatrix}
\begin{pmatrix}v\\ q\end{pmatrix}
+
\begin{pmatrix}\mu_f v\\ \mu_g q\end{pmatrix}
}_{N}
$$
fits the VOS form, with $x=(u,p)^{\intercal}$ and $y=(v,q)^{\intercal}$. Thus the same Lyapunov mechanism handles gradient dynamics and skew-symmetric coupling, giving a structured route to primal--dual acceleration~\cite{thekumparampil2022lifted,kovalev2022accelerated,jin2022sharper,yoon2021accelerated,Bot2023a,Attouch2022,He2021,Luo2024b}.

\item {\it Gradient and skew-symmetric operators.} When $K=N-\mu I$ is linear and skew-symmetric, the VOS framework gives Accelerated Gradient and Skew-Symmetric Splitting (AGSS) methods~\cite{chen2023accelerated}. These methods extend Hermitian and skew-Hermitian splitting (HSS)~\cite{bai2003hermitian,bai2007accelerated} from linear systems to nonlinear monotone equations, such as nonlinear diffusion--convection--reaction equations:
$$
\nabla F(u)+\bm b\cdot\nabla u+cu=0,
\qquad
\div\bm b=0.
$$

\item {\it Non-Euclidean settings.} For a convex minimization problem $\min_x f(x)$, let $\phi$ be a strongly convex mirror function. Assume that $f$ is relatively $\mu$-strongly convex and relatively $L$-smooth with respect to $\phi$, i.e., $f-\mu\phi$ and $L\phi- f$ are convex. A mirror version of the VOS flow is
$$
\left\{
\begin{aligned}
x'&=y-x,\\
(\nabla\phi(y))'&=-\mu^{-1}\nabla f(x)+\nabla\phi(x)-\nabla\phi(y).
\end{aligned}
\right.
$$
Discretizing this flow under mild additional assumptions yields the accelerated mirror descent methods proposed in \cite{ChenLuoWeiXu2025Accelerated}. 


\item {\it Merely convex problems.} For $\mu=0$, we use a projected homotopy strategy that changes the auxiliary time scale while preserving the equilibrium $x^\star$. Under the assumptions in Section~\ref{sec: convex homotopy}, the method reaches energy accuracy $\epsilon$ in $\mathcal O(\sqrt{L_F/\epsilon})$ iterations, with constants depending on the ball and initial data; here $L_F$ is the Lipschitz constant of $\nabla F$. This gives the accelerated rate $\mathcal O(1/k^2)$ at stage endpoints. The strategy also avoids estimating $\mu$ when $\mu>0$ is unknown.

\item {\it Monotone inclusions with skew-symmetric coupling.} We extend the framework to monotone inclusions of the form
$$
0\in\mathcal L(x)=\nabla f(x)+\partial g(x)+Kx,
$$
which separate the smooth gradient part, the nonsmooth monotone part, and the skew-symmetric coupling. The resulting flow combines Hessian-driven damping with skew-symmetric coupling.

\end{enumerate}
\medskip

These examples show the scope of a single design and proof template. The framework unifies existing methods and guides the design and analysis of accelerated schemes for structured operator problems.

\subsection{Relation to existing approaches}

Acceleration has a large and rapidly growing literature. We give a selective comparison with established approaches closest to the present framework.



{\it Estimate sequences and common energy arguments.} Estimate sequences~\cite{nesterov1983method} and linear coupling~\cite{AllenZhuOrecchia2017} provide constructive accounts of acceleration. Related potential-function methods give sharp analyses for smooth, composite, and proximal problems~\cite{bansal2019potential,taylor2021optimal_gradient_method}. Wilson, Recht, and Jordan~\cite{wilson2021lyapunov} relate estimate sequences to Lyapunov functions, while Diakonikolas and Orecchia~\cite{diakonikolas2019approximate} derive first-order methods and control their discretization errors through an approximate duality-gap identity. These approaches use common energy arguments for both analysis and algorithm design.


{\it Performance estimation and control methods.} The performance estimation problem (PEP)~\cite{DroriTeboulle2014,TaylorHendrickxGlineur2017,ryu2020operator} yields worst-case bounds and optimized coefficients. Integral quadratic constraints (IQC)~\cite{lessard2016analysis} and automated Lyapunov methods~\cite{taylor2018lyapunov,moucer2023systematic} provide systematic convergence certificates. These tools can assess the sharpness of energy estimates. Applying them to broader operator settings requires a suitable oracle model and constraints that describe the operator class.


{\it Continuous-time approaches.} Heavy-Ball flow~\cite{polyak1964some}, Nesterov-type ODEs~\cite{su2016differential}, Bregman Lagrangians~\cite{wibisono2016variational}, and inertial proximal dynamics~\cite{AIAlvarezAttouch2001,AIAttouchPeypouquetRedont2014} explain acceleration through damping and geometry. High-resolution ODEs retain additional discrete effects~\cite{shi2022understanding}. Connections with numerical integration and continuous and discrete Lyapunov functions are developed in~\cite{scieur2017integration,sanz2021connections}. These models guide algorithm design, but a central difficulty is the discretization gap: a stable or fast flow does not automatically yield a stable or fast algorithm, and continuous-time estimates need not carry over to the discrete scheme.




{\it Primal--dual and operator splitting.} Monotone-operator splitting separates proximal, smooth, and skew terms~\cite{AIBricenoAriasCombettes2011,condat2023proximal,RyuYin2022Large-Scale}. Accelerated methods for structured saddle-point problems include the primal--dual schemes of Chen, Lan, and Ouyang~\cite{ChenLanOuyang2014}, inertial corrected splitting~\cite{AIValkonen2020}, and accelerated Condat--V{\~u}~\cite{DriggsEhrhardtSchonliebTang2024}. Trajectory-based adaptive acceleration is also studied for operator splitting~\cite{poon2020geometry}. These methods combine acceleration with operator splitting; their assumptions, error measures, and update rules differ.

VOS complements these approaches by making a common cancellation and discretization argument explicit across the problems in Section~\ref{sec:application}. Implicit Euler provides a stable reference, and AOR and EPC control the errors introduced by explicit updates. This common proof structure is the focus of the survey.

The rest of the paper develops this template. We first review convex analysis and strong Lyapunov estimates, then derive VOS flows and stability-preserving discretizations, and finally apply them to convex optimization, saddle-point systems, and related operator problems.

\section{Bounds on convex functions}
\label{sec:bd-convex}

This section reviews basic bounds for convex functions~\cite[Chapter~2]{Nesterov:2013Introductory}. Let $V$ be a Hilbert space with inner product $(\cdot,\cdot)$ and induced norm $\|\cdot\|$. Let $V^*$ denote its dual space, with pairing $\langle\cdot,\cdot\rangle$. We identify $V^*$ with $V$ through the Riesz map associated with the chosen inner product. Unless stated otherwise, convex functions are real-valued and defined on $V$. Extended-valued functions $f:V\to\mathbb R\cup\{+\infty\}$ are used only for nonsmooth terms or constraints. Their effective domain is $\dom f:=\{x\in V:f(x)<+\infty\}$.

\subsection{Convex functions}

A continuous function $f:V\to\mathbb R$ is {\em convex} if
\begin{equation*}
	f(\alpha x+(1-\alpha)y)
	\leqslant
	\alpha f(x)+(1-\alpha)f(y)
	\qquad
	\forall\,x,y\in V,\quad \forall\,\alpha\in[0,1].
\end{equation*}
It is {\it strictly convex} if the inequality is strict whenever $x\ne y$ and $\alpha\in(0,1)$. A convex function is {\it $\mu$-strongly convex}, with $\mu>0$, if
$$
f_{-\mu}(\cdot):=f(\cdot)-\frac{\mu}{2}\|\cdot\|^2
$$
is convex. The notation $f_{+\mu}$ denotes $f+\frac{\mu}{2}\|\cdot\|^2$.

The function $f$ is {\em coercive} if
$$
f(x)\to+\infty
\qquad
\text{as }\|x\|\to\infty.
$$
Every $\mu$-strongly convex function with $\mu>0$ is coercive. Convexity alone does not imply coercivity; for example, $f(x)=e^{-x}$ is convex on $\mathbb R$ but not coercive. The direct method gives the following result: coercivity bounds minimizing sequences, weak compactness gives a weakly convergent subsequence, and convexity and continuity imply weak lower semicontinuity. For details, see~\cite[Proposition~1.2]{ekeland_convex_1987} or~\cite[Theorem~8.2.2]{Ciarlet89}.

\begin{theorem}
	\label{thm:wellposedS3}
	If $f:V\to\mathbb R$ is continuous, convex, and coercive, then the problem $\min_{x\in V}f(x)$ has a solution $x^\star\in V$. If $f$ is strictly convex, then the solution is unique.
\end{theorem}

Let $\mathcal C^1$ denote the class of continuously differentiable functions on $V$. For $0\le L<\infty$, let $\mathcal C_L^{1,1}$ denote the functions in $\mathcal C^1$ whose gradients are Lipschitz continuous with constant $L$:
\begin{equation*}
	\|\nabla f(x)-\nabla f(y)\|_*
	\leqslant
	L\|x-y\|
	\qquad
	\forall\,x,y\in V,
\end{equation*}
where $\|\cdot\|_*$ is the dual norm on $V^*$.

Let $\mathcal S^0_0$ denote the class of proper, closed, convex functions, and let $\mathcal S^0_\mu$, for $\mu>0$, denote those that are $\mu$-strongly convex. Here proper means that $\dom f\ne\emptyset$ and $f$ never takes the value $-\infty$; closed means lower semicontinuous. The superscript $0$ indicates that differentiability is not assumed. For $\mu\geqslant0$, set
$$
\mathcal S^1_\mu:=\mathcal S^0_\mu\cap\mathcal C^1,
\qquad
\mathcal S_{\mu,L}^{1,1}:=\mathcal S_\mu^1\cap\mathcal C_L^{1,1}.
$$
For $0\leqslant\mu\leqslant L<\infty$, define
$$
\mathcal S^1_{\mu,L}
:=
\left\{
f\in\mathcal S^1_0:
\mu\|x-y\|^2
\leqslant
\langle\nabla f(x)-\nabla f(y),x-y\rangle
\leqslant
L\|x-y\|^2
\quad
\forall\,x,y\in V
\right\}.
$$
For continuously differentiable convex functions on $V$, the characterizations of smoothness and strong convexity give $\mathcal S_{\mu,L}^{1,1}=\mathcal S^1_{\mu,L}$; see~\cite{lessard2016analysis}. When $\mu>0$, the ratio $\kappa=L/\mu$ is the condition number of $f$.

Let $A$ be a bounded SPD operator on $V$. Then
$$
(x,y)_A:=(Ax,y) = (x, Ay)
$$
defines an inner product, with induced norm $\|\cdot\|_A$. Its dual norm is denoted by $\|\cdot\|_{A^{-1}}$. In this metric, the strong convexity and gradient Lipschitz constants satisfy
\begin{equation*}
	\mu_A\|x-y\|_A^2
	\leqslant
	\langle\nabla f(x)-\nabla f(y),x-y\rangle
	\leqslant
	L_A\|x-y\|_A^2
	\qquad
	\forall\,x,y\in V.
\end{equation*}
A suitable metric may reduce the condition number to $\kappa_A=L_A/\mu_A\ll\kappa$.

For any bounded self-adjoint operator $A$, the bilinear form $(\cdot,\cdot)_A$ remains well defined. We use $\|x\|_A^2:=(Ax,x)$ for its quadratic form even when $A$ is not positive definite; this notation need not define a norm. The following identity still holds:
\begin{equation}\label{eq:squares}
	(x,y)_A
	=
	\frac12\|x+y\|_A^2
	-\frac12\|x\|_A^2
	-\frac12\|y\|_A^2
	\qquad
	\forall\,x,y\in V.
\end{equation}

\subsection{Bregman divergence and basic bounds}
\label{sec:Breg}

For $f\in\mathcal C^1$, define the Bregman difference
$$
D_f(y,x)
:=
f(y)-f(x)-\langle\nabla f(x),y-x\rangle.
$$
If $f$ is convex, then $D_f(y,x)\ge0$, and $D_f$ is called the {\it Bregman divergence} associated with $f$. If $f$ is strictly convex, then $D_f(y,x)=0$ if and only if $x=y$.

The map $f\mapsto D_f$ is linear and invariant under affine shifts:
$$
D_{\alpha f+\beta g}=\alpha D_f+\beta D_g,
\qquad
D_{f+p}=D_f
$$
for every affine function $p$. For fixed $x$, differentiation in the first variable gives
\begin{equation}\label{eq:gradBregman}
	\nabla D_f(\cdot,x)=\nabla f(\cdot)-\nabla f(x).
\end{equation}
Its Bregman divergence is
$$
D_{D_f(\cdot,x)}(z,y)
=
D_f(z,x)-D_f(y,x)-\langle\nabla D_f(y,x),z-y\rangle.
$$
Since $D_f(\cdot,x)$ and $f$ differ by an affine function, they have the same Bregman divergence:
$D_{D_f(\cdot,x)}(z,y)=D_f(z,y).$
Together with \eqref{eq:gradBregman}, this gives the three-point identity below, which generalizes \eqref{eq:squares} and expresses gradient cross terms through Bregman differences.

\begin{lemma}[Three-point identity for Bregman divergence~\cite{chen1993convergence}]
	\begin{equation}
		\langle\nabla f(y)-\nabla f(x),z-y\rangle
		=
		D_f(z,x)-D_f(y,x)-D_f(z,y).
	\end{equation}
\end{lemma}

In general, $D_f$ is not symmetric. Define its symmetrization by
\begin{equation*}
	M_{\nabla f}(x,y)
	:=
	\frac12\bigl(D_f(y,x)+D_f(x,y)\bigr)
	=
	\frac12\langle\nabla f(x)-\nabla f(y),x-y\rangle.
\end{equation*}
By the fundamental theorem of calculus,
\begin{equation}\label{eq:Df-int}
	\begin{aligned}
		D_f(y,x)
		&=
		\dual{\int_0^1\bigl(\nabla f(x+t(y-x))-\nabla f(x)\bigr)\dd t,\,y-x} \\
		&=
		\int_0^1 2M_{\nabla f}(x_t,x)\,\frac{\dd t}{t},
		\qquad
		x_t:=x+t(y-x).
	\end{aligned}
\end{equation}
Thus pointwise bounds on $M_{\nabla f}$ yield corresponding bounds on $D_f$.

\begin{lemma}\label{lem:Bregmandiv}
	For $f\in\mathcal C_L^{1,1}$,
	\begin{align}
		\label{eq:DL}
		\max\{D_f(y,x),M_{\nabla f}(x,y)\}
		\leqslant
		\frac{L}{2}\|x-y\|^2.
	\end{align}
	For $f\in\mathcal S_\mu^1$ with $\mu\geqslant0$,
	\begin{align}
		\label{eq:Dmu}
		\min\{D_f(y,x),M_{\nabla f}(x,y)\}
		\geqslant
		\frac{\mu}{2}\|x-y\|^2.
	\end{align}
	The inequalities hold for all $x,y\in V$.
\end{lemma}

\begin{proof}
	The bound \eqref{eq:DL} for $M_{\nabla f}(x,y)$ follows from the Cauchy--Schwarz inequality and the Lipschitz continuity of $\nabla f$. Applying this bound in \eqref{eq:Df-int} and using $\|x_t-x\|^2=t^2\|y-x\|^2$ gives the bound for $D_f(y,x)$. Interchanging $x$ and $y$ gives the same bound for $D_f(x,y)$.

	The bound \eqref{eq:Dmu} for $D_f(y,x)$ follows from
	$$
	D_{f_{-\mu}}(y,x)
	=
	D_f(y,x)-\frac{\mu}{2}\|y-x\|^2
	\ge0,
	$$
	since $f_{-\mu}$ is convex. Interchanging $x$ and $y$ gives the same bound for $D_f(x,y)$. Averaging these two bounds gives the lower bound for $M_{\nabla f}(x,y)$.
\end{proof}

\subsection{Conjugate function}

For any proper, closed, convex function $f:V\to\mathbb R\cup\{+\infty\}$, its {\it conjugate function}, or Fenchel dual, is
$$
f^*(\chi)
:=
\sup_{x\in V}\{\langle\chi,x\rangle-f(x)\},
\qquad
\chi\in V^*.
$$
If $f\in\mathcal S^{1,1}_{\mu,L}$ with $0<\mu\le L<\infty$, then $f^*\in\mathcal S^{1,1}_{1/L,1/\mu}$; see~\cite[Theorem~5.26]{beck2017first}. In this case, $\nabla f:V\to V^*$ is a bijection with inverse $\nabla f^*:V^*\to V$. Strong monotonicity gives injectivity. For each $\chi\in V^*$, the function $x\mapsto f(x)-\langle\chi,x\rangle$ is strongly convex and coercive, so it has a unique minimizer satisfying $\nabla f(x)=\chi$. Hence $\nabla f$ is surjective.

Let $\chi,\eta\in V^*$ and set
$$
x:=\nabla f^*(\chi),
\qquad
y:=\nabla f^*(\eta).
$$
Then $\chi=\nabla f(x)$ and $\eta=\nabla f(y)$, so
\begin{equation}\label{eq:M-id}
\langle\nabla f^*(\chi)-\nabla f^*(\eta),\chi-\eta\rangle
=
\langle x-y,\nabla f(x)-\nabla f(y)\rangle.
\end{equation}
The Fenchel--Young identity gives
\begin{equation}\label{eq:DfDf*}
	D_f(y,x)
	=
	D_{f^*}(\chi,\eta),
	\qquad
	\chi=\nabla f(x),
	\quad
	\eta=\nabla f(y).
\end{equation}
The arguments appear in reverse order: $D_f(y,x)=D_{f^*}(\nabla f(x),\nabla f(y))$.

For $f\in\mathcal S^{1,1}_{\mu,L}$ with $0<\mu\le L<\infty$, applying Lemma~\ref{lem:Bregmandiv} to $f^*$ and using \eqref{eq:M-id}--\eqref{eq:DfDf*} gives bounds in terms of gradient differences. More generally, the lower bound requires only convexity and smoothness, and the upper bound requires only strong convexity and differentiability.

\begin{lemma}\label{lem:dual-bregman-bounds}
	For $f\in\mathcal S^1_{0,L}$ with $L>0$, the gradient $\nabla f$ is $1/L$-co-coercive, and
	\begin{align}
		\label{eq:philowerL}
		\min\{D_f(y,x),M_{\nabla f}(x,y)\}
		\ge
		\frac{1}{2L}\|\nabla f(y)-\nabla f(x)\|_*^2.
	\end{align}
	For $f\in\mathcal S^1_\mu$ with $\mu>0$, the gradient $\nabla f$ is $\mu$-strongly monotone, and
	\begin{align}
		\label{eq:phiuppermu}
		\max\{D_f(y,x),M_{\nabla f}(x,y)\}
		\le
		\frac{1}{2\mu}\|\nabla f(y)-\nabla f(x)\|_*^2.
	\end{align}
	The inequalities hold for all $x,y\in V$.
\end{lemma}

We next recall a refined lower bound due to Nesterov~\cite{Nesterov:2013Introductory}.

\begin{lemma}\label{lem:refined-monotonicity}
	For $f\in\mathcal S^1_{\mu,L}$ with $0\leqslant\mu\leqslant L<\infty$ and $L>0$,
	\begin{equation}\label{eq:refineM}
		\dual{\nabla f(x)-\nabla f(y),x-y}
		\geqslant
		\frac{\mu L}{\mu+L}\|x-y\|^2
		+
		\frac{1}{\mu+L}\|\nabla f(x)-\nabla f(y)\|_*^2
	\end{equation}
	for all $x,y\in V$.
\end{lemma}

\begin{proof}
	If $L=\mu$, then $\nabla f(x)-\nabla f(y)=\mu(x-y)$, and \eqref{eq:refineM} holds with equality. If $L>\mu$, apply \eqref{eq:philowerL} to $f_{-\mu}\in \mathcal S^1_{0,L-\mu}$:
	\begin{equation}\label{eq:Mpsi}
		\dual{\nabla f_{-\mu}(x)-\nabla f_{-\mu}(y),x-y}
		\geqslant
		\frac{1}{L-\mu}
		\nm{\nabla f_{-\mu}(x)-\nabla f_{-\mu}(y)}_*^2.
	\end{equation}
	Expanding \eqref{eq:Mpsi} in terms of $f$ gives \eqref{eq:refineM}.
\end{proof}

We frequently use these bounds when one point is a global minimizer $x^\star$. For $f\in\mathcal C^1$, this gives $\nabla f(x^\star)=0$ and
$$
D_f(x,x^\star)=f(x)-f(x^\star),
\qquad
2M_{\nabla f}(x,x^\star)
=
\langle\nabla f(x),x-x^\star\rangle.
$$
Similar bounds for nonsmooth strongly convex functions use subgradients.


\section{Strong Lyapunov functions}
\label{sec:strong-Lyapunov-functoin}

We review the strong Lyapunov property introduced in~\cite{chen2021unified} and its role in energy decay. Consider the dynamical system
\begin{equation}\label{eq:ode-G}
	x'(t)=\mathcal G(x(t)),\qquad t>0,
\end{equation}
where $x:\mathbb R_+\to V$ and $\mathcal G:V\to V^*$ is a vector field. As before, we identify $V^*$ with $V$ through the Riesz map. Let $x^\star$ be an equilibrium, so $\mathcal G(x^\star)=0$. We study the convergence of $x(t)$ to $x^\star$ as $t\to\infty$.

\subsection{Strong Lyapunov condition and decay}

A quantitative Lyapunov condition gives energy decay rates. Let $\mathcal E\in\mathcal C^1$ satisfy $\mathcal E(x^\star)=0$. We call $\mathcal E$ a Lyapunov energy when it may vanish away from $x^\star$. Suppose that there are a set $W\subseteq V$ containing $x^\star$, constants $c>0$ and $q\ge1$, and a nonnegative continuous function $p:V\to\mathbb R$ such that $\mathcal E\ge0$ on $W$ and
\begin{equation}\label{eq:A}
	-\inprd{\nabla\mathcal E(x),\mathcal G(x)}
	\ge
	c\,\mathcal E^q(x)+p^2(x)
	\qquad \forall x\in W.
\end{equation}
Then $\mathcal E$ has the strong Lyapunov property on $W$. The condition is global if $W=V$ and local if it holds on a neighborhood of $x^\star$. Energy decay implies convergence to $x^\star$ when $\mathcal E$ also controls the distance to $x^\star$.

\begin{theorem}\label{thm:strongLya}
	Assume that $\mathcal E$ satisfies \eqref{eq:A}. If the trajectory $x(t)$ of \eqref{eq:ode-G} satisfies $x(t)\in W$ for all $t\ge 0$, then, with $\mathcal E_0=\mathcal E(x(0))$,
	\begin{equation}\label{eq:ineq-L-q}
		\mathcal E(x(t))
		+
		\int_0^t e^{c(s-t)}p^2(x(s))\,\dd s
		\le
		\mathcal E_0 e^{-ct},
		\qquad q=1,
	\end{equation}
	and
	\begin{equation}\label{eq:ineq-poly-q}
		\mathcal E(x(t))
		\le
		\left((q-1)ct+\mathcal E_0^{1-q}\right)^{-1/(q-1)},
		\qquad q>1.
	\end{equation}
	For $q>1$, the right-hand side is understood as zero when $\mathcal E_0=0$.
\end{theorem}

\begin{proof}
	Along the trajectory,
	\begin{equation}\label{eq:diff-L-q}
		\begin{aligned}
			\frac{\dd}{\dd t}\mathcal E(x(t))
			&=
			\inprd{\nabla\mathcal E(x(t)),x'(t)}
			=
			\inprd{\nabla\mathcal E(x(t)),\mathcal G(x(t))} \\
			&\le
			-c\,\mathcal E^q(x(t))-p^2(x(t)).
		\end{aligned}
	\end{equation}
	For $q=1$, multiplying \eqref{eq:diff-L-q} by $e^{ct}$ gives
	$$
		\frac{\dd}{\dd t}\bigl(e^{ct}\mathcal E(x(t))\bigr)
		\le
		-e^{ct}p^2(x(t)).
	$$
	Integrating from $0$ to $t$ gives \eqref{eq:ineq-L-q}.

	For $q>1$, assume first that $\mathcal E_0>0$ and that the energy remains positive up to time $t$. Then \eqref{eq:diff-L-q} gives
	$$
		\frac{\dd}{\dd t}\mathcal E^{1-q}(x(t))
		=
		(1-q)\frac{\frac{\dd}{\dd t}\mathcal E(x(t))}{\mathcal E^q(x(t))}
		\ge
		c(q-1).
	$$
	Thus
	$$
		\mathcal E^{1-q}(x(t))-\mathcal E_0^{1-q}
		\ge
		c(q-1)t.
	$$
	Rearranging gives \eqref{eq:ineq-poly-q}. By \eqref{eq:diff-L-q}, the energy is nonincreasing and remains zero once it reaches zero. Thus the bound also holds after that time and when $\mathcal E_0=0$.
\end{proof}

The vector field $\mathcal G$ may be set-valued, as in nonsmooth convex optimization. Then \eqref{eq:ode-G} is replaced by the differential inclusion
\begin{equation}\label{eq:di}
	x'(t)\in\mathcal G(x(t)),\qquad t>0.
\end{equation}
When the vector field depends on a subgradient $d\in\partial f(x)$, we write $\mathcal G(x,d)$. Along an absolutely continuous trajectory, choose a measurable selection $d(t)\in\partial f(x(t))$ such that, for almost every $t$,
$x'(t)=\mathcal G(x(t),d(t)).$
If $\mathcal E$ is also nonsmooth and convex, let $\partial\mathcal E(x,d)$ denote a chosen element of its subdifferential. The strong Lyapunov condition becomes
\begin{equation}
	-\inprd{\partial\mathcal E(x,d),\mathcal G(x,d)}
	\ge
	c\,\mathcal E^q(x)+p^2(x),
	\qquad \forall x\in W.
\end{equation}
The inequality must hold for the selections that generate the trajectory and satisfy the chain rule.

We do not develop the existence theory for \eqref{eq:di}. If $-\mathcal G$ is maximal monotone and $x(0)\in\dom\mathcal G$, classical monotone evolution theory gives a unique locally absolutely continuous solution; see~\cite{RyuBoyd2016Primer} and references therein. 

The main cases are $q=1$, which gives exponential energy decay, and $q=2$, which gives $\mathcal O(1/t)$ decay. Most applications use $W=V$; local estimates require the trajectory to remain in $W$.

\subsection{Implicit Euler inheritance}

The strong Lyapunov property gives a direct contraction proof for the implicit Euler method applied to $x'=\mathcal G(x)$.

\begin{theorem}\label{th:implicitEuler}
	Assume that $\mathcal E$ is nonnegative, convex, and continuously differentiable and satisfies
	\begin{equation}
		-\inprd{\nabla\mathcal E(x),\mathcal G(x)}
		\ge
		c_L\mathcal E(x),
		\qquad \forall x\in V,
	\end{equation}
	for some $c_L>0$. Assume that each implicit step has a solution, and let $\{x_k\}$ satisfy
	\begin{equation}
		x_{k+1}-x_k=\alpha_k\mathcal G(x_{k+1}),
		\qquad \alpha_k>0.
	\end{equation}
	Then the energy satisfies
	\begin{equation}\label{eq:implicit-euler-contract}
		\mathcal E(x_{k+1})
		\le
		\frac{1}{1+c_L\alpha_k}\mathcal E(x_k),
		\qquad k\ge 0.
	\end{equation}
\end{theorem}

\begin{proof}
	Let $\mathcal E_k=\mathcal E(x_k)$. By convexity of $\mathcal E$ and the strong Lyapunov property,
	\begin{align*}
		\mathcal E_{k+1}-\mathcal E_k
		&\le
		\inprd{\nabla\mathcal E(x_{k+1}),x_{k+1}-x_k} =
		\alpha_k\inprd{\nabla\mathcal E(x_{k+1}),\mathcal G(x_{k+1})} \le
		-c_L\alpha_k\mathcal E_{k+1}.
	\end{align*}
	Rearranging gives \eqref{eq:implicit-euler-contract}.
\end{proof}



\section{Gradient flow and Euler methods}
\label{sec: GD methods}

For a continuously differentiable function $f:V\to\mathbb R$, consider $\min_{x\in V}f(x)$. The associated gradient flow is
\begin{equation}\label{eq:gf}
	x'(t)=-\nabla f(x(t)),
\end{equation}
with initial condition $x(0)=x_0$. This corresponds to the vector field $\mathcal G(x)=-\nabla f(x)$.

This section revisits gradient flow and its Euler discretizations: implicit Euler is the proximal point algorithm and explicit Euler is gradient descent. Their continuous and discrete convergence estimates follow from the same strong Lyapunov framework. Preliminary versions of some arguments appeared in the unpublished note~\cite{chen2021unified}.

\subsection{Gradient flow}
\label{sec:gradient}

We first give several functions with the strong Lyapunov property for gradient flow.



\begin{proposition}\label{pro:stLygd}
	Assume $f\in\mathcal S_{\mu,L}^1$ with $0<\mu\le L<\infty$, and let $x^\star$ be its unique minimizer. The following functions satisfy strong Lyapunov estimates:
	\begin{align}
		\label{eq:Lt-cond-gd1}		E(x) &= f(x) - f(x^\star), \ \qquad \qquad \qquad -\inprd{ \nabla E(x),\mathcal{G}(x)} \geq \mu \, E(x) + \frac{1}{2}\|\nabla f(x)\|_*^2,
		\\
		\label{eq:Lt-cond-gd2}		E(x) &= \frac{1}{2}\|x - x^\star\|^2, \qquad\quad -\inprd{ \nabla E(x),\mathcal{G}(x)} \geq \frac{2\mu L}{L + \mu} E(x) + \frac{1}{L + \mu} \|\nabla f(x)\|_*^2,  \\
		\label{eq:Lt-cond-gd}	E(x) &= f(x) - f(x^\star)+ \frac{\mu}{2}\|x - x^\star\|^2, \quad  	-\inprd{ \nabla E(x),\mathcal{G}(x)} \geq \mu \, E(x) + \|\nabla f(x)\|_*^2.
	\end{align}
Theorem~\ref{thm:strongLya} gives exponential decay of each energy along \eqref{eq:gf}.	
\end{proposition}

\begin{proof}
	For \eqref{eq:Lt-cond-gd1}, use $-\inprd{\nabla E(x),\mathcal G(x)}=\|\nabla f(x)\|_*^2$ and the bound $f(x)-f(x^\star)\le\|\nabla f(x)\|_*^2/(2\mu)$ from \eqref{eq:phiuppermu}. The estimate \eqref{eq:Lt-cond-gd2} follows from \eqref{eq:refineM} with $y=x^\star$. For \eqref{eq:Lt-cond-gd}, use \eqref{eq:Dmu} in the form $\langle\nabla f(x),x-x^\star\rangle\ge f(x)-f(x^\star)+\frac{\mu}{2}\|x-x^\star\|^2$.
%
\end{proof}
We retain the extra squared-gradient terms to control discretization errors caused by evaluating the gradient at an earlier iterate.


When $\mu=0$, the preceding estimates no longer give exponential decay. For $c\in\mathbb R$, define the sublevel set
$$
S_c(f):=\{x\in V:f(x)\le c\}.
$$
Let $f_{\min}:=\inf_{x\in V}f(x)$ and $\argmin f:=\{x\in V:f(x)=f_{\min}\}$.

\begin{lemma}\label{lem:R0}
	Let $f$ be continuous, convex, and coercive. For any $f_0\ge f_{\min}$, there exists $0<R_0<\infty$ such that
	\begin{equation}\label{eq:BL}
		\sup_{x^\star\in\argmin f}\ \sup_{x\in S_{f_0}(f)}
		\|x-x^\star\|
		\le
		R_0.
	\end{equation}
\end{lemma}

\begin{proof}
	By coercivity, $S_{f_0}(f)$ is bounded. By Theorem~\ref{thm:wellposedS3}, $\argmin f$ is nonempty. Since $\argmin f\subseteq S_{f_0}(f)$, the distance between any point in $S_{f_0}(f)$ and any minimizer is bounded. This gives \eqref{eq:BL}.
\end{proof}

\begin{proposition}\label{pro:gradconvex}
	Let $f$ be continuously differentiable, convex, and coercive. Fix $x^\star\in\argmin f$, and set
	$$
	E(x)=f(x)-f(x^\star).
	$$
	Set $f_0=f(x_0)$ and choose $R_0>0$ as in \eqref{eq:BL}. For $\mathcal G(x)=-\nabla f(x)$, we have
	\begin{equation}\label{eq:diff-L-gf}
		-\inprd{\nabla E(x),\mathcal G(x)}
		=
		\|\nabla f(x)\|_*^2
		\ge
		\frac{1}{R_0^2}E^2(x),
		\qquad x\in S_{f_0}(f),
	\end{equation}
	Hence the gradient flow satisfies
	$$
	f(x(t))-f(x^\star)
	\le
	\frac{R_0^2}{t+R_0^2/E(x_0)},
	\qquad t>0.
	$$
	The bound is understood as zero when $E(x_0)=0$.
\end{proposition}

\begin{proof}
	For $x\in S_{f_0}(f)$, convexity gives
	\begin{equation}
		f(x)-f(x^\star)
		\le
		\langle\nabla f(x),x-x^\star\rangle
		\le
		R_0\|\nabla f(x)\|_*.
	\end{equation}
	This implies \eqref{eq:diff-L-gf}. Since
	$$
	-\inprd{\nabla E(x),\mathcal G(x)}
	=
	\|\nabla f(x)\|_*^2
	\ge 0,
	$$
	the trajectory remains in $S_{f_0}(f)$. The stated bound follows from Theorem~\ref{thm:strongLya} with $q=2$ and $c=1/R_0^2$.
\end{proof}

\subsection{Proximal point algorithm}

For a step size $\alpha>0$, the implicit Euler discretization of \eqref{eq:gf} is
\begin{equation}
	x_{k+1}=x_k-\alpha\nabla f(x_{k+1}).
\end{equation}
It is the optimality condition for
\begin{equation}
	x_{k+1}
	=
	\operatorname{prox}_{\alpha f}(x_k)
	:=
	\argmin_{x\in V}
	\left\{
	f(x)+\frac{1}{2\alpha}\|x-x_k\|^2
	\right\}.
\end{equation}
This is the proximal point algorithm (PPA)~\cite{rockafellar1976monotone}.

For $\mu>0$, Theorem~\ref{th:implicitEuler} applied to the squared-distance energy in Proposition~\ref{pro:stLygd} gives linear convergence. For $\mu=0$, Proposition~\ref{pro:gradconvex} gives sublinear decay of the continuous flow. To obtain exponential energy decay in rescaled time, introduce
\begin{equation}\label{eq:s-gf}
	x'=-\nabla f(x)/\gamma,\qquad \gamma'=-\gamma,
\end{equation}
with $x(0)=x_0$ and $\gamma(0)=\gamma_0>0$. Then $\gamma(t)=\gamma_0e^{-t}$, so the flow is a time-rescaled gradient flow.

Define
\begin{equation}\label{eq:Lt-rs-gd}
	E(x,\gamma)
	:=
	f(x)-f(x^\star)+\frac{\gamma}{2}\|x-x^\star\|^2.
\end{equation}

\begin{lemma}\label{lem:strongLyap-rescaled}
	Let $f$ be continuously differentiable and convex. For the flow \eqref{eq:s-gf}, the energy \eqref{eq:Lt-rs-gd} satisfies
	\begin{equation}
		-\inprd{\nabla E(x,\gamma),\mathcal G(x,\gamma)}
		\ge
		E(x,\gamma)
		+
		\frac{1}{\gamma}\|\nabla f(x)\|_*^2,
	\end{equation}
	where $\mathcal G(x,\gamma)=\bigl(-\nabla f(x)/\gamma,-\gamma\bigr)^{\intercal}$ is the vector field in \eqref{eq:s-gf}. The estimate holds for $x\in V$ and $\gamma>0$.
\end{lemma}

\begin{proof}
	Direct calculation gives
	\begin{align*}
		-\inprd{\nabla E(x,\gamma),\mathcal G(x,\gamma)}
		&=
		\langle\nabla f(x),x-x^\star\rangle
		+
		\frac{\gamma}{2}\|x-x^\star\|^2
		+
		\frac{1}{\gamma}\|\nabla f(x)\|_*^2
		\\
		&\ge
		f(x)-f(x^\star)
		+
		\frac{\gamma}{2}\|x-x^\star\|^2
		+
		\frac{1}{\gamma}\|\nabla f(x)\|_*^2
		\\
		&=
		E(x,\gamma)
		+
		\frac{1}{\gamma}\|\nabla f(x)\|_*^2.
	\end{align*}
\end{proof}

Integrating the energy inequality as in Theorem~\ref{thm:strongLya} gives
\begin{equation}
	E(x(t),\gamma(t))
	\le
	e^{-t}E(x_0,\gamma_0),
	\qquad t\ge 0.
\end{equation}
Thus even the merely convex case has exponential decay after time rescaling.

For $\alpha_k>0$, consider the following semi-implicit Euler discretization of \eqref{eq:s-gf}:
\begin{equation}\label{implicitEuler-gf-s-non}
	\begin{aligned}
		\frac{x_{k+1}-x_k}{\alpha_k}
		&=
		-\frac{1}{\gamma_k}\nabla f(x_{k+1}),
\qquad		\frac{\gamma_{k+1}-\gamma_k}{\alpha_k}
		&=
		-\gamma_{k+1}.
	\end{aligned}
\end{equation}
%
Then the $x$-update is
\begin{equation}
	x_{k+1}
	=
	\operatorname{prox}_{t_k f}(x_k)
	:=
	\argmin_{x\in V}
	\left\{
	f(x)+\frac{1}{2t_k}\|x-x_k\|^2
	\right\}, \quad
	t_k:=\frac{\alpha_k}{\gamma_k}.
\end{equation}
We prove the energy estimate below in the smooth case; the nonsmooth case follows from the same subgradient inequality.

Define
\begin{equation}
	E_k
	:=
	E(x_k,\gamma_k)
	=
	f(x_k)-f(x^\star)+\frac{\gamma_k}{2}\|x_k-x^\star\|^2.
\end{equation}

\begin{theorem}\label{thm:im-sgf-non}
	Let $f$ be continuously differentiable and convex, and let $x^\star\in\argmin f$. Let $(x_k,\gamma_k)$ be generated by \eqref{implicitEuler-gf-s-non} with $\gamma_0>0$ and $\alpha_k>0$. Then
	\begin{equation}\label{eq:conv-Lk-sgf}
		E_{k+1}
		\le
		\frac{1}{1+\alpha_k}E_k.
	\end{equation}
	Moreover, with $t_i:= \alpha_i/\gamma_i$,
	\begin{equation}\label{eq:Lkdecay}
		\max\left\{
		f(x_k)-f(x^\star),
		t_k\|\nabla f(x_{k+1})\|_*^2
		\right\}
		\le
		E_k
		\le
		\frac{E_0}{1+\gamma_0\sum_{i=0}^{k-1}t_i}.
	\end{equation}
\end{theorem}

\begin{proof}
	Split
	$$
	E_{k+1}-E_k
	=
	\underbrace{
	E(x_{k+1},\gamma_k)-E(x_k,\gamma_k)
	}_{{\rm I}_1}
	+
	\underbrace{
	E(x_{k+1},\gamma_{k+1})-E(x_{k+1},\gamma_k)
	}_{{\rm I}_2}.
	$$
	For fixed $\gamma>0$, the function $E(\cdot,\gamma)$ is $\gamma$-strongly convex. Hence
	\begin{align*}
		{\rm I}_1
		&\le
		\langle\nabla_xE(x_{k+1},\gamma_k),x_{k+1}-x_k\rangle
		-
		\frac{\gamma_k}{2}\|x_{k+1}-x_k\|^2
		\\
		&=
		-\frac{\alpha_k}{\gamma_k}\|\nabla f(x_{k+1})\|_*^2
		-
		\alpha_k\langle x_{k+1}-x^\star,\nabla f(x_{k+1})\rangle
		-
		\frac{\alpha_k^2}{2\gamma_k}\|\nabla f(x_{k+1})\|_*^2
		\\
		&\le
		-\alpha_k\bigl(f(x_{k+1})-f(x^\star)\bigr)
		-
		t_k\|\nabla f(x_{k+1})\|_*^2.
	\end{align*}
	Since $E$ is affine in $\gamma$,
	$$
	{\rm I}_2
	=
	\frac{\gamma_{k+1}-\gamma_k}{2}\|x_{k+1}-x^\star\|^2
	=
	-\alpha_k\frac{\gamma_{k+1}}{2}\|x_{k+1}-x^\star\|^2.
	$$
	Therefore
	$$
	E_{k+1}-E_k
	\le
	-\alpha_kE_{k+1}
	-
	t_k\|\nabla f(x_{k+1})\|_*^2.
	$$
	This proves \eqref{eq:conv-Lk-sgf} and
	$$
	t_k\|\nabla f(x_{k+1})\|_*^2\le E_k.
	$$
	Also, $f(x_k)-f(x^\star)\le E_k$, so the left inequality in \eqref{eq:Lkdecay} follows.

	By recursion,
	$$
	E_k
	\le
	\frac{E_0}{\prod_{i=0}^{k-1}(1+\alpha_i)}
	=
	E_0\prod_{i=0}^{k-1}\frac{\gamma_{i+1}}{\gamma_i}
	=
	E_0\frac{\gamma_k}{\gamma_0}.
	$$
Summing the identity
	$$
	\frac{1}{\gamma_{k+1}}-\frac{1}{\gamma_k}
	=
	\frac{\gamma_k-\gamma_{k+1}}{\gamma_{k+1}\gamma_k}
	=
	\frac{\alpha_k}{\gamma_k}
	=
	t_k
	$$
 gives the stated upper bound in \eqref{eq:Lkdecay}.
\end{proof}

For merely convex objectives, the usual gradient methods have sublinear worst-case rates. The proximal scheme above gives linear energy decay if $\alpha_k$ is bounded away from zero. Then $t_k=\alpha_k/\gamma_k$ grows at least geometrically, which can make the proximal subproblems harder to solve. For a fixed proximal step, $\sum_{i=0}^{k-1}t_i$ grows linearly with $k$, and \eqref{eq:Lkdecay} gives the classical $\mathcal O(1/k)$ PPA rate~\cite{guler1991convergence}. 

\subsection{Gradient descent methods}

The explicit Euler discretization of \eqref{eq:gf}
\begin{equation}
	x_{k+1}=x_k-\alpha_k\nabla f(x_k).
\end{equation}
is the gradient descent method. As the gradient is evaluated at the current iterate $x_k$, stability requires a step-size restriction.

%


\begin{proposition}\label{pro:gd-function-decay}
	Assume $f\in\mathcal S_{\mu,L}^{1,1}$ with $0\le\mu\le L<\infty$ and $L>0$. For gradient descent with $\alpha=1/L$,
	$$
	x_{k+1}=x_k-\frac{1}{L}\nabla f(x_k),
	$$
	we have
	\begin{equation}\label{eq:fdecay}
		f(x_{k+1})-f(x_k)
		\le
		-\frac{1}{2L}
		\left(
		\|\nabla f(x_{k+1})\|_*^2
		+
		\|\nabla f(x_k)\|_*^2
		\right).
	\end{equation}
	If $\mu>0$, let $x^\star$ be the unique minimizer of $f$. Then
	\begin{equation}\label{eq:fdecayrate}
		f(x_{k+1})-f(x^\star)
		\le
		\frac{L-\mu}{L+\mu}
		\bigl(f(x_k)-f(x^\star)\bigr).
	\end{equation}
	If $\mu=0$ and $f$ is coercive, choose $x^\star\in\argmin f$ and $R_0>0$ as in \eqref{eq:BL} with $f_0=f(x_0)$. Then
	\begin{equation}\label{eq:fsublineardecayrate}
		f(x_k)-f(x^\star)
		\le
		\frac{LR_0^2}{k+LR_0^2/(f(x_0)-f(x^\star))}.
	\end{equation}
	The bound is understood as zero when $f(x_0)=f(x^\star)$.
\end{proposition}

\begin{proof}
	The Bregman identity gives
	\begin{align*}
		f(x_{k+1})-f(x_k)
		&=\langle\nabla f(x_{k+1}),x_{k+1}-x_k\rangle-D_f(x_k,x_{k+1}) \\
		&=-\frac{1}{L}\langle\nabla f(x_{k+1}),\nabla f(x_k)\rangle-D_f(x_k,x_{k+1}) \\
		&=-\frac{1}{2L}\left(\|\nabla f(x_{k+1})\|_*^2+\|\nabla f(x_k)\|_*^2\right)\\
		&\quad+\frac{1}{2L}\|\nabla f(x_{k+1})-\nabla f(x_k)\|_*^2-D_f(x_k,x_{k+1}).
	\end{align*}
	By the Bregman lower bound \eqref{eq:philowerL},
	$$D_f(x_k,x_{k+1})\ge \frac{1}{2L}\|\nabla f(x_k)-\nabla f(x_{k+1})\|_*^2.$$
	This proves \eqref{eq:fdecay}. For the convergence rates, let $x^\star$ be a minimizer and set $E(x)=f(x)-f(x^\star)$ and $E_k=E(x_k)$.

	For $\mu>0$, applying $\|\nabla f(x)\|_*^2\ge 2\mu E(x)$ to $x_k$ and $x_{k+1}$ gives
	$$E_{k+1}-E_k\le -\frac{\mu}{L}(E_{k+1}+E_k).$$
	Rearranging gives \eqref{eq:fdecayrate}.

	For $\mu=0$, \eqref{eq:fdecay} keeps all iterates in $S_{f(x_0)}(f)$. On this set, the strong Lyapunov property \eqref{eq:diff-L-gf} gives $\|\nabla f(x)\|_*^2\ge R_0^{-2}E^2(x)$, so
	$$E_{k+1}-E_k\le -\frac{1}{2LR_0^2}(E_{k+1}^2+E_k^2)\le -\frac{1}{LR_0^2}E_{k+1}E_k.$$
	If $E_k$ and $E_{k+1}$ are positive, this is equivalent to
	\begin{equation}\label{eq:diffEk}
		\frac{1}{E_{k+1}}-\frac{1}{E_k}\ge \frac{1}{LR_0^2}.
	\end{equation}
	Summing \eqref{eq:diffEk} gives \eqref{eq:fsublineardecayrate}. If an iterate attains the minimum, all later iterates remain there, so the bound also holds in that case.
\end{proof}


\section{Variable and operator splitting}
\label{sec:splitting}

This section introduces the variable and operator splitting (VOS) flow. We first prove strong Lyapunov estimates for the continuous flow and then give examples.

\subsection{Flow}

Consider the operator equation
\begin{equation}\label{eq:operator-equation}
	\mathcal L(x)=\nabla F(x)+N(x)=0,
\end{equation}
where $F\in\mathcal C^1$ is convex and $N:V\to V^*$ is $\mu$-strongly monotone:
\begin{equation}
	\langle N(x)-N(y),x-y\rangle\ge \mu\|x-y\|^2,
	\qquad \mu>0.
\end{equation}
Strong convexity can be transferred from $F$ to $N$. If $F$ is $\mu$-strongly convex and $N$ is monotone, replace
$$
F(x)\quad\text{by}\quad F(x)-\frac{\mu}{2}\|x\|^2,
\qquad
N(x)\quad\text{by}\quad N(x)+\mu x.
$$
The sum $\nabla F+N$ is unchanged, while the shifted operator $N$ is $\mu$-strongly monotone. The merely convex case $\mu=0$ is treated in Section~\ref{sec: convex homotopy}. We assume that a solution $x^\star$ exists; strong monotonicity makes it unique. For set-valued $N$, the operator equation and flow are interpreted as inclusions. We use maximal-monotonicity assumptions on $N$, or the relevant operator sum, to ensure that implicit updates are well defined. We use single-valued notation to keep the calculation concise.

The VOS flow is
\begin{equation}\label{eq:splitflow}
	\left\{
	\begin{aligned}
		x'&=y-x,\\
		y'&=-\frac{1}{\mu}\bigl(\nabla F(x)+N(y)\bigr).
	\end{aligned}
	\right.
\end{equation}
Its equilibrium is $(x^\star,x^\star)$, where $x^\star$ solves \eqref{eq:operator-equation}. Let $\mathcal G$ be the vector field of \eqref{eq:splitflow}.
Since $\nabla F(x^\star)+N(x^\star)=0$, the vector field can also be written as
$$
\mathcal G(x,y)
=
\begin{pmatrix}
	y-x^\star-(x-x^\star)\\
	-\mu^{-1}\bigl(\nabla F(x)-\nabla F(x^\star)\bigr)
	-\mu^{-1}\bigl(N(y)-N(x^\star)\bigr)
\end{pmatrix}.
$$
This form is used only in the analysis; the algorithm does not need $x^\star$. Writing $y-x=(y-x^\star)-(x-x^\star)$ makes the cancellation in the Lyapunov estimate explicit.

\subsection{Strong Lyapunov functions}
We first use the energy
\begin{equation}\label{eq:DFmuy}
	E(x,y)
	=
	D_F(x,x^\star)+\frac{\mu}{2}\|y-x^\star\|^2.
\end{equation}
Since $F$ need not be strongly convex, $E(x_k,y_k)\to0$ implies $y_k\to x^\star$ strongly, but the energy alone does not control $x_k-x^\star$. For continuous VOS trajectories, $x'=y-x$ also transfers convergence from $y$ to $x$. The stronger energy below controls both variables directly. The use of an auxiliary variable is also standard in NAG~\cite{Nesterov:2013Introductory}.


A matrix calculation illustrates the cancellation. In finite dimensions, suppose that $\nabla D_F(x,x^\star)=A(x-x^\star)$, where $A$ is symmetric positive semidefinite, and that $N$ is linear with $\sym N\ge\mu I$. With $\bs e=(x-x^\star,y-x^\star)^{\intercal}$,
$$
-\inprd{\nabla E,\mathcal G}
=
\underbrace{
\bs e^{\intercal}
	\begin{pmatrix}
		-A & 0 \\
		0 & -\mu I
\end{pmatrix}}_{- (\nabla E)^{\intercal}}
\underbrace{
	\begin{pmatrix}
		-I & I  \\
		-\mu^{-1}A & -\mu^{-1}N
	\end{pmatrix}
\bs e}_{\mathcal{G}}
=
\bs e^{\intercal}
\begin{pmatrix}
	A&-A\\
	A&N
\end{pmatrix}
\bs e.
$$
Only the symmetric part contributes to a quadratic form: $\bs e^{\intercal}M\bs e=\bs e^{\intercal}(\sym M)\bs e$, where $\sym M=\frac12(M+M^{\intercal})$. Thus
$$
\sym
\begin{pmatrix}
	A&-A\\
	A&N
\end{pmatrix}
=
\begin{pmatrix}
	A&0\\
	0&\sym N
\end{pmatrix}
\ge
\begin{pmatrix}
	A&0\\
	0&\mu I
\end{pmatrix}.
$$
Thus
$$
-\inprd{\nabla E,\mathcal G}\ge 2 E.
$$
This linear case motivates the nonlinear proof below.

\begin{theorem}\label{th:SLy1}
	Let $E$ be defined by \eqref{eq:DFmuy}, and let $\mathcal G$ be the vector field of \eqref{eq:splitflow}. Then we have the strong Lyapunov property
	\begin{equation}\label{eq:SLy1}
		-\inprd{\nabla E(x,y),\mathcal G(x,y)}
		\ge
		 E(x,y)+D_F(x^\star,x)+\frac{\mu}{2}\|y-x^\star\|^2.
	\end{equation}
	Consequently, along any solution defined for $t\ge0$,
	$$
	 E(x(t),y(t))\le  E(x(0),y(0))e^{-t}.
	$$
\end{theorem}

\begin{proof}
	Using $\nabla F(x^\star)+N(x^\star)=0$, we get
	\begin{align*}
		-\inprd{\nabla E(x,y),\mathcal G(x,y)}
		&=
		\langle\nabla F(x)-\nabla F(x^\star),x-y\rangle
		+
		\langle y-x^\star,\nabla F(x)+N(y)\rangle
		\\
		&=
		\langle\nabla F(x)-\nabla F(x^\star),x-x^\star\rangle
		+
		\langle y-x^\star,N(y)-N(x^\star)\rangle.
	\end{align*}
	By the Bregman identity,
	$$
	\langle\nabla F(x)-\nabla F(x^\star),x-x^\star\rangle
	=
	D_F(x,x^\star)+D_F(x^\star,x).
	$$
	By the strong monotonicity of $N$,
	$$
	\langle y-x^\star,N(y)-N(x^\star)\rangle
	\ge
	\mu\|y-x^\star\|^2.
	$$
	Combining these two estimates gives \eqref{eq:SLy1}. The decay follows from Theorem~\ref{thm:strongLya}.
\end{proof}

To control both variables directly, set
$$
f=F_{+\mu}=F+\frac{\mu}{2}\|\cdot\|^2,
\qquad
K:=N-\mu I.
$$
Then $f$ is $\mu$-strongly convex and $K$ is monotone. The flow \eqref{eq:splitflow} becomes
\begin{equation}\label{eq:splitflowfK}
	\left\{
	\begin{aligned}
		x'&=y-x,\\
		y'&=x-y-\frac{1}{\mu}\bigl(\nabla f(x)+K(y)\bigr).
	\end{aligned}
	\right.
\end{equation}
For this form, consider
\begin{equation}\label{eq:DfmuyK}
	E_f(x,y)
	=
	D_f(x,x^\star)+\frac{\mu}{2}\|y-x^\star\|^2.
\end{equation}
Since $f$ is $\mu$-strongly convex, $E_f(x,y)\ge\frac{\mu}{2}(\|x-x^\star\|^2+\|y-x^\star\|^2)$. Thus decay of $E_f$ controls convergence of both variables.

\begin{theorem}\label{th:SLy2}
	Let $E_f$ be defined by \eqref{eq:DfmuyK}, and let $\mathcal G$ be the vector field of \eqref{eq:splitflowfK}. Then we have the strong Lyapunov property
	\begin{equation}\label{eq:SLyfK}
		-\langle\nabla E_f(x,y),\mathcal G(x,y)\rangle
		\ge
		E_f(x,y)+\frac{\mu}{2}\|y-x\|^2 .
	\end{equation}
	Consequently, along any solution defined for $t\ge0$,
	$$
	E_f(x(t),y(t))\le E_f(x(0),y(0))e^{-t}.
	$$
\end{theorem}

\begin{proof}
	Using $\nabla f(x^\star)+K(x^\star)=0$, direct calculation gives
	\begin{align*}
		-\langle\nabla E_f(x,y),\mathcal G(x,y)\rangle
		&=
		\langle\nabla f(x)-\nabla f(x^\star),x-y\rangle
		-\mu\langle y-x^\star,x-y\rangle  \\
		&\quad
		+\langle y-x^\star,\nabla f(x) -\nabla f(x^\star) + K(y) - K(x^\star)\rangle  \\
		&=
		\langle\nabla f(x)-\nabla f(x^\star),x-x^\star\rangle
		+\langle y-x^\star,K(y)-K(x^\star)\rangle  \\
		&\quad
		-\mu\langle y-x^\star,x-y\rangle .
	\end{align*}
	By the strong convexity of $f$,
	$$
	\langle\nabla f(x)-\nabla f(x^\star),x-x^\star\rangle
	\ge
	D_f(x,x^\star)+\frac{\mu}{2}\|x-x^\star\|^2 .
	$$
	By monotonicity of $K$, $\langle y-x^\star,K(y)-K(x^\star)\rangle\ge0$. Using \eqref{eq:squares},
	$$
	-\mu\langle y-x^\star,x-y\rangle
	=
	\frac{\mu}{2}\|y-x^\star\|^2
	+
	\frac{\mu}{2}\|y-x\|^2
	-
	\frac{\mu}{2}\|x-x^\star\|^2 .
	$$
	Therefore,
	\begin{align*}
		-\langle\nabla E_f(x,y),\mathcal G(x,y)\rangle
		&\ge
		D_f(x,x^\star)
		+\frac{\mu}{2}\|y-x^\star\|^2
		+\frac{\mu}{2}\|y-x\|^2  \\
		&=
		E_f(x,y)+\frac{\mu}{2}\|y-x\|^2 .
	\end{align*}
	This proves \eqref{eq:SLyfK}. The decay estimate follows from Theorem~\ref{thm:strongLya}.
\end{proof}

\subsection{Examples}
\label{sec:example}

The examples below cover strongly convex and composite optimization, saddle-point systems, and mirror flows.

\subsubsection{Strongly convex optimization}
\label{sec:VOSstrongconvex}

For $\mu>0$, consider
\begin{equation}\label{eq:minf}
	\min_{x\in V} f(x),
	\qquad
	f\in\mathcal S_\mu^1.
\end{equation}
Choose
$$
\nabla F(x)=\nabla f(x)-\mu x,
\qquad
N(y)=\mu y.
$$
Then \eqref{eq:splitflowfK} becomes
$$
\left\{
\begin{aligned}
	x'&=y-x,\\
	y'&=x-y-\frac{1}{\mu}\nabla f(x).
\end{aligned}
\right.
$$
Eliminating $y=x'+x$ gives the Heavy-Ball flow~\cite{polyak1964some}
\begin{equation}\label{eq: HB flow}
	x''+2x'+\frac{1}{\mu}\nabla f(x)=0.
\end{equation}
Equivalently, with $v=x'$, the velocity form is
\begin{equation}\label{eq:velocity}
	x'=v,
	\qquad
	v'=-2v-\frac{1}{\mu}\nabla f(x),
\end{equation}
where $v$ represents the velocity. We use the VOS form because it exposes the strong Lyapunov structure in Theorems~\ref{th:SLy1} and~\ref{th:SLy2}. This structure is less visible in the velocity form \eqref{eq:velocity} or the second-order equation \eqref{eq: HB flow}. For comparison, adding the potential and kinetic energies gives the mechanical energy
$$
E_{\rm m}(x,v)
:=
f(x)-f(x^\star)
+
\frac{\mu}{2}\|v\|^2
$$
which satisfies
$$
\frac{\mathrm d}{\mathrm dt}E_{\rm m}(x(t),v(t))
=
-2\mu \|v(t)\|^2\leq 0.
$$
This identity gives dissipation, but the right-hand side controls only the velocity and does not directly yield contraction of the full energy.

\subsubsection{Composite convex optimization}
\label{sec:VOScomposite}

Consider
\begin{equation}\label{eq:min-f-g}
	\min_{x\in V} f(x)+g(x),
\end{equation}
where $f$ is smooth and $\mu$-strongly convex with $\mu>0$, and $g$ is proper, closed, and convex. When $g$ is smooth, choose
$$
\nabla F(x)=\nabla f(x)-\mu x,
\qquad
N(y)=\mu y+\nabla g(y).
$$
Then \eqref{eq:splitflow} or equivalently \eqref{eq:splitflowfK} gives
$$
\left\{
\begin{aligned}
	x'&= y - x,\\
	y'&= x - y -\frac{1}{\mu}\bigl(\nabla f(x)+\nabla g(y)\bigr).
\end{aligned}
\right.
$$
The same splitting covers nonsmooth $g$ by replacing $\nabla g(y)$ with $\partial g(y)$ and interpreting the $y$-equation as a differential inclusion. The Lyapunov energy $E_f$ in \eqref{eq:DfmuyK} contains no Bregman term for $g$. Strong convexity of $f$ ensures that $E_f\to0$ implies strong convergence of both $x$ and $y$ to $x^\star$.

\subsubsection{Saddle-point systems with bilinear coupling}
\label{sec:VOSsaddle}

Consider
\begin{equation}\label{eq:saddleflow}
	\min_{u\in V}\max_{p\in Q}
	\mathcal H(u,p)
	=
	f^u(u)-f^p(p)+\langle Bu,p\rangle,
\end{equation}
Here $V$ and $Q$ are Hilbert spaces, $f^u$ and $f^p$ are continuously differentiable and strongly convex with constants $\mu_u>0$ and $\mu_p>0$, and $B:V\to Q^*$ is bounded and linear. We use $B^{\intercal}$ for its adjoint under the Riesz identification.

Set
$$
x=(u,p)^{\intercal},
\qquad
y=(v,q)^{\intercal},
\qquad
f(x):=f^u(u)+f^p(p),
$$
and define
$$
M_\mu:=
\begin{pmatrix}
\mu_u I&0\\
0&\mu_p I
\end{pmatrix},
\qquad
K:=
\begin{pmatrix}
0&B^{\intercal}\\
-B&0
\end{pmatrix}.
$$
Then
$$
\nabla F(x)=\nabla f(x)-M_\mu x,
\qquad
N(y)=M_\mu y+Ky,
$$
and the optimality operator satisfies
$$
\begin{pmatrix}
\nabla f^u(u)+B^{\intercal}p\\
\nabla f^p(p)-Bu
\end{pmatrix}
=
\nabla f(x)+Kx
=
\nabla F(x)+N(x).
$$
Since $K^{\intercal}=-K$, we have $\sym K=0$ and $\sym N=M_\mu$. In the metric induced by $M_\mu$, $N$ is $1$-strongly monotone. The corresponding VOS flow is
\begin{equation}
\left\{
\begin{aligned}
x'&=y-x,\\
y'&=x-y-M_\mu^{-1}\bigl(\nabla f(x)+Ky\bigr).
\end{aligned}
\right.
\end{equation}
In the original variables, this becomes
$$
\left\{
\begin{aligned}
u'&=v-u,
&\qquad
v'&=u-v-\frac{1}{\mu_u}\bigl(\nabla f^u(u)+B^{\intercal}q\bigr),\\
p'&=q-p,
&\qquad
q'&=p-q-\frac{1}{\mu_p}\bigl(\nabla f^p(p)-Bv\bigr).
\end{aligned}
\right.
$$

\subsubsection{Saddle-point formulation of composite convex optimization}

We next consider the composite convex problem
\begin{equation}
	\min_{x\in V} f(x)+g(Bx),
\end{equation}
where $B:V\to Q$ is bounded and linear, and $f:V\to\mathbb R$ and $g:Q\to\mathbb R$ are strongly convex with Lipschitz continuous gradients.

Introducing $p=Bx$ and a multiplier $\lambda\in Q^*$ gives
$$
\min_{x\in V,\ p\in Q}\sup_{\lambda\in Q^*}
\left\{
f(x)+g(p)+\langle\lambda,Bx-p\rangle
\right\}.
$$
Eliminating $p$ gives the saddle-point problem
$$
\min_{x\in V}\sup_{\lambda\in Q^*}
\left\{
f(x)-g^*(\lambda)+\langle Bx,\lambda\rangle
\right\}.
$$
By the conjugacy relations in Section~\ref{sec:bd-convex}, $g^*$ is smooth and strongly convex. Thus this problem has the form \eqref{eq:saddleflow}, and the VOS methods for bilinear coupling apply.

\subsubsection{Convex optimization beyond the Euclidean metric}
Mirror descent uses the geometry induced by a mirror function~\cite{NemirovskyYudin1983,AIKrichene2015}. Let $\phi$ be a Legendre function, and let $f$ be continuously differentiable on $\operatorname{int}\dom\phi$. Assume that a minimizer $x^\star$ lies in this interior, that $\nabla\phi$ and $\nabla\phi^*$ are inverse maps between the interiors of their domains, and that $x(t)$ and $y(t)$ remain in $\operatorname{int}\dom\phi$. For $\mu>0$, we say that $f$ is relatively $\mu$-strongly convex with respect to $\phi$ if
$$
f_{-\mu}:=f-\mu\phi
$$
is convex. A mirror version of the VOS flow, proposed in~\cite{ChenLuoWeiXu2025Accelerated}, is
\begin{equation}\label{eq:VOSflow}
	\left\{
	\begin{aligned}
		x'&=y-x,\\
		(\nabla\phi(y))'&=-\mu^{-1}\nabla f_{-\mu}(x)-\nabla\phi(y).
	\end{aligned}
	\right.
\end{equation}
At equilibrium, $y^\star=x^\star$ and $\nabla f(x^\star)=\nabla f_{-\mu}(x^\star)+\mu\nabla\phi(x^\star)=0$.
Let
$$
\eta=\nabla\phi(y),
\qquad
\eta^\star=\nabla\phi(x^\star).
$$
The Lyapunov function is
\begin{equation}\label{eq:lya_func}
	 E(x,\eta)
	:=
	D_{\fb}(x,x^\star)+\mu D_{\phi^*}(\eta,\eta^\star).
\end{equation}

\begin{theorem}\label{th:mirror-lyap}
	Let $E$ be defined by \eqref{eq:lya_func}, and let $\mathcal G(x,\eta)$ be the vector field of \eqref{eq:VOSflow} in the variables $(x,\eta)$.
	Then
	\begin{equation}\label{eq:stLy}
		-\nabla E(x,\eta)\cdot\mathcal G(x,\eta)
		=
		 E(x,\eta)+D_{\fb}(x^\star,x)+\mu D_\phi(y,x^\star).
	\end{equation}
	Consequently, along any solution defined for $t\ge0$,
	\begin{equation}
		 E(x(t),\eta(t))
		\le
		e^{-t} E(x(0),\eta(0)).
	\end{equation}
\end{theorem}

\begin{proof}
	We have
	$$
	\partial_x E
	=
	\nabla\fb(x)-\nabla\fb(x^\star),
	\qquad
	\partial_\eta E
	=
	\mu\bigl(\nabla\phi^*(\eta)-\nabla\phi^*(\eta^\star)\bigr)
	=
	\mu(y-x^\star).
	$$
	Since $\nabla\fb(x^\star)=-\mu\eta^\star$, the Bregman identities give
	\begin{align*}
		-\nabla E(x,\eta)\cdot\mathcal G(x,\eta)
		&=
		\langle\nabla\fb(x)-\nabla\fb(x^\star),x-y\rangle
		\\
		&\quad
		+
		\mu\left\langle y-x^\star,\mu^{-1}\bigl(\nabla\fb(x)-\nabla\fb(x^\star)\bigr)+\eta-\eta^\star\right\rangle
		\\
		&=
		\langle\nabla\fb(x)-\nabla\fb(x^\star),x-x^\star\rangle
		+
		\mu\langle y-x^\star,\eta-\eta^\star\rangle
		\\
		&=
		D_{\fb}(x,x^\star)+D_{\fb}(x^\star,x)
		+
		\mu D_{\phi^*}(\eta,\eta^\star)
		+
		\mu D_{\phi^*}(\eta^\star,\eta)
		\\
		&=
		 E(x,\eta)+D_{\fb}(x^\star,x)+\mu D_\phi(y,x^\star).
	\end{align*}
	The last equality uses the conjugacy identity $D_{\phi^*}(\eta^\star,\eta)=D_\phi(y,x^\star)$. This proves \eqref{eq:stLy}. The decay estimate follows from Theorem~\ref{thm:strongLya}.
\end{proof}



\section{Accelerated gradient methods}
\label{sec:AGD}

We derive accelerated gradient methods by discretizing the VOS flow \eqref{eq:splitflow}. The implicit Euler method inherits the strong Lyapunov decay directly, but is usually too costly. We therefore use it as a reference scheme and design implementable discretizations whose errors can be controlled.

\subsection{Accelerated Over-Relaxation}
\label{sec:aor-vos}

An implicit--explicit discretization of flow \eqref{eq:splitflow} combined with accelerated over-relaxation (AOR)~\cite{hadjidimos1978accelerated} gives the following {\it forward AOR-VOS} scheme:
\begin{subnumcases}{\label{eq:AOR-VOS-forward}}
\frac{x_{k+1}-x_k}{\alpha} = y_k-x_{k+1},\\
\frac{y_{k+1}-y_k}{\alpha} = -\frac{1}{\mu}\bigl(2\nabla F(x_{k+1})-\nabla F(x_k)+N(y_{k+1})\bigr).
\end{subnumcases}
Here $N$ is treated implicitly. The method first computes $x_{k+1}$ from $y_k$ and then updates $y_{k+1}$. The term $2\nabla F(x_{k+1})-\nabla F(x_k)$ is the AOR extrapolation of $\nabla F$.

The corresponding {\it backward AOR-VOS} scheme is
\begin{subnumcases}{\label{eq:AOR-VOS-backward}}
\label{eq:AGDx}
\frac{x_{k+1}-x_k}{\alpha} = 2y_{k+1}-y_k-x_{k+1},\\
\label{eq:AGDy}
\frac{y_{k+1}-y_k}{\alpha} = -\frac{1}{\mu}\bigl(\nabla F(x_k)+N(y_{k+1})\bigr).
\end{subnumcases}
Here $y_{k+1}$ is updated first, and the $x$-equation uses the AOR extrapolation $2y_{k+1}-y_k$. 

The extrapolation $2y_{k+1}-y_k$ is the classical AOR step for linear iterations~\cite{hadjidimos1978accelerated}. It appears in the primal--dual method of Chambolle and Pock~\cite{chambolle2011first} with $\theta=1$. He and Yuan~\cite{he2012convergence} interpreted it through a preconditioned proximal point method and proved convergence using a contractive quadratic Lyapunov functional. 
The contribution here is the Lyapunov analysis rather than the extrapolation itself.

For the VOS flow, consider
\begin{equation}\label{eq: AGD Lyapunov}
	E(x,y)
	=
	D_F(x,x^\star)+\frac{\mu}{2}\|y-x^\star\|^2.
\end{equation}
Its strong Lyapunov property was established in Theorem~\ref{th:SLy1}. For the backward scheme, define the cross term and the shifted energy
\begin{equation}\label{eq:mod-lyap1}
	C(x,y)
	:=
	\inprd{\nabla F(x)-\nabla F(x^\star),y-x^\star},
	\qquad
	\mathcal E^\alpha:=E-\alpha C.
\end{equation}
For the forward scheme, the modified energy is $\mathcal E^\alpha=E+\alpha C$. We analyze the backward scheme; the forward case follows by the same argument with the sign reversed.

Let $\bs z=(x,y)^{\intercal}$ and $\bs z^\star=(x^\star,x^\star)^{\intercal}$. The linear case explains the correction. Suppose that $\nabla D_F(x,x^\star)=A(x-x^\star)$ and define
$$
\mathcal D
=
\begin{pmatrix}
	A&0\\
	0&\mu I
\end{pmatrix},
\qquad
\mathcal A^{\sym}
=
\begin{pmatrix}
	0&A\\
	A&0
\end{pmatrix}.
$$
Then
$$
E(\bs z)
=
\frac12\|\bs z-\bs z^\star\|_{\mathcal D}^2,
\qquad
\mathcal E^\alpha(\bs z)
=
\frac12\|\bs z-\bs z^\star\|_{\mathcal D-\alpha\mathcal A^{\sym}}^2.
$$
The AOR correction symmetrizes the discrepancy from implicit Euler:
$$
\inprd{\nabla E(\bs z_{k+1}),\bs z_{k+1}-\bs z_k}
=
\alpha\inprd{\nabla E(\bs z_{k+1}),\mathcal G(\bs z_{k+1})}
+
\alpha\inprd{\bs z_{k+1}-\bs z^\star,\bs z_{k+1}-\bs z_k}_{\mathcal A^{\sym}}.
$$
The last term expands by the identity of squares:
\begin{equation*}
	\begin{aligned}
		&\alpha\inprd{\bs z_{k+1}-\bs z^\star,\bs z_{k+1}-\bs z_k}_{\mathcal A^{\sym}}\\
		={}&
		\frac{\alpha}{2}
		\left(
		\|\bs z_{k+1}-\bs z^\star\|_{\mathcal A^{\sym}}^2
		+
		\|\bs z_{k+1}-\bs z_k\|_{\mathcal A^{\sym}}^2
		-
		\|\bs z_k-\bs z^\star\|_{\mathcal A^{\sym}}^2
		\right).
	\end{aligned}
\end{equation*}
Hence
\begin{equation}
	\mathcal E^\alpha(\bs z_{k+1})-\mathcal E^\alpha(\bs z_k)
	=
	\alpha\inprd{\nabla E(\bs z_{k+1}),\mathcal G(\bs z_{k+1})}
	-
	\frac12\|\bs z_{k+1}-\bs z_k\|_{\mathcal D-\alpha\mathcal A^{\sym}}^2.
\end{equation}
The following bound controls the correction in the nonlinear case.

\begin{lemma}\label{lem:cross Breg}
	Assume that $F$ is convex and $L_F$-smooth with $L_F>0$. Let $\mu>0$. Then, for any $(x,y)$ and $(\hat x,\hat y)$,
	\begin{equation}\label{eq:cross-breg}
		\left|
		\inprd{y-\hat y,\nabla F(x)-\nabla F(\hat x)}
		\right|
		\le
		\sqrt{\frac{L_F}{\mu}}
		\left(
		D_F(x,\hat x)+\frac{\mu}{2}\|y-\hat y\|^2
		\right).
	\end{equation}
\end{lemma}

\begin{proof}
	By the Cauchy--Schwarz and Young inequalities,
	$$
	\left|
	\inprd{y-\hat y,\nabla F(x)-\nabla F(\hat x)}
	\right|
	\le
	\frac{1}{2\sqrt{\mu L_F}}
	\|\nabla F(x)-\nabla F(\hat x)\|_*^2
	+
	\frac{\sqrt{\mu L_F}}{2}\|y-\hat y\|^2.
	$$
	By the co-coercivity bound \eqref{eq:philowerL},
	$$
	\frac{1}{2\sqrt{\mu L_F}}
	\|\nabla F(x)-\nabla F(\hat x)\|_*^2
	\le
	\sqrt{\frac{L_F}{\mu}}D_F(x,\hat x).
	$$
	Combining these estimates gives \eqref{eq:cross-breg}.
\end{proof}

\begin{theorem}[Convergence of the AOR-VOS method]
	\label{thm:convergence rate of AOR-VOS}
	Suppose that $F$ is convex and $L_F$-smooth, and that $N$ is $\mu$-strongly monotone, with $L_F>0$ and $\mu>0$. Let $(x_k,y_k)$ be generated by the backward AOR-VOS scheme~\eqref{eq:AOR-VOS-backward} with $0<\alpha\le\sqrt{\mu/L_F}$. Then, for all $k\ge0$,
	\begin{equation}\label{eq:AOR-rates}
		\max\left\{
		\mathcal E^\alpha(x_{k+1},y_{k+1}),
		\frac{\alpha}{1+\alpha}E(x_{k+1},y_{k+1})
		\right\}
		\le
		\frac{1}{(1+\alpha)^{k+1}}
		\mathcal E^\alpha(x_0,y_0).
	\end{equation}
	The same estimate holds for the forward AOR-VOS scheme~\eqref{eq:AOR-VOS-forward}, with the corresponding sign change in $\mathcal E^\alpha$.
\end{theorem}
\begin{proof}
	The backward scheme is a correction of implicit Euler:
	$$
	\bs z_{k+1}-\bs z_k
	=
	\alpha\mathcal G(\bs z_{k+1})
	+
	\alpha
	\begin{pmatrix}
	y_{k+1}-y_k\\
	\mu^{-1}\bigl(\nabla F(x_{k+1})-\nabla F(x_k)\bigr)
	\end{pmatrix}.
	$$
	By the Bregman expansion of $E$,
	\begin{align}
		E(\bs z_{k+1})-E(\bs z_k)
		&=
		\alpha\inprd{\nabla E(\bs z_{k+1}),\mathcal G(\bs z_{k+1})}
		-
		D_E(\bs z_k,\bs z_{k+1}) \notag\\
		&\quad+
		\alpha\inprd{
		\nabla E(\bs z_{k+1}),
		\begin{pmatrix}
		y_{k+1}-y_k\\
		\mu^{-1}\bigl(\nabla F(x_{k+1})-\nabla F(x_k)\bigr)
		\end{pmatrix}}.
	\end{align}
	The last cross term equals
	\begin{align*}
		&\alpha\inprd{\nabla F(x_{k+1})-\nabla F(x^\star),y_{k+1}-y_k}
		+
		\alpha\inprd{y_{k+1}-x^\star,\nabla F(x_{k+1})-\nabla F(x_k)}\\
		={}&
		\alpha C(x_{k+1},y_{k+1})
		-
		\alpha C(x_k,y_k)
		+
		\alpha\inprd{y_{k+1}-y_k,\nabla F(x_{k+1})-\nabla F(x_k)},
	\end{align*}
	where $C(x,y)$ is defined in \eqref{eq:mod-lyap1} and used to shift the energy to $\mathcal E^{\alpha}$. 
	Therefore
	\begin{equation}
		\begin{aligned}
			\mathcal E^\alpha(\bs z_{k+1})-\mathcal E^\alpha(\bs z_k)
			={}&
			\alpha\inprd{\nabla E(\bs z_{k+1}),\mathcal G(\bs z_{k+1})}\\
			&-
			D_E(\bs z_k,\bs z_{k+1})
			+
			\alpha\inprd{y_{k+1}-y_k,\nabla F(x_{k+1})-\nabla F(x_k)}.
		\end{aligned}
	\end{equation}
	Since
	$$
	D_E(\bs z_k,\bs z_{k+1})
	=
	D_F(x_k,x_{k+1})
	+
	\frac{\mu}{2}\|y_k-y_{k+1}\|^2,
	$$
	Lemma~\ref{lem:cross Breg} and $\alpha\le\sqrt{\mu/L_F}$ give
	$$
	-D_E(\bs z_k,\bs z_{k+1})
	+
	\alpha\inprd{y_{k+1}-y_k,\nabla F(x_{k+1})-\nabla F(x_k)}
	\le0.
	$$
	The strong Lyapunov property in Theorem~\ref{th:SLy1} therefore yields
	\begin{equation}\label{eq: decay of E quadratic appendix}
		\mathcal E^\alpha(\bs z_{k+1})-\mathcal E^\alpha(\bs z_k)
		\le
		-\alpha
		\left(
		E(\bs z_{k+1})
		+
		D_F(x^\star,x_{k+1})
		+
		\frac{\mu}{2}\|y_{k+1}-x^\star\|^2
		\right).
	\end{equation}
	Applying Lemma~\ref{lem:cross Breg} with $x=x^\star$, $\hat x=x_{k+1}$, $y=y_{k+1}$, and $\hat y=x^\star$ gives
	$$
	\alpha
	\left|
	\inprd{y_{k+1}-x^\star,\nabla F(x^\star)-\nabla F(x_{k+1})}
	\right|
	\le
	D_F(x^\star,x_{k+1})
	+
	\frac{\mu}{2}\|y_{k+1}-x^\star\|^2.
	$$
	Hence we can replace $E$ by $\mathcal E^{\alpha}$ on the right-hand side of \eqref{eq: decay of E quadratic appendix}:
	$$
	\mathcal E^\alpha(\bs z_{k+1})
	\le
	E(\bs z_{k+1})
	+
	D_F(x^\star,x_{k+1})
	+
	\frac{\mu}{2}\|y_{k+1}-x^\star\|^2.
	$$
	Combining this estimate with \eqref{eq: decay of E quadratic appendix} gives
	$$
	\mathcal E^\alpha(\bs z_{k+1})-\mathcal E^\alpha(\bs z_k)
	\le
	-\alpha\mathcal E^\alpha(\bs z_{k+1}),
	$$
	and therefore
	$$
	\mathcal E^\alpha(\bs z_{k+1})
	\le
	\rho\, \mathcal E^\alpha(\bs z_k),
	\qquad
	\rho:=\frac{1}{1+\alpha}.
	$$
	Iteration yields
	$$
	\mathcal E^\alpha(\bs z_{k+1})
	\le
	\rho^{k+1}\mathcal E^\alpha(\bs z_0).
	$$
	Finally, \eqref{eq: decay of E quadratic appendix} gives
	$$
	E(\bs z_{k+1})
	\le
	\frac{1}{\alpha}
	\left(
	\mathcal E^\alpha(\bs z_k)-\mathcal E^\alpha(\bs z_{k+1})
	\right)
	\le
	\alpha^{-1}\rho^k\mathcal E^\alpha(\bs z_0).
	$$
	This proves \eqref{eq:AOR-rates}.
\end{proof}

\begin{remark}[Bregman energy and gradient norms]\label{rem:bregman-gradient}\rm
Our convergence estimates concern Bregman energies rather than gradient norms directly. For convex, $L_F$-smooth $F$, the co-coercivity bound \eqref{eq:philowerL} gives
$$
\|\nabla F(x)-\nabla F(x^\star)\|_*^2\le 2L_F D_F(x,x^\star).
$$
In particular, if $f$ is convex and $L$-smooth and $x^\star$ minimizes $f$, then
$$
\|\nabla f(x)\|_*^2\le 2L D_f(x,x^\star).
$$
Thus a bound on $D_f(x_k,x^\star)$ gives the same rate for the squared gradient norm. 
\end{remark}

\subsection{Extrapolation by Predictor--Corrector methods}
\label{sec: EPC}
We start from the scheme in which the predictor $\tilde x_{k+1}$ is computed first and then used to update $y_{k+1}$:
\begin{equation}\label{eq:EPC}
	\left\{
	\begin{aligned}
		\frac{\tilde x_{k+1}-x_k}{\alpha} &= y_k-\tilde x_{k+1},\\
		\frac{y_{k+1}-y_k}{\alpha} &= -\frac{1}{\mu}\left(\nabla F(\tilde x_{k+1})+N(y_{k+1})\right).
	\end{aligned}
	\right.
\end{equation}

\begin{lemma}\label{lem:NAG-GS}
	For the scheme~\eqref{eq:EPC} and the Lyapunov function \eqref{eq: AGD Lyapunov},
	\begin{equation}\label{eq:EPC-prelim}
	\begin{aligned}
		(1+\alpha) E(\tilde x_{k+1},y_{k+1})
		&\le
		E(\bs z_k)
		-\frac{\mu}{2}\|y_{k+1}-y_k\|^2\\
&\quad		-\alpha\langle\nabla F(\tilde x_{k+1})-\nabla F(x^\star),y_{k+1}-y_k\rangle.
\end{aligned}
	\end{equation}
\end{lemma}

\begin{proof}
	Set $\tilde{\bs z}_{k+1}=(\tilde x_{k+1},y_{k+1})^{\intercal}$. The scheme can be written as a correction of the implicit Euler step:
	$$
	\tilde{\bs z}_{k+1}-\bs z_k
	=
	\alpha\mathcal G(\tilde{\bs z}_{k+1})
	-
	\alpha
	\begin{pmatrix}
		y_{k+1}-y_k\\
		0
	\end{pmatrix}.
	$$
	Using the Bregman expansion of $E$ gives
	\begin{align*}
		E(\tilde{\bs z}_{k+1})-E(\bs z_k)
		&=
		\langle\nabla E(\tilde{\bs z}_{k+1}),\tilde{\bs z}_{k+1}-\bs z_k\rangle
		-
		D_E(\bs z_k,\tilde{\bs z}_{k+1})\\
		&=
		\alpha\langle\nabla E(\tilde{\bs z}_{k+1}),\mathcal G(\tilde{\bs z}_{k+1})\rangle
		-
		D_E(\bs z_k,\tilde{\bs z}_{k+1})\\
		&\quad
		-\alpha\langle\nabla F(\tilde x_{k+1})-\nabla F(x^\star),y_{k+1}-y_k\rangle.
	\end{align*}
	By Theorem~\ref{th:SLy1} and
	$$
	 - D_E(\bs z_k,\tilde{\bs z}_{k+1})
	=
	 - D_F(x_k,\tilde x_{k+1}) - \frac{\mu}{2}\|y_k-y_{k+1}\|^2,
	$$
	we drop the negative term $- D_F(x_k,\tilde x_{k+1})$ and obtain
	$$
	E(\tilde{\bs z}_{k+1})-E(\bs z_k)
	\le
	-\alpha E(\tilde{\bs z}_{k+1})
	-\frac{\mu}{2}\|y_{k+1}-y_k\|^2
	-\alpha\langle\nabla F(\tilde x_{k+1})-\nabla F(x^\star),y_{k+1}-y_k\rangle.
	$$
	Rearranging gives \eqref{eq:EPC-prelim}.
\end{proof}

To cancel the sign-indefinite cross term in \eqref{eq:EPC-prelim}, we add a corrector step:
\begin{equation}\label{eq: EPC-VOS}
	\left\{
	\begin{aligned}
		\frac{\tilde x_{k+1}-x_k}{\alpha} &= y_k-\tilde x_{k+1},\\
		\frac{y_{k+1}-y_k}{\alpha} &= -\frac{1}{\mu}\left(\nabla F(\tilde x_{k+1})+N(y_{k+1})\right),\\
		\frac{x_{k+1}-x_k}{\alpha} &= y_{k+1}-x_{k+1}.
	\end{aligned}
	\right.
\end{equation}
Here $\tilde x_{k+1}$ is the predictor and $x_{k+1}$ is the corrector. Subtracting the first equation from the third gives
\begin{equation}\label{eq:xyrelation}
	x_{k+1}-\tilde x_{k+1}
	=
	\frac{\alpha}{1+\alpha}(y_{k+1}-y_k).
\end{equation}
This relation is illustrated in Figure~\ref{fig:EPC}. Since $x_{k+1}$ is obtained by a predictor--corrector extrapolation from $x_k$ and $y_k$, we call \eqref{eq: EPC-VOS} the EPC-VOS scheme.

\begin{figure}[htbp]
	\begin{center}
		\includegraphics[width=5.05cm]{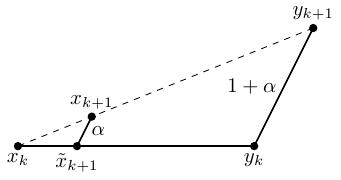}
		\caption{Extrapolation by the predictor--corrector method. The triangles $\Delta(x_k,\tilde x_{k+1},x_{k+1})$ and $\Delta(x_k,y_k,y_{k+1})$ are similar.}
		\label{fig:EPC}
	\end{center}
\end{figure}

\begin{theorem}[Convergence of the EPC-VOS method]
	\label{thm:conv-ex0-ode-NAG}
	Suppose that $F$ is convex and $L_F$-smooth, and that $N$ is $\mu$-strongly monotone, with $L_F>0$ and $\mu>0$. Let $(x_k,y_k)$ be generated by the EPC-VOS scheme~\eqref{eq: EPC-VOS} with $0<\alpha\le\sqrt{\mu/L_F}$. Then, for all $k\ge0$,
	\begin{equation}\label{eq:EPC-rate}
		E(x_{k+1},y_{k+1})
		\le
		\frac{1}{1+ \alpha}
		E(x_k,y_k).
	\end{equation}
\end{theorem}
\begin{proof}
	Let $\tilde{\bs z}_{k+1}=(\tilde x_{k+1},y_{k+1})$. Applying Lemma~\ref{lem:NAG-GS} to the first two equations of \eqref{eq: EPC-VOS} gives
	$$
	 E(\tilde{\bs z}_{k+1})
	\le
	\frac{ E(\bs z_k)}{1+\alpha}
	-\frac{\mu}{2(1+\alpha)}\|y_{k+1}-y_k\|^2
	-\frac{\alpha}{1+\alpha}\langle\nabla F(\tilde x_{k+1})-\nabla F(x^\star),y_{k+1}-y_k\rangle.
	$$
	Next, using the $L_F$-smoothness of $F$ and \eqref{eq:xyrelation},
	\begin{align*}
		&\quad  E(\bs z_{k+1})- E(\tilde{\bs z}_{k+1})
		=
		D_F(x_{k+1},x^\star)-D_F(\tilde x_{k+1},x^\star)\\
		&\le
		\langle\nabla F(\tilde x_{k+1})-\nabla F(x^\star),x_{k+1}-\tilde x_{k+1}\rangle
		+\frac{L_F}{2}\|x_{k+1}-\tilde x_{k+1}\|^2\\
		&=
		\frac{\alpha}{1+\alpha}\langle\nabla F(\tilde x_{k+1})-\nabla F(x^\star),y_{k+1}-y_k\rangle
		+\frac{L_F\alpha^2}{2(1+\alpha)^2}\|y_{k+1}-y_k\|^2.
	\end{align*}
	Adding the two estimates cancels the cross term and yields
	$$
	 E(\bs z_{k+1})
	\le
	\frac{1}{1+\alpha} E(\bs z_k)
	+
	\left(
	\frac{L_F\alpha^2}{2(1+\alpha)^2}
	-
	\frac{\mu}{2(1+\alpha)}
	\right)
	\|y_{k+1}-y_k\|^2.
	$$
	The last coefficient is nonpositive if $L_F\alpha^2\le\mu(1+\alpha)$. This holds for $\alpha \leq \sqrt{\mu/L_F}$. Therefore \eqref{eq:EPC-rate} follows.
\end{proof}

\subsection{One extra gradient step}

For the strongly convex problem $\min_x f(x)$, one can add a gradient descent step before the VOS update. Let
$$
x_k^+ = x_k-\frac1L\nabla f(x_k).
$$
We replace $x_k$ by $x_k^+$ in the first equation and obtain the EGD-VOS scheme
\begin{equation}\label{eq:HNAG-GS}
	\left\{
	\begin{aligned}
		\frac{x_{k+1}-x_k^+}{\alpha} &= y_k-x_{k+1},\\
		\frac{y_{k+1}-y_k}{\alpha} &= x_{k+1}-y_{k+1}-\frac1\mu\nabla f(x_{k+1}).
	\end{aligned}
	\right.
\end{equation}
We use the Lyapunov function
$$
E(x^+,y)=f(x^+)-f(x^\star)+\frac{\mu}{2}\|y-x^\star\|^2.
$$

\begin{theorem}\label{thm:hbGSGD}
	Assume $f\in\mathcal S_{\mu,L}^1$ with $0<\mu\le L<\infty$. For the scheme~\eqref{eq:HNAG-GS} with $\alpha=\sqrt{\mu/L}$,
	\begin{equation}\label{eq:decay-Lk-ex0}
		E(x_{k+1}^+,y_{k+1})
		\le
		\frac{1}{1+\sqrt{\mu/L}}E(x_k^+,y_k),
		\quad k\ge 0.
	\end{equation}
\end{theorem}

\begin{proof}
	Using
	$$
	f(x_{k+1})-f(x_k^+)
	=
	\langle\nabla f(x_{k+1}),x_{k+1}-x_k^+\rangle
	-
	D_f(x_k^+,x_{k+1}),
	$$
	and repeating the implicit-Euler Lyapunov estimate for \eqref{eq:HNAG-GS}, we obtain
	\begin{equation}\label{eq:conv-explicitEuler-agf-s}
		(1+\alpha)E(x_{k+1},y_{k+1})
		\le
		E(x_k^+,y_k)
		+
		\frac{\alpha^2}{2\mu}\|\nabla f(x_{k+1})\|_*^2.
	\end{equation}
	The extra gradient step gives, by \eqref{eq:fdecay},
	\begin{equation}\label{eq:GDdecay}
		E(x_{k+1}^+,y_{k+1})-E(x_{k+1},y_{k+1})
		=
		f(x_{k+1}^+)-f(x_{k+1})
		\le
		-\frac{1}{2L}\|\nabla f(x_{k+1})\|_*^2.
	\end{equation}
	Multiplying \eqref{eq:GDdecay} by $1+\alpha$ and adding it to \eqref{eq:conv-explicitEuler-agf-s} gives
	$$
	(1+\alpha)E(x_{k+1}^+,y_{k+1})
	\le
	E(x_k^+,y_k)
	+
	\left(
	\frac{\alpha^2}{2\mu}
	-
	\frac{1+\alpha}{2L}
	\right)
	\|\nabla f(x_{k+1})\|_*^2.
	$$
	For $\alpha=\sqrt{\mu/L}$, the coefficient of $\|\nabla f(x_{k+1})\|_*^2$ is nonpositive, because $L\alpha^2=\mu\le\mu(1+\alpha)$. Hence
	$$
	(1+\alpha)E(x_{k+1}^+,y_{k+1})
	\le
	E(x_k^+,y_k),
	$$
	which proves \eqref{eq:decay-Lk-ex0}.
\end{proof}

Substituting $x_k^+=x_k-\frac1L\nabla f(x_k)$ into \eqref{eq:HNAG-GS} gives the equivalent formulation
\begin{equation}\label{eq:HNAG-GS 2}
	\left\{
	\begin{aligned}
		\frac{x_{k+1}-x_k}{\alpha}
		&=
		y_k-x_{k+1}-\beta\nabla f(x_k),
		\quad \alpha\beta=\frac1L,\\
		\frac{y_{k+1}-y_k}{\alpha}
		&=
		x_{k+1}-y_{k+1}-\frac1\mu\nabla f(x_{k+1}).
	\end{aligned}
	\right.
\end{equation}
This is a discretization of the Hessian-driven Nesterov accelerated gradient (HNAG) flow introduced in~\cite{chen2019orderoptimizationmethodsbased}:
\begin{equation}\label{eq:flow}
	\left\{
	\begin{aligned}
		x'&=y-x-\beta\nabla f(x),\\
		y'&=x-y-\frac1\mu\nabla f(x).
	\end{aligned}
	\right.
\end{equation}
The HNAG flow also yields the HNAG$^{++}$ scheme~\cite{ChenXu2025Super-Fast}, obtained from \eqref{eq:HNAG-GS 2} with the larger step size
$$
\alpha=\sqrt{2}\sqrt{\mu/L}.
$$
For functions satisfying local asymptotic symmetry (LAS)~\cite{ChenXu2025Super-Fast}, the scheme attains the faster asymptotic linear rate
$$
(1+2\sqrt{2}\sqrt{\mu/L})^{-1}.
$$
This example shows that the VOS framework can yield sharp convergence rates comparable to those obtained through PEP~\cite{kim2016optimized_first_order_methods} and IQC~\cite{lessard2016analysis,taylor2018lyapunov} analyses.


\subsection{Homotopy perturbation for convex cases}
\label{sec: convex homotopy}

When $\mu=0$, we use a projected homotopy strategy. The energy decays at rate $\mathcal O(1/k^2)$ at stage endpoints. The same strategy is useful when $\mu>0$ is difficult to estimate.

\subsubsection{Projected perturbation}

Fix $\epsilon>0$. Consider the $\epsilon$-scaled VOS flow
\begin{equation}\label{eq:perturbed VOS}
	\left\{
	\begin{aligned}
		x'&=y-x,\\
		y'&\in-\frac{1}{\epsilon}\left(\nabla F(x)+N(y)\right),
	\end{aligned}
	\right.
\end{equation}
where $F$ is convex and $L_F$-smooth with $L_F>0$, and $N$ is monotone. The parameter $\epsilon$ changes only the auxiliary dynamics and preserves the equilibrium:
$$
x^\star=y^\star,
\qquad
0\in\nabla F(x^\star)+N(x^\star).
$$
This differs from Tikhonov or viscosity regularization, which perturbs the optimality system~\cite{AITikhonovAttouchLaszlo2024}, and from Catalyst~\cite{linCatalystAccelerationFirstorder}, where the regularized subproblem changes with the outer iterate. It also differs from restart methods~\cite{ODonoghueCandes2015}, which reset the dynamics rather than change the auxiliary time scale.

Define
\begin{equation}\label{eq: perturbed Lya fun}
	E(x,y;\epsilon)
	:=
	D_F(x,x^\star)+\frac{\epsilon}{2}\|y-x^\star\|^2,
\end{equation}
where $x^\star$ is a fixed solution. Here $\epsilon$ is treated as a parameter. For fixed $(x,y)$, $E(x,y;\epsilon)$ is nondecreasing in $\epsilon$.

\begin{lemma}\label{lem:pertrubed Lya property}
Let $E$ be the Lyapunov energy in \eqref{eq: perturbed Lya fun}, and let $\mathcal G_\epsilon$ denote the vector field of the $\epsilon$-scaled VOS flow \eqref{eq:perturbed VOS}. Then the perturbed Lyapunov estimate is	\begin{equation}\label{eq:pertrubed Lya property}
		-\inprd{\nabla E(x,y;\epsilon),\mathcal G_\epsilon(x,y)}
		\ge
		D_F(x,x^\star)
		=
		E(x,y;\epsilon)-\frac{\epsilon}{2}\|y-x^\star\|^2.
	\end{equation}
\end{lemma}

\begin{proof}
	Choose $n^\star\in N(x^\star)$ such that $\nabla F(x^\star)+n^\star=0$, and let $n(y)\in N(y)$. The same cancellation as in the proof of Theorem~\ref{th:SLy1} gives
	\begin{align*}
		-\inprd{\nabla E(x,y;\epsilon),\mathcal G_\epsilon(x,y)}
		={}&
		\inprd{\nabla F(x)-\nabla F(x^\star),x-x^\star}
		+
		\inprd{n(y)-n^\star,y-x^\star}.
	\end{align*}
	Monotonicity makes the second term nonnegative, and the first equals $D_F(x,x^\star)+D_F(x^\star,x)$. Dropping $D_F(x^\star,x)\ge0$ proves the estimate.
\end{proof}

The cross terms cancel as before, but monotonicity alone gives no positive multiple of $\|y-x^\star\|^2$. Rewriting the remaining Bregman term as the full energy therefore introduces $\frac{\epsilon}{2}\|y-x^\star\|^2$ on the right-hand side. We add projection to control this term. Assume that $R>0$ and
\begin{equation}
	x^\star\in B_R,
	\qquad
	B_R:=\{v\in V:\|v\|\le R\}.
\end{equation}
The projected EPC-VOS scheme is
\begin{equation}\label{eq:EPCepsilon}
    \left\{ 
\begin{aligned}
       \frac{\tilde x_{k+1}-x_k}{\alpha}
		&=
		y_k-\tilde x_{k+1},\\
		\frac{\tilde y_{k+1}-y_k}{\alpha}
		&\in
		-\frac{1}{\epsilon}\left(\nabla F(\tilde x_{k+1})+N(\tilde y_{k+1})\right),\\
		\frac{x_{k+1}-x_k}{\alpha}
		&=
		\tilde y_{k+1}-x_{k+1},\\
		y_{k+1} 
		&=
		P_{B_R}(\tilde y_{k+1})
		:=
		\operatorname{argmin}_{y\in B_R}\|y-\tilde y_{k+1}\|.	
\end{aligned}
    \right.
\end{equation}

The projection step in \eqref{eq:EPCepsilon} is explicit:
$$
P_{B_R}(\tilde y)
=
\min\left\{1,\frac{R}{\|\tilde y\|}\right\}\tilde y.
$$
Here $P_{B_R}(0)=0$. The projection acts as a safeguard and is inactive when $\tilde y_{k+1}\in B_R$. 
Since $x^\star\in B_R$, the variational inequality gives $\|y_{k+1}-x^\star\|^2+\|\tilde y_{k+1}-y_{k+1}\|^2\le\|\tilde y_{k+1}-x^\star\|^2$. Hence
\begin{equation}\label{eq:projection-decrease}
E(x_{k+1},y_{k+1};\epsilon)
\le
E(x_{k+1},\tilde y_{k+1};\epsilon),
\quad
\text{and }\ 
\|y_{k+1}-x^\star\|
\le
2R.
\end{equation}
Thus the projection step decreases the Lyapunov energy and keeps the stored auxiliary iterates uniformly bounded.

\begin{theorem}[Convergence of the projected EPC-VOS method]
	\label{thm:convergence rate of EPC-VOS}
	Suppose that $F$ is convex and $L_F$-smooth with $L_F>0$, that $N$ is monotone, and that
	$$
	0\in\nabla F(x^\star)+N(x^\star),
	\qquad
	x^\star\in B_R.
	$$
	Let $(x_k,y_k)$ be generated by \eqref{eq:EPCepsilon}, and assume that the implicit $\tilde y_{k+1}$-update is well defined. For $\epsilon>0$, if the step size $\alpha>0$ satisfies
	\begin{equation}\label{eq:projected-step-condition}
		L_F\alpha^2\le\epsilon(1+\alpha),
	\end{equation}
	then the perturbed contraction is
	\begin{equation}\label{eq:EPC-E-pre-bound}
		E(x_{k+1},y_{k+1};\epsilon)
		\le
		\frac{1}{1+\alpha}E(x_k,y_k;\epsilon)
		+
		\frac{2\alpha}{1+\alpha}\epsilon R^2.
	\end{equation}
In particular, $\alpha=\sqrt{\epsilon/L_F}$ is admissible, and
	\begin{equation}\label{eq: perturbed AOR-VOS Ealpha decay}
		E(x_k,y_k;\epsilon)
		\le
		\left(\frac{1}{1+\sqrt{\epsilon/L_F}}\right)^k
		E(x_0,y_0;\epsilon)
		+
		2\epsilon R^2.
	\end{equation}
\end{theorem}

\begin{proof}
	Repeat the predictor--corrector calculation in Theorem~\ref{thm:conv-ex0-ode-NAG}, with $\mu$ replaced by $\epsilon$ and \eqref{eq:pertrubed Lya property} in place of the strong Lyapunov estimate. Set $\bs z_k=(x_k,y_k)^{\intercal}$, $\tilde{\bs z}_{k+1}=(\tilde x_{k+1},\tilde y_{k+1})^{\intercal}$, and $\delta y_k=\tilde y_{k+1}-y_k$. The second update selects a direction $\mathcal G_\epsilon(\tilde{\bs z}_{k+1})$. The predictor expansion gives
	\begin{align}
		E(\tilde{\bs z}_{k+1};\epsilon)-E(\bs z_k;\epsilon)
		\le{}&-\alpha D_F(\tilde x_{k+1},x^\star)-D_E(\bs z_k,\tilde{\bs z}_{k+1})\notag\\
		&-\alpha\inprd{\nabla F(\tilde x_{k+1})-\nabla F(x^\star),\delta y_k}.
	\end{align}
	Here $E(\bs z;\epsilon)=E(x,y;\epsilon)$ and $D_E$ is its Bregman divergence with $\epsilon$ fixed. The corrector expansion is unchanged:
	\begin{align}
		&E(x_{k+1},\tilde y_{k+1};\epsilon)-E(\tilde x_{k+1},\tilde y_{k+1};\epsilon)\notag\\
		={}&\inprd{\nabla F(\tilde x_{k+1})-\nabla F(x^\star),x_{k+1}-\tilde x_{k+1}}+D_F(x_{k+1},\tilde x_{k+1}).
	\end{align}
	Using $(1+\alpha)(x_{k+1}-\tilde x_{k+1})=\alpha\delta y_k$, the three-point identity, and the projection decrease \eqref{eq:projection-decrease}, we obtain
	\begin{align}
		&E(x_{k+1},y_{k+1};\epsilon)-E(x_k,y_k;\epsilon)\notag\\
		\le{}&-\alpha D_F(x_{k+1},x^\star)-D_E(\bs z_k,\tilde{\bs z}_{k+1}) +(1+\alpha)D_F(x_{k+1},\tilde x_{k+1}).
	\end{align}
	As in the preceding proof, smoothness and \eqref{eq:projected-step-condition} give
	$$
	(1+\alpha)D_F(x_{k+1},\tilde x_{k+1})
	\le\frac{L_F\alpha^2}{2(1+\alpha)}\|\delta y_k\|^2
	\le\frac{\epsilon}{2}\|\delta y_k\|^2
	\le D_E(\bs z_k,\tilde{\bs z}_{k+1}).
	$$
	Thus the energy change is at most $-\alpha D_F(x_{k+1},x^\star)$. Recalling \eqref{eq: perturbed Lya fun} gives the additional term that projection must control:
	$$
	(1+\alpha)E(x_{k+1},y_{k+1};\epsilon)
	\le E(x_k,y_k;\epsilon)+\frac{\alpha\epsilon}{2}\|y_{k+1}-x^\star\|^2
	\le E(x_k,y_k;\epsilon)+2\alpha\epsilon R^2.
	$$
	This proves \eqref{eq:EPC-E-pre-bound}. Iteration gives \eqref{eq: perturbed AOR-VOS Ealpha decay}.
\end{proof}

For fixed $\epsilon$, reaching $E(x_k,y_k;\epsilon)=\mathcal O(\epsilon)$ requires
$$
k
=
\mathcal O\left(
\sqrt{\frac{L_F}{\epsilon}}\,|\ln\epsilon|
\right).
$$
The homotopy schedule below removes the logarithmic factor.


\subsubsection{Homotopy perturbation}

For a fixed $\epsilon>0$, Theorem~\ref{thm:convergence rate of EPC-VOS} gives
\begin{equation}
	E_{k+1}^{\epsilon}
	\le
	\delta_\epsilon E_k^\epsilon
	+
	(1-\delta_\epsilon)C_R\epsilon,
	\qquad
	\delta_\epsilon:=\frac{1}{1+\alpha_\epsilon},
	\qquad
	\alpha_\epsilon:=\sqrt{\frac{\epsilon}{L_F}},
\end{equation}
where
$$
E_k^\epsilon:=E(x_k,y_k;\epsilon),
\qquad
C_R:=2R^2.
$$
Each fixed-$\epsilon$ stage contracts the energy toward $C_R\epsilon$. At the end of a stage, we halve $\epsilon_\ell$ and increase $m_\ell$ by a factor of approximately $\sqrt2$. The total work is then comparable to the length of the last stage.

\begin{algorithm}
	\caption{Projected homotopy perturbation for convex optimization.}
	\label{alg: homotopy perturbed VOS}
	\begin{algorithmic}[1]
		\STATE \textbf{Parameters:} initial pair $(x_0,y_0)$; radius $R > 0$ with $x^\star\in B_R$; Lipschitz constant $L_F>0$; initial perturbation $\epsilon_0 > 0$; initial inner iteration count $m_0 \in \mathbb{N}$
		\STATE Set $y_0=P_{B_R}(y_0)$ and $k_0=0$
		\FOR{$\ell=0,1,\dots$}
			\STATE Set $\alpha_\ell=\sqrt{\epsilon_\ell/L_F}$ and apply the projected EPC-VOS scheme~\eqref{eq:EPCepsilon} for $m_\ell$ steps with $(\epsilon,\alpha)=(\epsilon_\ell,\alpha_\ell)$
             \STATE $\epsilon_{\ell+1}=\epsilon_\ell/2,\quad m_{\ell+1}=\lceil\sqrt2\,m_\ell\rceil$
			\STATE $k_{\ell+1}=k_\ell+m_\ell$
		\ENDFOR
	\end{algorithmic}
\end{algorithm}

%
A practical stopping criterion can terminate the algorithm within a stage, before all $m_{\ell}$ iterations are completed.

\begin{theorem}[Homotopy rate]\label{thm:homotopy-perturbation}
    Choose $(x_0,y_0)$ and $\epsilon_0$ such that
	$$
 y_0 \in B_R, \quad \text{ and }	\quad E(x_0,y_0;\epsilon_0)\le (4R^2+1)\epsilon_0,
	$$
	and set
	$$
	m_0\geq M(\epsilon_0,R,L_F):=
	\left(\sqrt{\frac{L_F}{\epsilon_0}}+1\right)
	\ln\bigl(2(4R^2+1)\bigr).
	$$
	Let $(x_k,y_k)$ be generated by Algorithm~\ref{alg: homotopy perturbed VOS}. Suppose that the assumptions of Theorem~\ref{thm:convergence rate of EPC-VOS} hold at every stage with the same ball $B_R$. Then 
	\begin{equation}\label{eq:Ek-ek}
		 E(x_{k_\ell},y_{k_\ell};\epsilon_\ell)
		\le
		(4R^2+1)\epsilon_\ell,
		\quad \ell\ge 0.
	\end{equation}
	Moreover,
	\begin{equation}\label{eq:rate}
		 E(x_{k_\ell},y_{k_\ell};\epsilon_\ell)
		\le
		\frac{4R^2+1}{\left(c k_\ell+\epsilon_0^{-1/2}\right)^2},
		\quad \ell\ge 0,
	\end{equation}
	where
	$\displaystyle c=\frac{\sqrt2-1}{\sqrt{\epsilon_0}\left(m_0+\frac{1}{\sqrt2-1}\right)}.$
	If $0<\epsilon_0\le L_F$ and $m_0=\lceil M(\epsilon_0,R,L_F)\rceil$, then $\mathcal O(\sqrt{L_F/\epsilon})$ inner iterations suffice to reach energy accuracy $\mathcal O(\epsilon)$, with constants depending on $R$.
\end{theorem}

\begin{proof}
	The initial energy bound gives \eqref{eq:Ek-ek} for $\ell=0$. Using $m_\ell\alpha_\ell\ge m_0\alpha_0$ and $\alpha_\ell\le\alpha_0$ gives
	$$
	(1+\alpha_\ell)^{-m_\ell}
	\le\exp\left(-\frac{m_\ell\alpha_\ell}{1+\alpha_\ell}\right)
		\le\exp\left(-\frac{m_0\alpha_0}{1+\alpha_0}\right)
	\le\frac{1}{2(4R^2+1)},
	$$
	since $m_0\geq M(\epsilon_0,R,L_F)= (1+1/\alpha_0)\ln(2(4R^2+1))$.
	Run stage $\ell$ with $\epsilon_\ell$, then lower the energy parameter to $\epsilon_{\ell+1}=\epsilon_\ell/2$. The fixed-parameter estimate and the induction hypothesis give
	\begin{align*}
	E(x_{k_{\ell+1}},y_{k_{\ell+1}};\epsilon_{\ell+1})
	&\le E(x_{k_{\ell+1}},y_{k_{\ell+1}};\epsilon_\ell)\\
	&\le \frac{E(x_{k_\ell},y_{k_\ell};\epsilon_\ell)}{2(4R^2+1)}+2R^2\epsilon_\ell\\
	&\le\left(\frac12+2R^2\right)\epsilon_\ell
	=(4R^2+1)\epsilon_{\ell+1}.
	\end{align*}
	This proves \eqref{eq:Ek-ek}. To account for integer rounding, use $m_{\ell+1}\le\sqrt2\,m_\ell+1$, which implies
	$$
	m_\ell\le (\sqrt2)^{\ell}\left(m_0+\frac{1}{\sqrt2-1}\right).
	$$
	Summing this bound yields
	$$
	k_\ell = \sum_{j = 0}^{\ell -1} m_{j} \le\frac{m_0+1/(\sqrt2-1)}{\sqrt2-1}(2^{\ell/2}-1)
	=\frac1c\left(\epsilon_\ell^{-1/2}-\epsilon_0^{-1/2}\right).
	$$
	Combining this with \eqref{eq:Ek-ek} gives \eqref{eq:rate}. For $m_0=\lceil M(\epsilon_0,R,L_F)\rceil$ and $\epsilon_0\le L_F$, we have $\sqrt{\epsilon_0}m_0\le(\sqrt{L_F}+\sqrt{\epsilon_0})\ln(2(4R^2+1))+\sqrt{\epsilon_0}$. Thus $c^{-1}=\mathcal O(\sqrt{L_F})$ for fixed $R$, which gives the stated complexity.
\end{proof}

\begin{remark}[Energy decay and strong convergence]\rm
Since $F$ is convex and $L_F$-smooth, co-coercivity gives, as in Remark~\ref{rem:bregman-gradient},
$$
\|\nabla F(x_{k_\ell})-\nabla F(x^\star)\|_*^2
\le 2L_F D_F(x_{k_\ell},x^\star)
\le 2L_F E(x_{k_\ell},y_{k_\ell};\epsilon_\ell)
=\mathcal O(k_\ell^{-2}).
$$
Thus the homotopy estimate controls the squared gradient difference. 
Strong convergence of the iterates requires additional properties of $F$, as described below.

A sufficient condition for $x_{k_\ell}\to x^\star$ strongly is
$$
\inf_{\|x-x^\star\|\ge r}D_F(x,x^\star)>0
\qquad\text{for every }r>0.
$$
Uniform convexity of $F$ ensures this condition. Ordinary coercivity only bounds sublevel sets; in finite dimensions, coercivity of $D_F(\cdot,x^\star)$ and its unique zero at $x^\star$ suffice, whereas in Hilbert spaces further compactness or a growth condition is needed. Since $\epsilon_\ell\to0$, strong convergence of $y_{k_\ell}$ requires a separate argument.
\end{remark}

\begin{remark}\rm
A practical choice is $\epsilon_0=0.1L_F$ when the initial energy condition holds. The stage length depends only logarithmically on $R$, although the energy bound contains $R^2$. If $2(4R^2+1)=10^4$, then
$$
m_0\geq M(\epsilon_0,R,L_F)= 
(\sqrt{10}+1)\ln(10^4)\approx38.336.
$$
A smaller $\epsilon_0$ requires a larger $m_0$ and must still satisfy the initial energy condition. 
\end{remark}

\begin{remark}\rm
For AOR-VOS, the additional cross term depends on $\epsilon$ through $\alpha$. The argument in~\cite{ChenWei2026ATPD} may be adapted to account for this dependence in the homotopy analysis.
\end{remark}

\subsection{Examples}
\label{sec:examples}

We apply the preceding schemes to some examples outlined in Section~\ref{sec:example}. The same framework gives both the derivation of the algorithms and their convergence analysis.

\subsubsection{\bf Accelerated gradient methods for convex optimization}

For convex optimization $\min_x f(x)$, Polyak~\cite{polyak1964some} introduced the Heavy-Ball (HB) method
\begin{equation}
	x_{k+1}=x_k-\gamma\nabla f(x_k)+\beta(x_k-x_{k-1}).
\end{equation}
For quadratic objectives with spectrum in $[\mu,L]$, the parameters
\begin{equation}\label{eq:polyak}
	\gamma=\frac{4}{(\sqrt L+\sqrt\mu)^2},
	\qquad
	\beta=\left(\frac{\sqrt L-\sqrt\mu}{\sqrt L+\sqrt\mu}\right)^2
\end{equation}
give an accelerated linear rate.

This acceleration does not extend globally to all smooth strongly convex functions. The parameters \eqref{eq:polyak} may fail to converge~\cite{lessard2016analysis}, while globally convergent variants generally have the gradient-descent rate $1-\mathcal O(1/\kappa)$~\cite{GhadimFeyzmaJohans2015Global,sun2019non,saab2022adaptive,shi2022understanding}. Moreover, no fixed-parameter HB method accelerates over the full smooth strongly convex class~\cite{goujaud2023provable}.

Nesterov's accelerated gradient method~\cite{Nesterov:2013Introductory} uses a different mechanism: the gradient is evaluated at an extrapolated point,
\begin{equation}\label{eq:NAG}
\begin{aligned}
	x_{k+1}
	&=
	x_k-\gamma\nabla f\bigl(x_k+\beta(x_k-x_{k-1})\bigr)+\beta(x_k-x_{k-1}),\\
	& \rho=\frac{\mu}{L},
	\quad
	\gamma =\frac1L,
	\quad
	\beta=\frac{1-\sqrt{\rho}}{1+\sqrt{\rho}}.
\end{aligned}
\end{equation}
Using estimate sequences, Nesterov proved the accelerated rate $1-\sqrt{1/\kappa(f)}$. It is also useful to note that NAG can be derived as a discretization of the HNAG flow \eqref{eq:flow}, with different step sizes for the $x$- and $y$-equations; see~\cite{ChenXu2025Super-Fast}.

The AOR-HB method~\cite{wei2024accelerated} restores global acceleration by extrapolating the gradient.
For $0<\mu<L$, choosing $\alpha=\sqrt{\mu/(L-\mu)}$ and eliminating $y_k$ in the forward AOR-VOS scheme \eqref{eq:AOR-VOS-forward}
gives
\begin{equation}
\begin{aligned}
x_{k+1}
&=x_k-\gamma\bigl(2\nabla f(x_k)-\nabla f(x_{k-1})\bigr)+\beta(x_k-x_{k-1}),\\
\rho &=\frac{\mu}{L}, \qquad
\gamma
=\frac{1}{L\left(1+2\sqrt{\rho(1-\rho)}\right)},
\qquad
\beta=\frac{1}{1+2\sqrt{\rho(1-\rho)}}.
\end{aligned}
\end{equation}
The method modifies HB through gradient extrapolation and different parameter choices. Compared with AOR-HB in~\cite{wei2024accelerated}, this formulation improves the contraction factor from $1/(1+\frac12\alpha)$ to $1/(1+\alpha)$ and increases the step parameter from $\sqrt{\rho}$ to $\sqrt{\rho/(1-\rho)}$.

We test the methods on smooth multidimensional piecewise objectives, following~\cite{VanScoyFreemanLynch2018}. Let
\begin{equation}\label{eq:smooth-strongly-convex-ex}
	f(x)
	=
	\sum_{i=1}^p h(a_i^\top x-b_i)+\frac{\mu}{2}\|x\|^2,
	\quad
	h(s)=
	\begin{cases}
		\frac12s^2e^{-r/s}, & s>0,\\
		0, & s\le 0.
	\end{cases}
\end{equation}
Here $A=[a_1,\ldots,a_p]\in\mathbb R^{d\times p}$ and $b\in\mathbb R^p$. We scale $A$ so that $\|A\|=\sqrt{L-\mu}$. Then $f$ is $\mu$-strongly convex and $L$-smooth. In the experiment, the entries of $A$ and $b$ are sampled from the standard normal distribution and then $A$ is rescaled. We set
$$
\mu=1,
\quad
L=10^4,
\quad
d=200,
\quad
p=20,
\quad
r=10^{-6}.
$$

\begin{figure}[htp]
	\begin{minipage}{0.50\textwidth}
		\centering
		\includegraphics[width=0.8\linewidth]{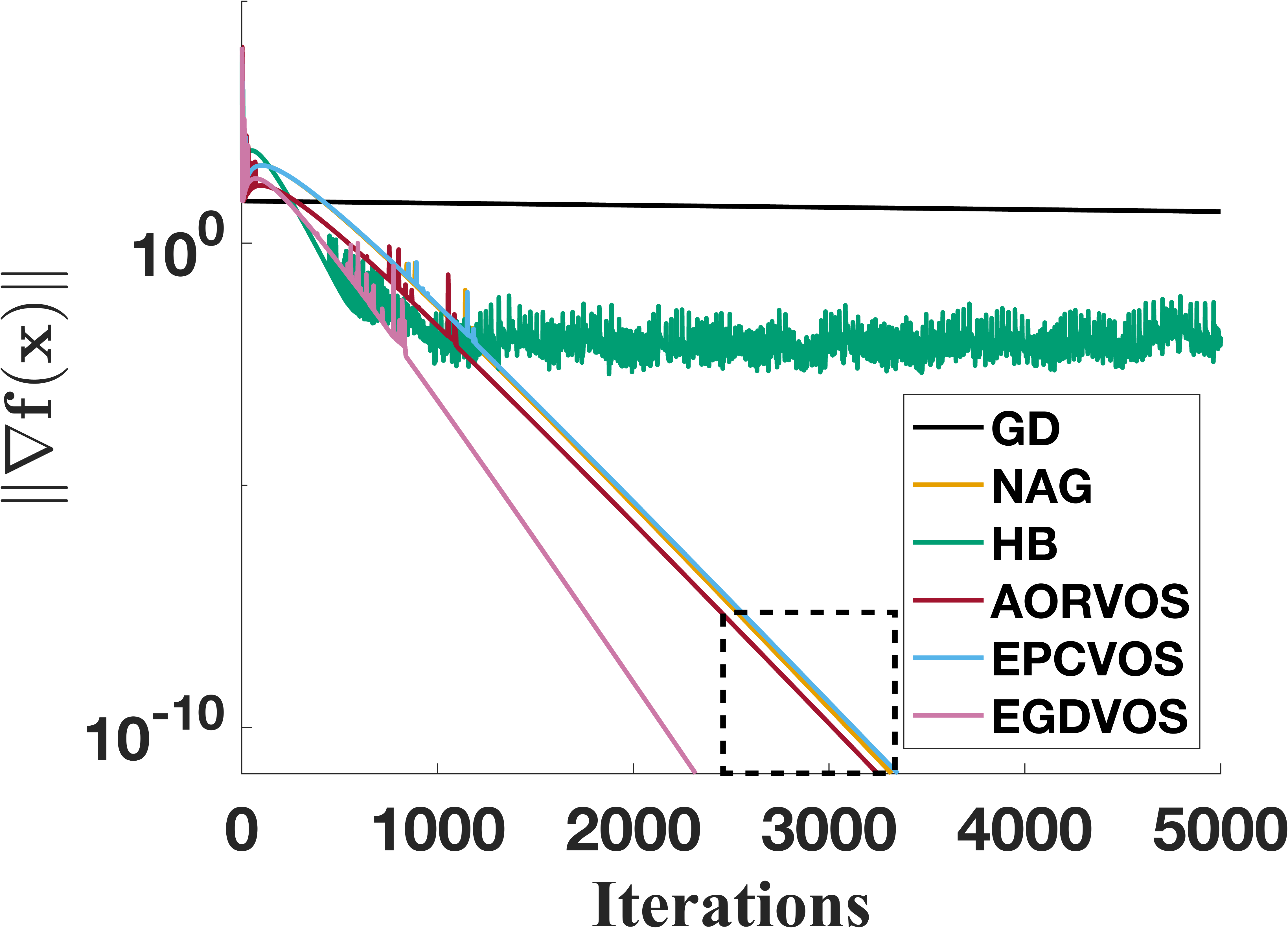}
	\end{minipage}
	\hfill
	\begin{minipage}{0.45\textwidth}
		\centering
		\includegraphics[width=0.8\linewidth]{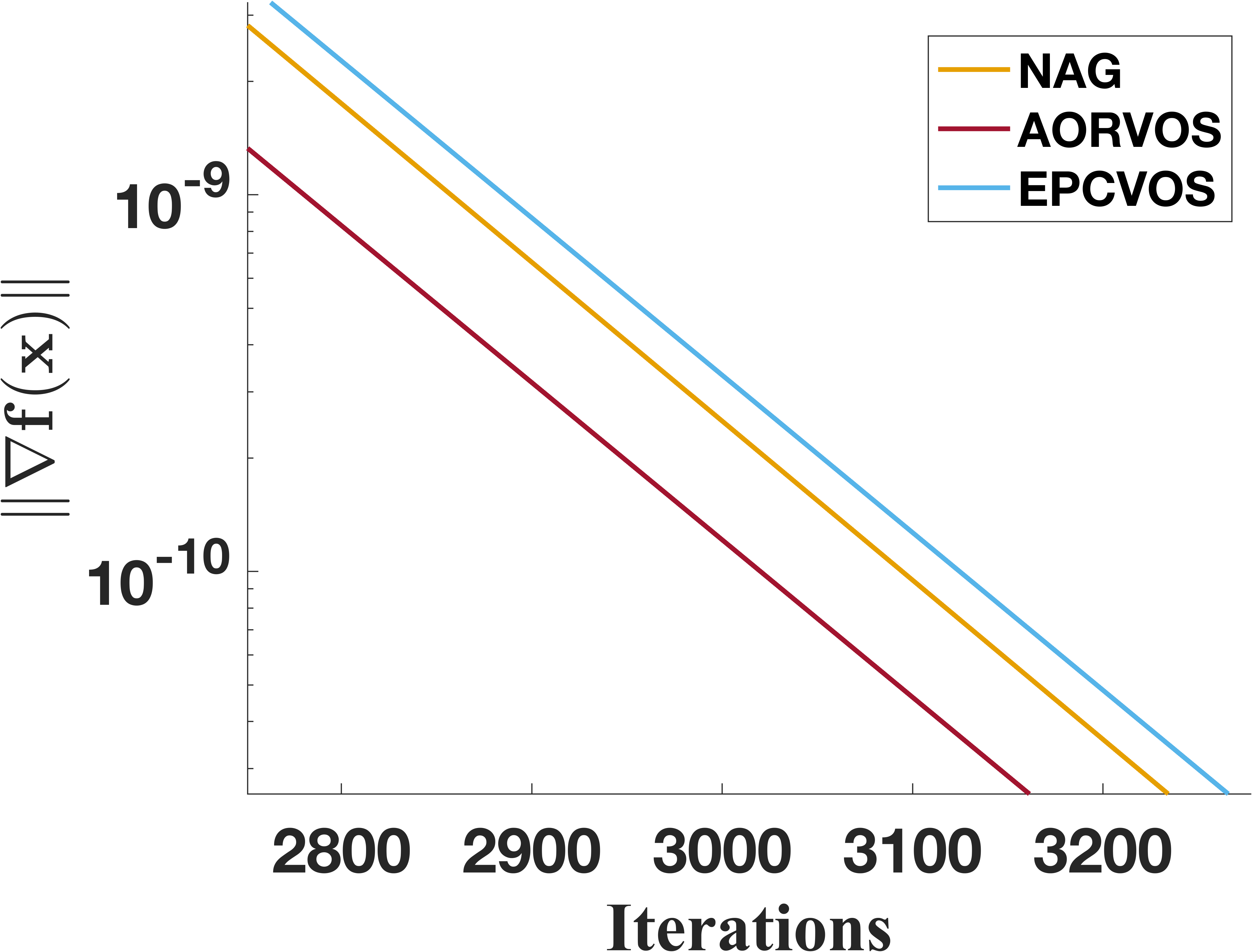}
	\end{minipage}
\caption{Smooth strongly convex problem \eqref{eq:smooth-strongly-convex-ex}. Left: decay of $\|\nabla f(x_k)\|$ versus iteration. Right: zoom-in of the dashed rectangle.}	\label{fig:multiobj_graderr}
\end{figure}

We compare three VOS variants, AOR-VOS \eqref{eq:AOR-VOS-forward}, EPC-VOS \eqref{eq: EPC-VOS}, and EGD-VOS \eqref{eq:HNAG-GS}, with GD, NAG \eqref{eq:NAG}, and HB with parameters \eqref{eq:polyak}. Figure~\ref{fig:multiobj_graderr} shows the decay of the computable quantity $\|\nabla f(x_k)\|$. Our analysis focuses on Bregman energies rather than gradient norms; their relation is given in Remark~\ref{rem:bregman-gradient}. In this test, HB reaches a plateau, while the other methods show linear convergence. GD is slow because the condition number is large and its rate is not accelerated. The VOS methods are comparable to NAG. EGD-VOS is faster in this example because it allows the larger step size $\alpha=\sqrt{2\mu/L}$ and is equivalent to the HNAG$^{++}$ scheme developed in~\cite{ChenXu2025Super-Fast}.

\subsubsection{\bf Accelerated proximal-gradient methods}

Consider the composite optimization problem
$$
\min_x \{f(x)+g(x)\},
$$
where $f$ is $\mu$-strongly convex and $L$-smooth, and $g$ is proper, closed, convex, and possibly nonsmooth. This problem fits the VOS framework with
$$
\nabla F(x)=\nabla f(x)-\mu x,
\qquad
N(y)=\partial g(y)+\mu y.
$$
Thus the nonsmooth part is treated implicitly through the resolvent of $N$, equivalently through the proximal operator of $g$.

The EPC-VOS discretization gives the accelerated proximal-gradient scheme
\begin{equation}
\left\{
\begin{aligned}
\tilde x_{k+1}
&=
\frac{1}{1+\alpha}(x_k+\alpha y_k),\\
y_{k+1}
&=
\operatorname{prox}_{\frac{\alpha}{(1+\alpha)\mu}g}
\left(
\frac{1}{1+\alpha}y_k
+
\frac{\alpha}{(1+\alpha)\mu}
\bigl(\mu\tilde x_{k+1}-\nabla f(\tilde x_{k+1})\bigr)
\right),\\
x_{k+1}
&=
\frac{1}{1+\alpha}(x_k+\alpha y_{k+1}).
\end{aligned}
\right.
\end{equation}
This is the proximal analogue of the smooth EPC-VOS method. The proximal step is equivalent to the implicit inclusion
$$
0\in
\partial g(y_{k+1})
+
\frac{(1+\alpha)\mu}{\alpha}
\left(
y_{k+1}
-
\frac{1}{1+\alpha}y_k
-
\frac{\alpha}{(1+\alpha)\mu}
\bigl(\mu\tilde x_{k+1}-\nabla f(\tilde x_{k+1})\bigr)
\right).
$$
Backward and forward VOS discretizations give closely related proximal variants. For example, the backward AOR-VOS scheme \eqref{eq:AOR-VOS-backward} is:
\begin{equation}
\left\{
	\begin{aligned}
	y_{k+1} &= \text{prox}_{\frac{\alpha}{(1+\alpha)\mu} g} \left(\frac{1}{1+\alpha} y_k + \frac{\alpha}{(1+\alpha)\mu} (\mu x_k - \nabla f(x_k))\right),\\
	x_{k+1} &= \frac{1}{1+\alpha}(x_k + 2\alpha y_{k+1} - \alpha y_k).
\end{aligned}
\right.
\end{equation}

Since $g$ is treated implicitly, the algorithm and proof extend directly from the smooth VOS setting. The gradient $\nabla g$ is replaced by an element of $\partial g$, and the key property is the monotonicity of $\partial g$. This differs from extrapolated-sequence methods, where extending smooth acceleration to the composite case requires a more careful proximal-gradient construction, as in FISTA~\cite{beck2009fast}.

\begin{figure}
	\centering
	\hfill
	\begin{minipage}{0.65\textwidth}
		\centering
		\begin{minipage}{0.47\linewidth}
			\centering
			\includegraphics[width=0.9\linewidth]{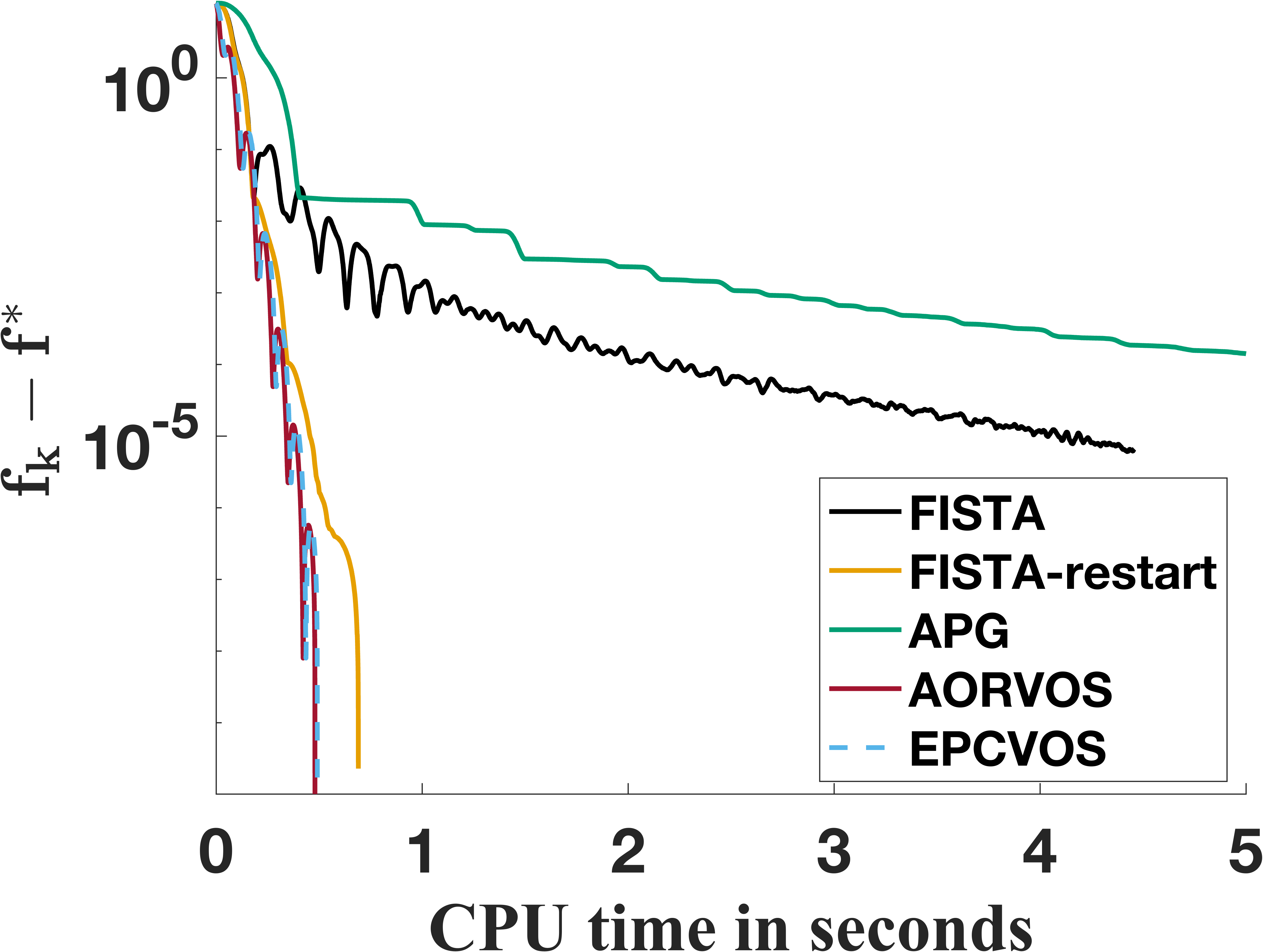}
			
			\small (a) Strongly convex case with $\mu=0.3$ and $\lambda=0.001$.
		\end{minipage}
		\hfill
		\begin{minipage}{0.47\linewidth}
			\centering
			\includegraphics[width=0.9\linewidth]{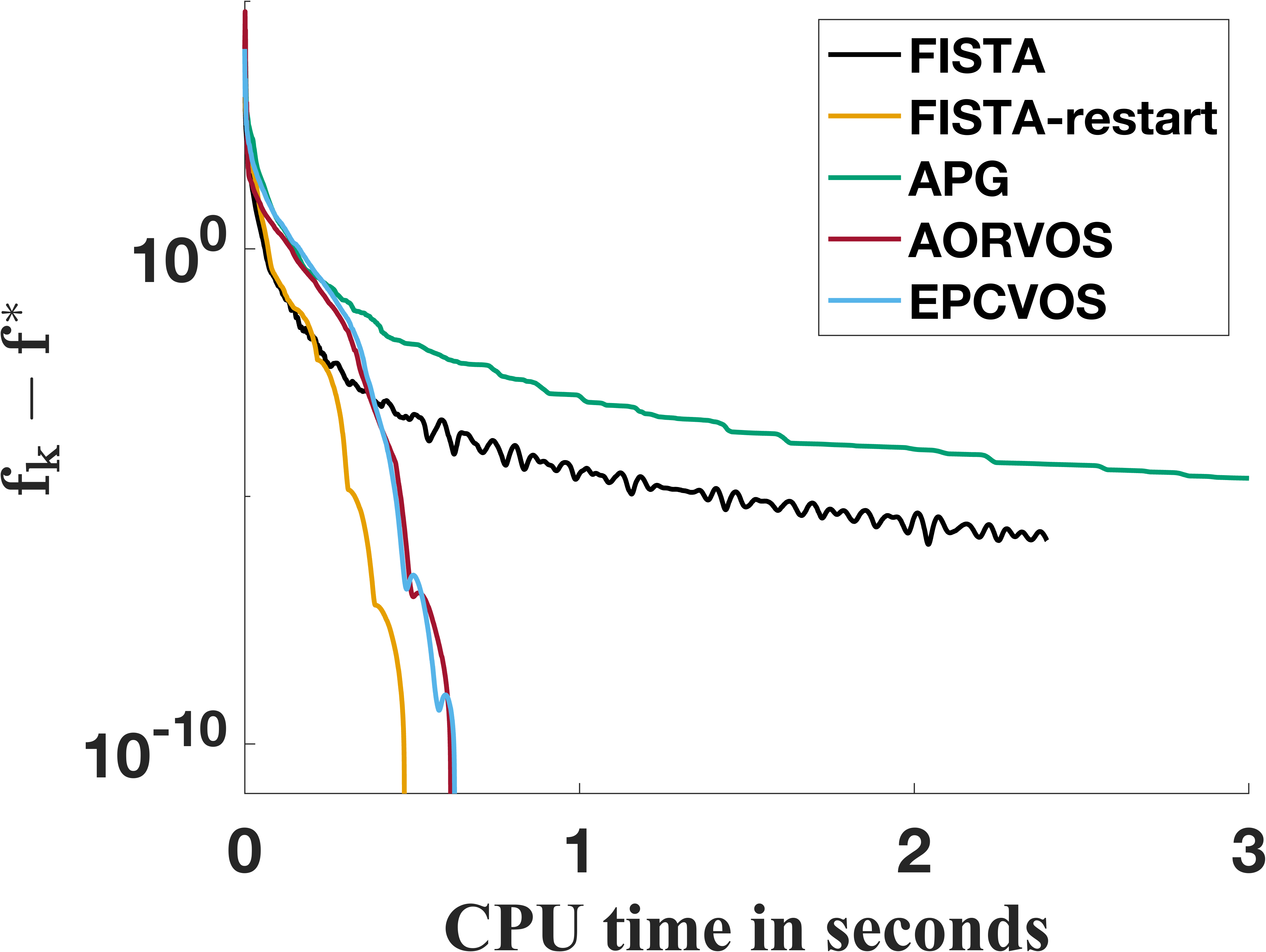}
			
			\small (b) Convex case with $\mu=0$ and $\lambda=0.1$.
		\end{minipage}
		\captionof{figure}{Regularized logistic regression \eqref{eq:log reg}. Objective value gap versus CPU time.}
		\label{fig:logreg_VOSnonsmooth}
	\end{minipage}
	\hfill
	\begin{minipage}{0.33\textwidth}
		\centering
		\includegraphics[width=0.825\linewidth]{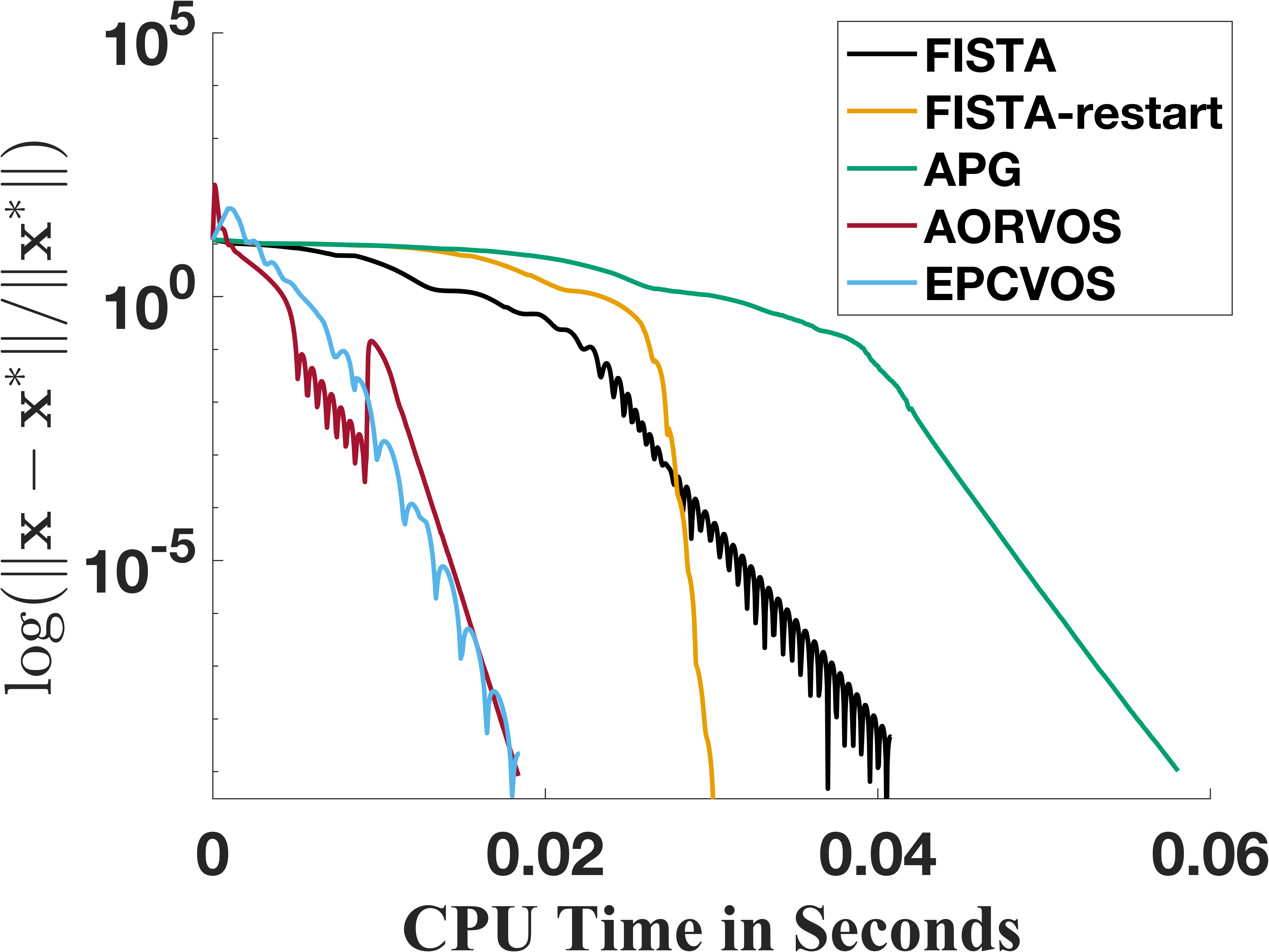}
		\captionof{figure}{$\ell_1-\ell_2$ minimization problem \eqref{eq:LASSOl1-2} with $\lambda=0.5$. Relative error versus CPU time.}
		\label{fig:lasso}
	\end{minipage}
\end{figure}

We test the methods on regularized logistic regression:
\begin{equation}\label{eq:log reg}
	\min_{x\in\mathbb R^d}
	\left\{
	\sum_{i=1}^m\log\left(1+\exp(-b_i a_i^\top x)\right)
	+
	\frac{\mu}{2}\|x\|^2
	+
	\lambda\|x\|_1
	\right\},
\end{equation}
where $(a_i,b_i)\in\mathbb R^d\times\{-1,1\}$. The smooth part
$$
f(x)=\sum_{i=1}^m\log\left(1+\exp(-b_i a_i^\top x)\right)+\frac{\mu}{2}\|x\|^2
$$
is $\mu$-strongly convex and $L$-smooth with
$\displaystyle L=\frac14\lambda_{\max}\left(\sum_{i=1}^m a_i a_i^\top\right)+\mu.$
We use the \texttt{a2a} dataset~\cite{chang2011libsvm}, with $m=2265$ and $d=119$. We stop each method after $20000$ iterations or when
$$
\|x_{k+1} - x_k\| \leq 10^{-8}\|x_k\|.
$$
This practical stopping rule does not certify optimality; Figure~\ref{fig:logreg_VOSnonsmooth} compares the decay of the full objective gap. 

We compare AOR-VOS and EPC-VOS with FISTA~\cite{beck2009fast}, FISTA with adaptive restart~\cite{ODonoghueCandes2015}, and the APG method of~\cite{li2015accelerated}. Figure~\ref{fig:logreg_VOSnonsmooth} (a) shows the strongly convex case with $\mu=0.3$ and $\lambda=0.001$. Figure~\ref{fig:logreg_VOSnonsmooth} (b) shows the merely convex case $\mu=0$ and $\lambda=0.1$, where EPC-VOS uses the projection in \eqref{eq:EPCepsilon} and homotopy perturbation with $R = 100, \epsilon_0=0.2L$ and $m_0=20$. AOR-VOS uses homotopy perturbation with $\epsilon_0=0.002L$ and $m_0=150$. These stage lengths are empirical choices. The EPC choice is below the sufficient bound $M(0.2L,100,L)\approx36.535$, and the EPC theorem does not cover the AOR homotopy variant. In both tests, the VOS methods show much faster convergence than FISTA and APG. The VOS methods with homotopy perturbation are comparable to FISTA with adaptive restart, while in the strongly convex case, they use the strong convexity parameter to accelerate convergence.

Next, we consider the following nonconvex Lasso-type problem arising in compressed sensing \cite{lou2018fast,yin2015minimization}:
\begin{equation}\label{eq:LASSOl1-2}
    \min_{x} \frac{1}{2} \|Ax - b\|^2_2 + \lambda (\|x\|_1 - \|x\|_2),
\end{equation}
where $\|\cdot\|_i$ denotes the $\ell_i$ norm for $i= 1,2$. Although $g(x) =  \lambda (\|x\|_1 - \|x\|_2)$ is nonconvex, its proximal operator has a closed-form expression \cite[Lemma 1]{lou2018fast}. This experiment is included only as a numerical illustration; the convex convergence theory developed above does not apply directly to this nonconvex objective. We generate a Gaussian matrix $A$ of size $128\times 512$ and a sparse reference vector $x$ with $5$ nonzero entries. The data vector is obtained by matrix-vector multiplication, $b=Ax$. We set $\lambda = 0.5$ and use the step size $\alpha = 1/L_F = 1/\lambda_{\max}(AA^{\top})$ in FISTA and APG. For VOS methods using homotopy perturbation, we set $R = 50, \epsilon_0=1$ and $m_0=60$ for projected EPC-VOS and $\epsilon_0=0.6$ and $m_0=180$ for AOR-VOS. We plot the decay of the relative $\ell_2$ error in Figure \ref{fig:lasso}. The VOS methods reach the prescribed accuracy in less CPU time in this example.


\smallskip



\section{Extensions}
\label{sec:exten}

We extend the framework to monotone inclusions of the form
\begin{equation}\label{eq:mi}
	0\in\mathcal L(x):=\nabla f(x)+\partial g(x)+Kx,
\end{equation}
where $f$ is convex and $L$-smooth, $g$ is proper, closed, and convex, $f+g$ is $\mu$-strongly convex, and $K$ is a bounded skew-symmetric linear operator, $K^\intercal=-K$. This formulation separates the smooth gradient, nonsmooth monotone, and skew-symmetric components. We restrict the convergence analysis and applications below to $L>0$. 

\subsection{Flow and stability}

We consider the inclusion
\begin{equation}\label{eq:ode}
	\left\{
	\begin{aligned}
		x'&\in y-x-\beta\bigl(\mathcal L(x)+K(y-x)\bigr),\\
		y'&\in\gamma^{-1}\bigl(\mu(x-y)-\mathcal L(x)-Kx'\bigr),\\
		\gamma'&=\mu-\gamma.
	\end{aligned}
	\right.
\end{equation}
Here $\beta(t)\ge0$ and $\gamma(t)>0$. When $K=0$, the system reduces to the Hessian-driven Nesterov accelerated gradient (HNAG) flow~\cite{chen2019orderoptimizationmethodsbased}; see also \eqref{eq:flow}. When $\beta=0$, it reduces to the accelerated primal--dual gradient flow~\cite{Luo2024ax}. Primal--dual dynamics with Hessian-driven damping also appear in~\cite{AIHeTianLiFang2025}; here we study the Lyapunov estimate and discretization of the coupled system \eqref{eq:ode}.

For $\xi\in\partial g(x)$, define
$$
q = q(x,y;\xi):=\nabla f(x)+\xi+Ky.
$$
Let $	\bs z=(x,y,\gamma)\in V\times V\times\mathbb R_+.$ Since $q(x,y;\xi)\in\mathcal L(x)+K(y-x)$, the vector field of \eqref{eq:ode} can be written as
\begin{equation}\label{eq:G-extension}
	\mathcal G(\bs z;\xi)
	:=
	\begin{pmatrix}
		y-x-\beta q\\
		\gamma^{-1}\left[\mu(x-y)-(\nabla f(x)+\xi+Kx)-K\bigl(y-x-\beta q\bigr)\right]\\
		\mu-\gamma
	\end{pmatrix}.
\end{equation}
Thus \eqref{eq:ode} is equivalent to the differential inclusion
$$
\bs z'=\mathcal G(\bs z;\xi),
\qquad
\xi\in\partial g(x).
$$

For $\bs z=(x,y,\gamma)\in V\times V\times\mathbb R_+$, define
\begin{equation}\label{eq:Lyap}
	E(\bs z):=D_{f+g}(x,x^\star)+\frac{\gamma}{2}\|y-x^\star\|^2,
\end{equation}
where $x^\star$ solves \eqref{eq:mi}. Here $D_{f+g}(x,x^\star)$ denotes the subgradient Bregman divergence with the fixed subgradient $-Kx^\star\in\partial(f+g)(x^\star)$, namely
$$
D_{f+g}(x,x^\star):=(f+g)(x)-(f+g)(x^\star)+\dual{Kx^\star,x-x^\star}.
$$
Choosing $\xi^\star\in\partial g(x^\star)$ so that $\nabla f(x^\star)+\xi^\star+Kx^\star=0$, the $\mu$-strong convexity of $f+g$ gives
$$
E(\bs z)\ge\frac{\mu}{2}\|x-x^\star\|^2+\frac{\gamma}{2}\|y-x^\star\|^2\ge0.
$$

For later use, let
\begin{equation}\label{eq:E-x-p}
	\partial E(\bs z;\xi)
	:=
	\begin{pmatrix}
		\nabla f(x)+\xi+Kx^\star\\
		\gamma(y-x^\star)\\
		\frac12\|y-x^\star\|^2
	\end{pmatrix},
	\qquad
	\xi\in\partial g(x).
\end{equation}
If $g$ is smooth and $\xi=\nabla g(x)$, then
$$
\nabla f(x)+\xi+Kx^\star
=
\nabla(f+g)(x)-\nabla(f+g)(x^\star)
=
\partial_xD_{f+g}(x,x^\star).
$$

\begin{theorem}[Strong Lyapunov property]
	\label{thm:strong-lyap-extension}
	For every $\bs z=(x,y,\gamma)\in V\times V\times\mathbb R_+$ and every $\xi\in\partial g(x)$, the following pointwise strong Lyapunov property holds:
	\begin{equation}\label{eq:sL-ode}
		-\dual{\partial E(\bs z;\xi),\mathcal G(\bs z;\xi)}
		\ge
		E(\bs z)
		+
		\beta\|\nabla f(x)+\xi+Ky\|_*^2
		+
		\frac{\mu}{2}\|y-x\|^2.
	\end{equation}
\end{theorem}

\begin{proof}
	Set $q:=q(x,y;\xi)=\nabla f(x)+\xi+Ky$. By \eqref{eq:G-extension} and \eqref{eq:E-x-p},
	$$
	-\dual{\partial E(\bs z;\xi),\mathcal G(\bs z;\xi)}
	=
	\mathbb I_x+\mathbb I_y+\mathbb I_\gamma,
	$$
	where
	$$
	\begin{aligned}
		\mathbb I_x&:=-\dual{\nabla f(x)+\xi+Kx^\star,\ y-x-\beta q},\\
		\mathbb I_y&:=-\dual{y-x^\star,\ -K(y-x-\beta q)+\mu(x-y)-\nabla f(x)-\xi-Kx},\\
		\mathbb I_\gamma&:=\frac{\gamma-\mu}{2}\|y-x^\star\|^2.
	\end{aligned}
	$$
	The $x$-term is
	$$
	\begin{aligned}
		\mathbb I_x
		&=
		-\dual{\nabla f(x)+\xi+Kx^\star,\ y-x}
		+
		\beta\dual{\nabla f(x)+\xi+Kx^\star,\ q}\\
		&=
		-\dual{\nabla f(x)+\xi+Kx^\star,\ y-x}
		+
		\beta\dual{K(x^\star-y),q}
		+
		\beta\|q\|_*^2.
	\end{aligned}
	$$
	Following the order of the $y$-equation, the $y$-term is
	$$
	\begin{aligned}
		\mathbb I_y
		&=
		-\mu\dual{y-x^\star,x-y}
		+
		\dual{y-x^\star,\nabla f(x)+\xi+Kx}
		+
		\dual{y-x^\star,K(y-x-\beta q)}\\
		&=
		-\mu\dual{y-x^\star,x-y}
		+
		\dual{y-x^\star,\nabla f(x)+\xi+Kx^\star}\\
		&\quad
		+
		\dual{y-x^\star,K(x-x^\star)}
		+
		\dual{y-x^\star,K(y-x-\beta q)}.
	\end{aligned}
	$$
	The gradient and subgradient terms involving $\nabla f(x)+\xi+Kx^\star$ combine as
	$$
	\dual{\nabla f(x)+\xi+Kx^\star,\ x-x^\star}	\ge
	D_{f+g}(x,x^\star)
	+
	\frac{\mu}{2}\|x-x^\star\|^2,
	$$
by the $\mu$-strong convexity of $f+g$ and the choice $\xi\in\partial g(x)$.
	The $K$-terms cancel by skew-symmetry:
$$
\begin{aligned}
	{}&	\dual{y-x^\star,K(y-x-\beta q)}
	+
	\dual{y-x^\star,K(x-x^\star)}
	+
	\beta\dual{K(x^\star-y),q}
	\\
	={}&	\dual{y-x^\star,K(y-x^\star-\beta q)}
	+
	\beta\dual{K(x^\star-y),q}
	\\
		={}&-\beta	\dual{y-x^\star, K q}
	+
	\beta\dual{K(x^\star-y),q}=0.
\end{aligned}
	$$
	Finally,
	$$
	-\mu\dual{y-x^\star,x-y}
	=
	-\frac{\mu}{2}\|x-x^\star\|^2
	+
	\frac{\mu}{2}\|y-x^\star\|^2
	+
	\frac{\mu}{2}\|y-x\|^2.
	$$
	Combining these identities gives
	$$
	\begin{aligned}
		-\dual{\partial E(\bs z;\xi),\mathcal G(\bs z;\xi)}
		&\ge
		D_{f+g}(x,x^\star)
		+
		\frac{\gamma}{2}\|y-x^\star\|^2
		+
		\beta\|q\|_*^2
		+
		\frac{\mu}{2}\|y-x\|^2\\
		&=
		E(\bs z)
		+
		\beta\|\nabla f(x)+\xi+Ky\|_*^2
		+
		\frac{\mu}{2}\|y-x\|^2.
	\end{aligned}
	$$
	This proves \eqref{eq:sL-ode}.
\end{proof}

Compared with the VOS flow in \eqref{eq:splitflow}, the same gradient--variable cancellation gives the Bregman term $D_{f+g}(x,x^\star)$. The additional $K$-terms cancel because $K$ is skew-symmetric and thus the strong Lyapunov structure is preserved.

\subsection{Gradient and skew-symmetric splitting methods}
We first consider a smooth monotone equation with skew-symmetric coupling:
\begin{equation}\label{eq:me}
	0=\mathcal L(x)=\nabla f(x)+Kx,
\end{equation}
where $f\in\mathcal S_{\mu,L}^1$ with $\mu>0$ and $K^{\intercal}=-K$. This case is already covered by the VOS framework in Section~\ref{sec:splitting} with 
$F=f_{-\mu}, N(y)=Ky+\mu y.$ 
The flow is
\begin{equation}\label{eq:adf-agss}
	x'=y-x,\qquad
	y'=x-y-\frac1\mu\bigl(\nabla f(x)+Ky\bigr).
\end{equation}
The Lyapunov function is
\begin{equation}\label{eq:vos-lyap1}
	 E(x,y)=D_{f_{-\mu}}(x,x^\star)+\frac{\mu}{2}\|y-x^\star\|^2.
\end{equation}
The strong Lyapunov property is established in Theorem~\ref{th:SLy1}.




\subsubsection{Implicit--explicit schemes}

If $K$ is treated implicitly, the algorithms and convergence analysis follow directly from \cref{sec:aor-vos}. Applying the backward AOR-VOS scheme \eqref{eq:AOR-VOS-backward} to \eqref{eq:adf-agss} gives
\begin{equation}\label{eq: intro IMEX}
	\left\{
	\begin{aligned}
		\frac{x_{k+1}-x_k}{\alpha}
		&=
		2y_{k+1}-y_k-x_{k+1},\\
		\frac{y_{k+1}-y_k}{\alpha}
		&=
		x_k-y_{k+1}
		-
		\frac1\mu\bigl(\nabla f(x_k)+Ky_{k+1}\bigr).
	\end{aligned}
	\right.
\end{equation}
This scheme is implicit only in $y_{k+1}$. Each iteration requires solving the shifted skew-symmetric system
\begin{equation}\label{eq:linearskeweqn}
	\bigl(\mu(1+\alpha)I+\alpha K\bigr)y_{k+1}
	=
	\mu y_k+\alpha\bigl(\mu x_k-\nabla f(x_k)\bigr).
\end{equation}
The accelerated convergence rate follows from Theorem~\ref{thm:convergence rate of AOR-VOS}.

The EPC-VOS scheme \eqref{eq: EPC-VOS} can be applied in the same way:
\begin{equation}\label{eq:epc-vos-exten}
	\left\{
	\begin{aligned}
		\frac{\tilde x_{k+1}-x_k}{\alpha}
		&=
		y_k-\tilde x_{k+1},\\
		\frac{y_{k+1}-y_k}{\alpha}
		&=
		\tilde x_{k+1}-y_{k+1}
		-
		\frac1\mu\bigl(\nabla f(\tilde x_{k+1})+Ky_{k+1}\bigr),\\
		\frac{x_{k+1}-x_k}{\alpha}
		&=
		y_{k+1}-x_{k+1}.
	\end{aligned}
	\right.
\end{equation}
Again, the only implicit step is the solve
\begin{equation}
	\bigl(\mu(1+\alpha)I+\alpha K\bigr)y_{k+1}
	=
	\mu y_k+\alpha\bigl(\mu\tilde x_{k+1}-\nabla f(\tilde x_{k+1})\bigr).
\end{equation}
The linear convergence rate follows from Theorem~\ref{thm:conv-ex0-ode-NAG}.

\subsubsection{Fully explicit schemes}
The implicit schemes above require applying $(cI+K)^{-1}$. In finite dimensions, we avoid this solve by splitting
$$
K=S_B-2B^{\intercal}, 
\quad
\text{ with }\
B={\rm tril}(K), \quad K = B - B^{\intercal},
\quad
S_B=B+B^{\intercal}.
$$
Applying AOR to the skew-symmetric term gives the explicit triangular scheme
\begin{equation}\label{eq:intro-agss}
	\left\{
	\begin{aligned}
		\frac{x_{k+1}-x_k}{\alpha}
		&=
		2y_{k+1}-y_k-x_{k+1},\\
		\frac{y_{k+1}-y_k}{\alpha}
		&=
		x_k-y_{k+1}
		-
		\frac1\mu\bigl(\nabla f(x_k)+2By_{k+1}-S_By_k\bigr).
	\end{aligned}
	\right.
\end{equation}
Each iteration now solves the triangular system
\begin{equation}
	\bigl(\mu(1+\alpha)I+2\alpha B\bigr)y_{k+1}
	=
	\mu(y_k+\alpha x_k)-\alpha\nabla f(x_k)+\alpha S_By_k,
\end{equation}
which can be computed by forward substitution.

We use the following bound for the additional cross terms. It follows from Lemma~\ref{lem:cross Breg} and the estimate $\|w\|_{S_B}^2\leq \|S_B\|\|w\|^2.$

\begin{lemma}\label{lem:cross Breg AGSS-ex}
	Suppose $f\in\mathcal S_{\mu,L}^1$ with $0<\mu<L$. For any $(x,y)$ and $(\hat x,\hat y)$, and any $\delta\in(0,1)$,
	\begin{align*}
		&\left|
		\dual{\nabla f_{-\mu}(\hat x)-\nabla f_{-\mu}(x),\hat y-y}
		-
		\frac12\|\hat y-y\|_{S_B}^2
		\right| \\
		&\quad\le
		\sqrt{\frac{L-\mu}{\delta\mu}}
		\left(
		D_{f_{-\mu}}(x,\hat x)
		+
		\frac{\delta\mu}{2}\|\hat y-y\|^2
		\right)
		+
		\frac{\|S_B\|}{2}\|\hat y-y\|^2.
	\end{align*}
\end{lemma}

\begin{theorem}\label{th:AGSS-ex}
	Suppose $f\in\mathcal S_{\mu,L}^1$ with $0<\mu<L$. Assume first that $K\ne0$. Let $\delta\in(0,1)$ and choose
	\begin{equation}\label{eq:ak}
		\alpha
		=
		\min\left\{
		\sqrt{\frac{\delta\mu}{L-\mu}},
		\frac{(1-\delta)\mu}{\|S_B\|}
		\right\},
		\qquad
		\rho:=\frac1{1+\alpha}.
	\end{equation}
	Then the scheme \eqref{eq:intro-agss} satisfies
	$$
	 \mathcal E^\alpha(x_{k+1},y_{k+1})
	\le
	\rho \mathcal E^\alpha(x_k,y_k),
	$$
	where
	$$
	 \mathcal E^\alpha(x,y)
	:=
	D_{f_{-\mu}}(x,x^\star)
	+
	\frac12\|y-x^\star\|_{\mu I+\alpha S_B}^2
	-
	\alpha\dual{\nabla f_{-\mu}(x)-\nabla f_{-\mu}(x^\star),y-x^\star}.
	$$
	Consequently, for all $k\ge0$,
	$$
	 \mathcal E^\alpha(x_{k+1},y_{k+1})
	\le
	\rho^{k+1} \mathcal E^\alpha(x_0,y_0),
	\qquad
	 E(x_{k+1},y_{k+1})
	\le
	\alpha^{-1}\rho^k \mathcal E^\alpha(x_0,y_0),
	$$
	where $ E$ is defined in \eqref{eq:vos-lyap1}. The largest admissible step size is
	\begin{equation}\label{eq:max-ak}
		\alpha^\star
		=
		\max_{\delta\in(0,1)}
		\min\left\{
		\sqrt{\frac{\delta\mu}{L-\mu}},
		\frac{(1-\delta)\mu}{\|S_B\|}
		\right\}
		=
		\frac{2\mu}{\|S_B\|+\sqrt{\|S_B\|^2+4\mu(L-\mu)}}.
	\end{equation}
If $K=0$, take $\alpha=\sqrt{\mu/(L-\mu)}$; the same estimates hold.
\end{theorem}
\begin{proof}
	The proof follows the argument of Theorem~\ref{thm:convergence rate of AOR-VOS}. The only new term comes from the AOR splitting $K=2B-S_B$:
	$$
	\begin{aligned}
		&-\alpha (y_{k+1}-x^\star,y_{k+1}-y_k)_{S_B}\\
		={}&
		-\frac{\alpha}{2}
		\left(
		\|y_{k+1}-x^\star\|_{S_B}^2
		+
		\|y_{k+1}-y_k\|_{S_B}^2
		-
		\|y_k-x^\star\|_{S_B}^2
		\right).
	\end{aligned}
	$$
	This term is included in the modified energy $ \mathcal E^\alpha$. Lemma~\ref{lem:cross Breg AGSS-ex} controls the cross term involving $\nabla f_{-\mu}$ by the fraction $\delta$ of the $\mu$-term, requiring $\alpha\le\sqrt{\delta\mu/(L-\mu)}$. The remaining fraction $1-\delta$ controls the $S_B$ term and ensures $\mu I+\alpha S_B\succeq0$, requiring $\alpha\le(1-\delta)\mu/\|S_B\|$. Hence \eqref{eq:ak} gives
	$$
	 \mathcal E^\alpha(x_{k+1},y_{k+1})
	\le
\frac1{1+\alpha}\, \mathcal E^\alpha(x_k,y_k).
	$$
	Iteration gives the stated linear rate. Optimizing over $\delta\in(0,1)$ gives \eqref{eq:max-ak}. When $K=0$, the scheme reduces to backward AOR-VOS, so Theorem~\ref{thm:convergence rate of AOR-VOS} gives the stated estimates.
\end{proof}

\begin{remark}\rm
	The fully explicit scheme avoids the shifted skew-symmetric solve in \eqref{eq:linearskeweqn}, but its step size depends on the ratio $\mu/\|S_B\|$. If $\mu/\|S_B\|\ll\sqrt{\mu/L}$ and evaluating $\nabla f$ is more costly than solving the linear system, the implicit--explicit scheme \eqref{eq: intro IMEX} may be preferable. In that case, \eqref{eq:linearskeweqn} can also be solved inexactly with a decreasing inner tolerance.
\end{remark}

\subsubsection{Policy evaluation and MSPBE}

We evaluate the methods on policy evaluation problems from reinforcement learning~\cite{du2017stochastic}. With a linear value function approximation and $\ell_2$ regularization, minimizing the mean squared projected Bellman error (MSPBE) is equivalent to the weighted least-squares problem~\cite{dann2014policy}
\begin{equation}\label{eq: MSPBE_orig}
    \min_{u\in\mathbb R^n}\operatorname{MSPBE}(u)
    =
    \frac12\|Bu-b\|_{A^{-1}}^2
    +
    \frac{\lambda}{2}\|u\|^2,
\end{equation}
where $A\in\mathbb R^{m\times m}$ is symmetric positive definite, $B\in\mathbb R^{m\times n}$, and $\lambda>0$. Introducing $p$ by $Ap=Bu-b$ gives the equivalent saddle-point problem
\begin{equation}
    \min_{u\in\mathbb R^n}\max_{p\in\mathbb R^m}
    \left\{
    \frac{\lambda}{2}\|u\|^2
    -
    \frac12\|p\|_A^2
    -
    \langle b,p\rangle
    +
    \langle Bu,p\rangle
    \right\}.
\end{equation}
This formulation avoids forming $A^{-1}$. Finding a saddle point $(u^\star,p^\star)$ is equivalent to solving \eqref{eq:me} with
$$
f(x)
=
\frac{\lambda}{2}\|u\|^2
+
\frac12\|p\|_A^2
+
\langle b,p\rangle,
\quad
K=
\begin{pmatrix}
0&B^{\intercal}\\
-B&0
\end{pmatrix},
\quad
x=
\begin{pmatrix}
u\\
p
\end{pmatrix}.
$$

We first use synthetic data. For $m=2500$ and $n=50$, we generate Gaussian matrices $B$ and $C$ with independent standard normal entries. We set $\lambda = 1$ and $A=C+C^{\intercal}$, shift it to be positive definite, and rescale it so that its smallest eigenvalue is 1. In this case, $\mu = 1$ and the condition number is $\kappa=L=\|A\|$. 

We apply AOR-VOS \eqref{eq:intro-agss} and EPC-VOS-IMEX \eqref{eq:epc-vos-exten}, and compare them with AG-OG~\cite{li2023nesterov}, APDG~\cite{kovalev2022accelerated}, and LPD~\cite{thekumparampil2022lifted}. Figure~\ref{fig:MSPBE_results}(a) shows the relative error versus iteration. Both VOS schemes show linear decay and are faster than the tested accelerated baselines in this example. EPC-VOS-IMEX is fastest. Its implicit step allows a larger step size, and the resulting linear solve reduces to inverting $\beta I_n+B^{\intercal}B$, which is cheap because $n=50$.

\begin{figure}[htp]
\centering
\begin{minipage}{0.47\textwidth}
\centering
\includegraphics[width=0.725\linewidth]{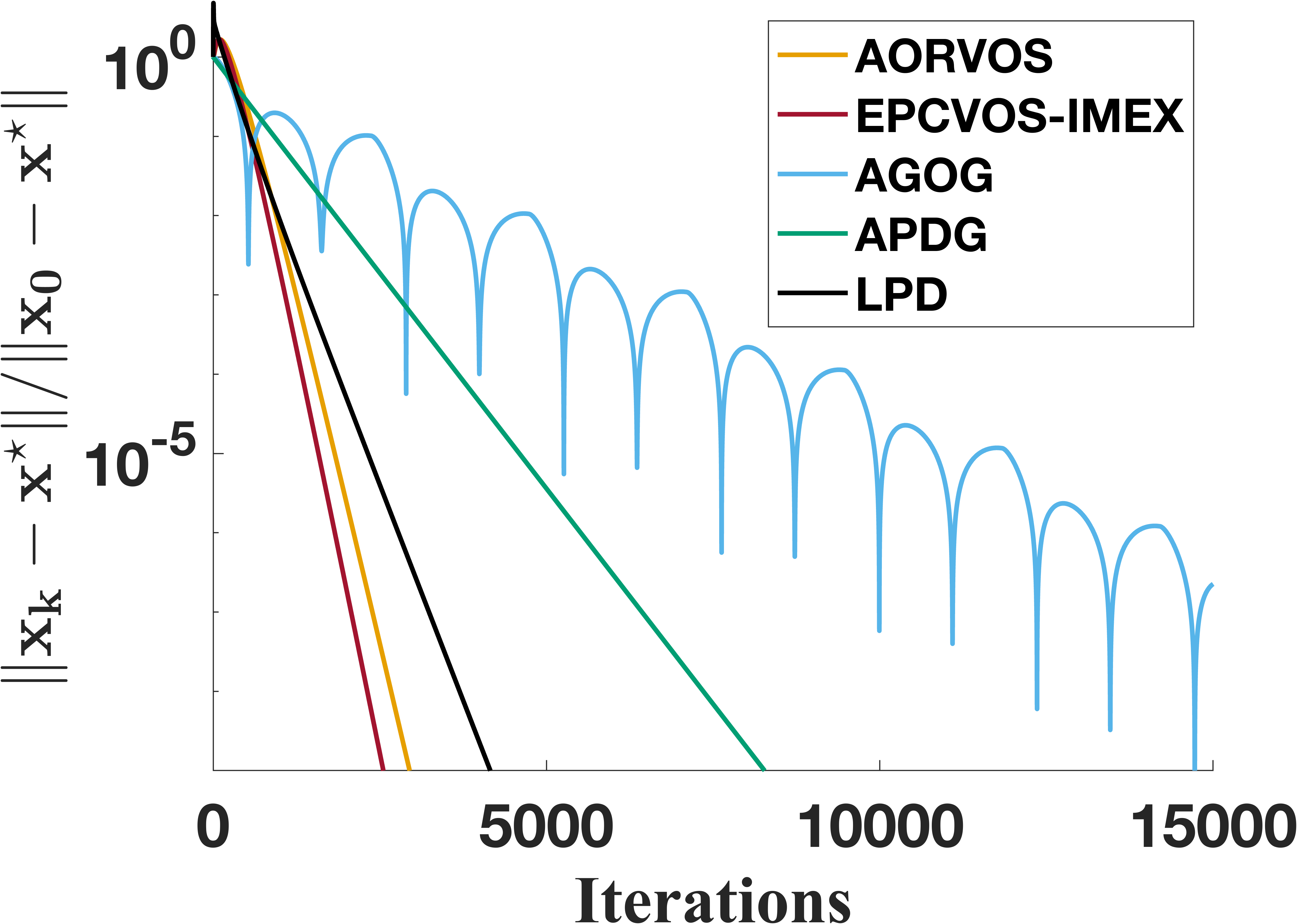}\\
{\small (a) Synthetic data}
\end{minipage}
\hfill
\begin{minipage}{0.47\textwidth}
\centering
\includegraphics[width=0.7\linewidth]{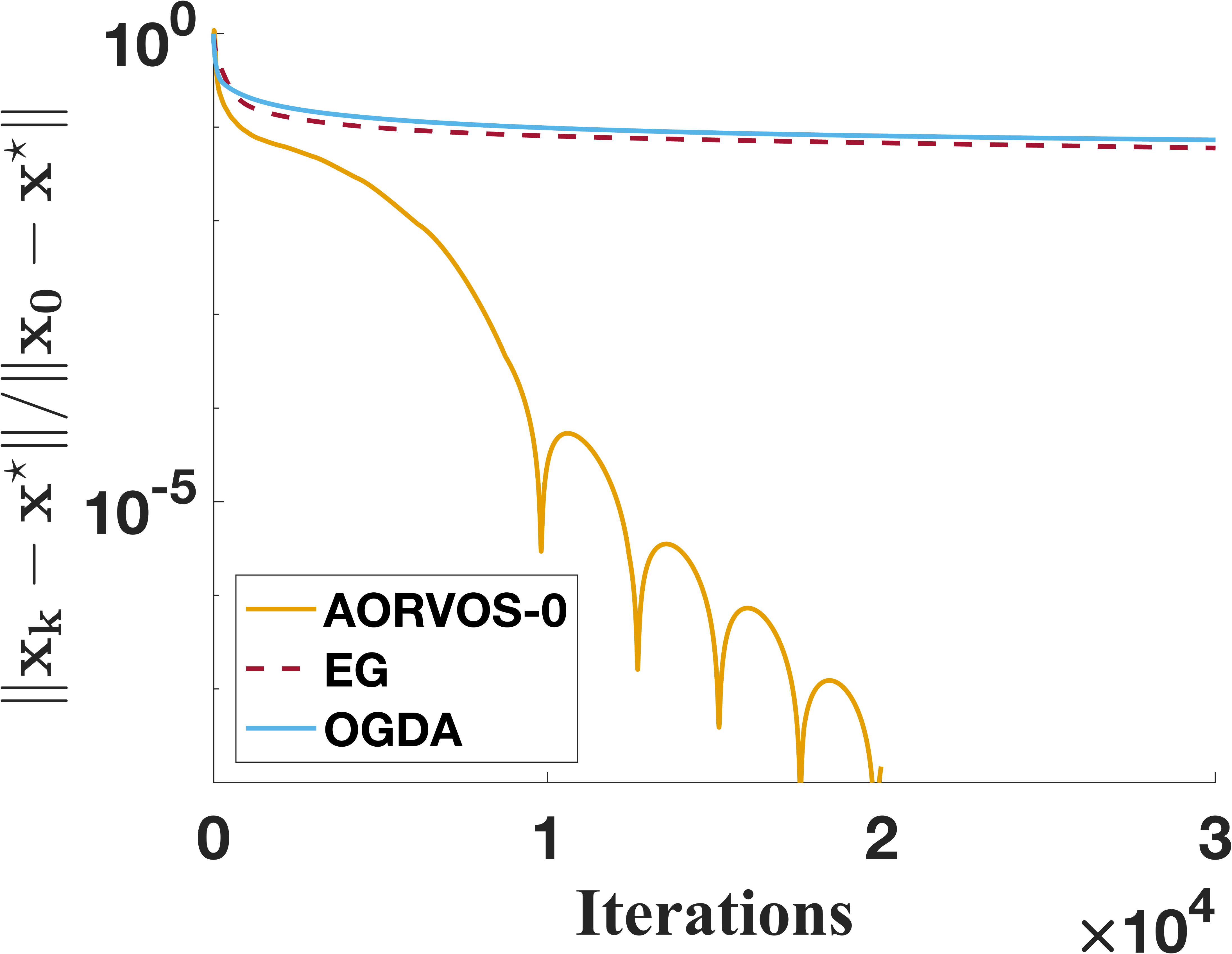}\\
{\small (b) Random MDP}
\end{minipage}
\caption{MSPBE policy evaluation tests. (a) Relative error versus iteration for synthetic data. (b) Relative error versus iteration for a random MDP.}
\label{fig:MSPBE_results}
\end{figure}

For a Markov decision process under policy $\pi$ with feature map $\phi:\mathcal S\to\mathbb R^n$, the matrices in \eqref{eq: MSPBE_orig} are
$$
A=\mathbb E[\phi(s_t)\phi(s_t)^{\intercal}],
\quad
B=\mathbb E[\phi(s_t)(\phi(s_t)-\gamma_{\rm dis}\phi(s_{t+1}))^{\intercal}],
\quad
b=\mathbb E[r_t\phi(s_t)].
$$
Here $\gamma_{\rm dis}\in[0,1)$ is the discount factor, and the expectations are taken with respect to the stationary distribution. We test a random MDP benchmark problem with $400$ states, $10$ actions, and $201$ features \cite[Section 3.1.1]{dann2014policy}. The transition probabilities, rewards, and features are sampled from uniform distributions, and $\gamma_{\rm dis}=0.95$.

Since $\mu$ is difficult to estimate in this setting, we use AOR-VOS with homotopy perturbation, denoted by AOR-VOS-0, with
$
\epsilon_0=0.01\max\{L,\|B\|^2\},
m_0=60.
$
We compare AOR-VOS-0 with the extragradient (EG) method~\cite{korpelevich1976extragradient} and optimistic gradient descent ascent (OGDA)~\cite{popov1980modification}. These baselines do not require $\mu$ to choose their step sizes. For strongly convex--strongly concave problems, their linear convergence factor is $1-\Theta(1/\kappa)$, where $\kappa$ is the condition number of the full saddle-point operator~\cite{mokhtari2020unified}. Figure~\ref{fig:MSPBE_results}(b) shows that the homotopy VOS method gives faster decay in this test.

\subsubsection{Comparison with HSS for linear systems}

Consider the linear algebraic system
\begin{equation}\label{eq:Ax-b}
	\mathcal A x=b,
	\qquad
	\mathcal A=\mathcal A^{\rm s}+K,
\end{equation}
where $\mathcal A^{\rm s}:=(\mathcal A+\mathcal A^{\intercal})/2$ and $K:=(\mathcal A-\mathcal A^{\intercal})/2$. Then $K^{\intercal}=-K$. Assume that $\mathcal A^{\rm s}$ is symmetric positive definite. Equation \eqref{eq:Ax-b} is the monotone equation \eqref{eq:me} with
$$
f(x)=\frac12\langle\mathcal A^{\rm s}x,x\rangle-\langle b,x\rangle.
$$
The condition $f\in\mathcal S_{\mu,L}^1$ is equivalent to $\lambda_{\min}(\mathcal A^{\rm s})\ge\mu$ and $\lambda_{\max}(\mathcal A^{\rm s})\le L$.

Bai, Golub, and Ng~\cite{bai2003hermitian} proposed the Hermitian/skew-Hermitian splitting (HSS) method
\begin{equation}\label{eq:HSS}
	\left\{
	\begin{aligned}
		(\alpha I+\mathcal A^{\rm s})x_{k+\frac12}
		&=
		(\alpha I-K)x_k+b,\\
		(\alpha I+K)x_{k+1}
		&=
		(\alpha I-\mathcal A^{\rm s})x_{k+\frac12}+b.
	\end{aligned}
	\right.
\end{equation}
HSS alternates between the symmetric and skew-symmetric parts. Hence each step requires solvers for $\alpha I+\mathcal A^{\rm s}$ and $\alpha I+K$. Taking $\mu$ and $L$ as the extreme eigenvalues of $\mathcal A^{\rm s}$, the choice $\alpha=\sqrt{\mu L}$ minimizes the spectral-radius bound for HSS, giving
$$
\frac{\sqrt{\kappa(\mathcal A^{\rm s})}-1}{\sqrt{\kappa(\mathcal A^{\rm s})}+1}.
$$
Several variants were later developed~\cite{bai2007accelerated,bai2007successive}. In contrast, the IMEX scheme \eqref{eq: intro IMEX} has contraction factor $(1+\sqrt{\mu/L})^{-1}=(1+\kappa(\mathcal A^{\rm s})^{-1/2})^{-1}$ in the VOS estimate. The two bounds have the same accelerated scaling in $\kappa(\mathcal A^{\rm s})$, but concern different error measures. The IMEX scheme avoids the solve with $\alpha I+\mathcal A^{\rm s}$.

We test this difference on the convection--diffusion equation
\begin{equation}\label{eq: convection-diffusion eq}
	-\Delta u+\bm b\cdot\nabla u=w
\end{equation}
on $\Omega=[0,1]^2$, with Dirichlet boundary conditions and constant convection $\bm b=(10,10)^{\intercal}$. We use linear Lagrange finite elements on a uniform triangulation with mesh size $h$, implemented in the $i$FEM package~\cite{chen2008ifem}. The resulting linear system is solved by the explicit scheme \eqref{eq:intro-agss}, the IMEX scheme \eqref{eq: intro IMEX}, and HSS \eqref{eq:HSS}. All methods start from the zero initial guess and stop when the relative residual in the $L^\infty$-norm is below $10^{-10}$.

\begin{figure}[htp]
\begin{minipage}{0.47\textwidth}
\centering
\includegraphics[width=0.75\linewidth]{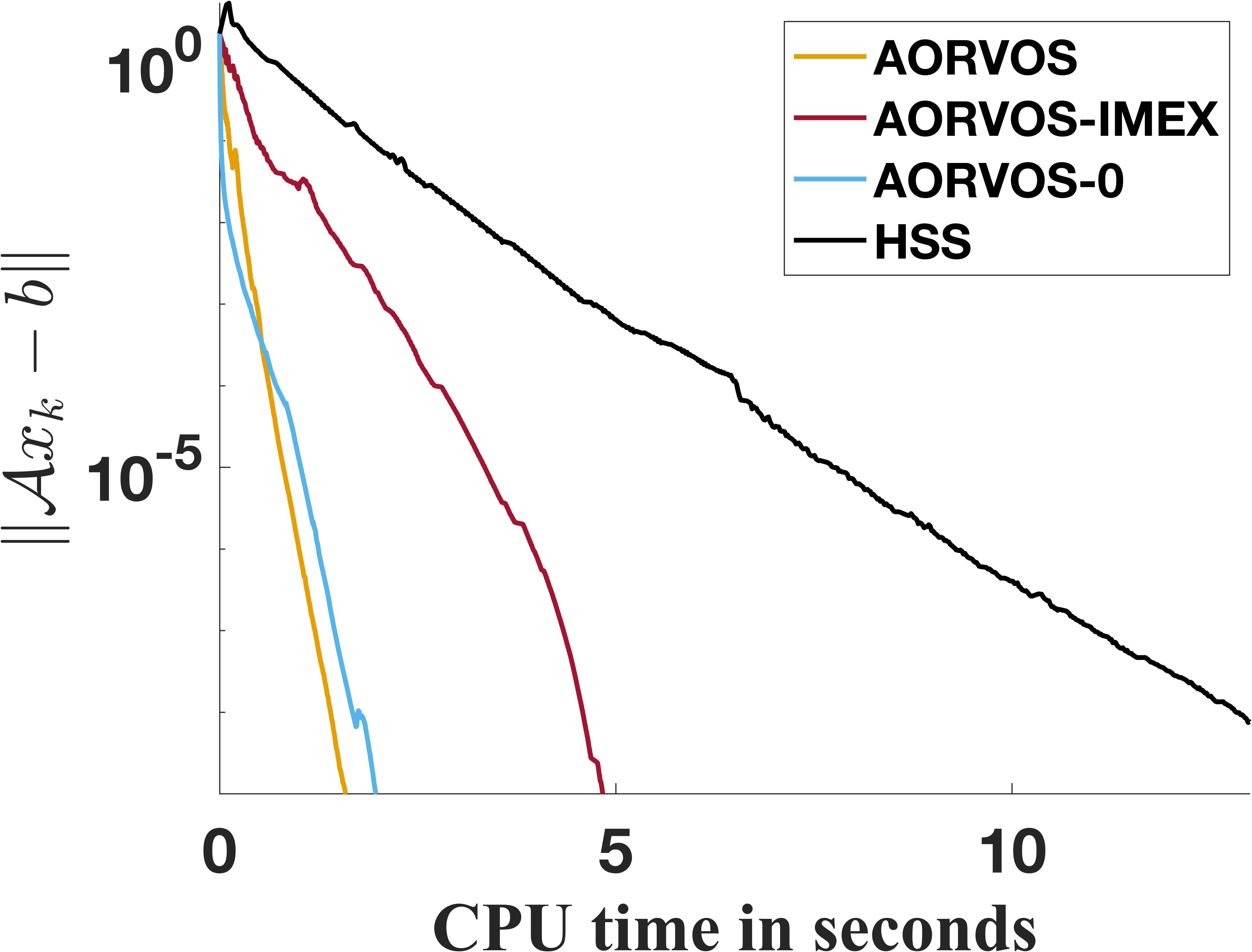}
\end{minipage}
\hfill
\begin{minipage}{0.47\textwidth}
\centering
\includegraphics[width=0.75\linewidth]{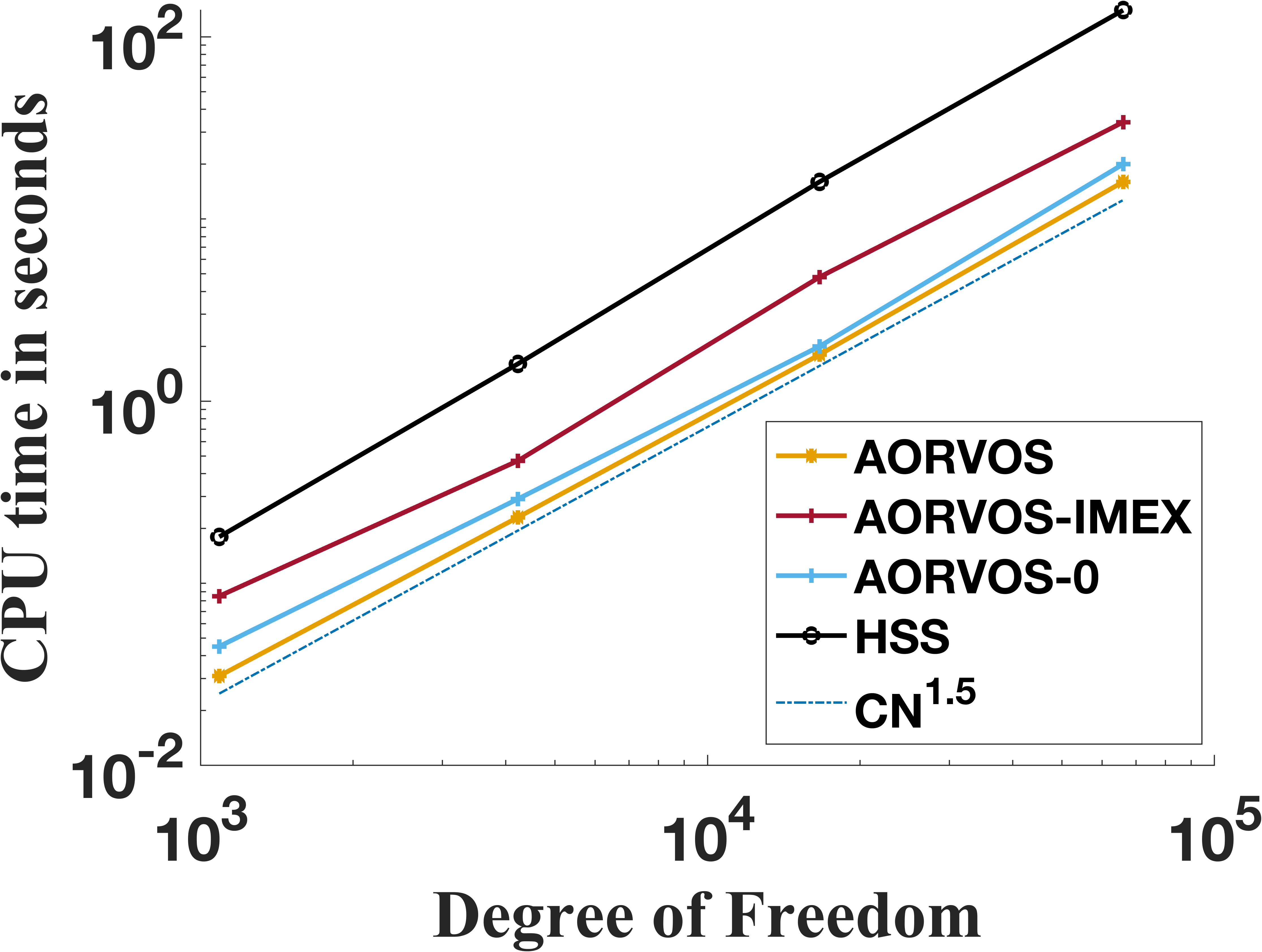}
\end{minipage}
\caption{Convection--diffusion problem \eqref{eq: convection-diffusion eq}. Left: relative residual versus CPU time for $h=1/128$. Right: CPU time versus degrees of freedom.}
\label{fig:convection_diffusion_result}
\end{figure}

Figure~\ref{fig:convection_diffusion_result} shows the decay of the relative residual and the CPU-time growth. For HSS and the IMEX scheme, the shifted skew-symmetric systems are solved by BiCGSTAB. The inner iteration stops when the relative residual is below $10^{-10}$ or after $20$ iterations.

The explicit VOS scheme avoids matrix inversions and retains the expected accelerated scaling in this example. For this problem, $\sqrt{L/\mu}=\mathcal O(1/h)=\mathcal O(N^{1/2})$, where $N$ is the number of degrees of freedom. If each iteration costs $O(N)$ and the iteration count is $O(N^{1/2})$, the resulting model predicts $O(N^{3/2})$ work. The observed CPU time therefore scales like $\mathcal O(N^{3/2})$. Preconditioning the symmetric part could further reduce the effective condition number.

We also test the homotopy perturbation strategy from Algorithm~\ref{alg: homotopy perturbed VOS} with the explicit scheme \eqref{eq:intro-agss}, using $\epsilon_0=0.1L$ and $m_0=60$. This variant gives comparable performance without estimating $\mu$, and thus avoids computing the smallest eigenvalue of the system matrix.

\subsection{Accelerated methods for general monotone inclusions}
\label{sec:agss-wo-aor}

We now return to the monotone inclusion problem \eqref{eq:mi}. The nonsmooth term $g$ is handled by a proximal step, while $\gamma_k$ is updated dynamically. This leads to the Accelerated Proximal-Gradient and Skew-Symmetric Splitting (APGSS) method below.

\begin{equation}\label{eq:semi-ode}
	\left\{
	\begin{aligned}
		\frac{x_{k+1}-x_k}{\alpha_k}
		&\in
		y_k-x_{k+1}
		-
		\beta_k\bigl(\nabla f(x_k)+\partial g(x_{k+1})+Ky_k\bigr),\\
		\gamma_k\frac{y_{k+1}-y_k}{\alpha_k}
		&\in
		\mu(x_{k+1}-y_{k+1})
		-
		\nabla f(x_{k+1})
		-
		\partial g(x_{k+1})
		-
		K\bar x_{k+1},
	\end{aligned}
	\right.
\end{equation}
where
$$
\bar x_{k+1}:=x_{k+1}+\frac{x_{k+1}-x_k}{\alpha_k}.
$$
The parameter $\gamma_k$ is updated implicitly:
\begin{equation}\label{eq:gamak}
	\frac{\gamma_{k+1}-\gamma_k}{\alpha_k}
	=
	\mu-\gamma_{k+1}.
\end{equation}

The first inclusion in \eqref{eq:semi-ode} is a proximal step. More precisely,
$$
x_{k+1}
=
\operatorname{prox}_{\eta_k g}(\tilde{x}_k),
\quad
\text{ with }\
\eta_k
=
\frac{\alpha_k\beta_k}{1+\alpha_k}, \quad \tilde{x}_k
=
\frac{x_k+\alpha_k y_k}{1+\alpha_k}
-
\eta_k\bigl(\nabla f(x_k)+Ky_k\bigr).
$$
Thus
$$
\xi_{k+1}:=\frac{\tilde{x}_k-x_{k+1}}{\eta_k}\in\partial g(x_{k+1}).
$$
With this choice, the second inclusion in \eqref{eq:semi-ode} gives
\begin{equation}
	y_{k+1}
	=
	\frac{\alpha_k}{\gamma_k+\mu\alpha_k}
	\left(
	\frac{\gamma_k}{\alpha_k}y_k
	+
	\mu x_{k+1}
	-
	\nabla f(x_{k+1})
	-
	\xi_{k+1}
	-
	K\bar x_{k+1}
	\right).
\end{equation}

\subsubsection{Convergence analysis}

The convergence estimate below follows from the strong Lyapunov property in Theorem~\ref{thm:strong-lyap-extension}. 

\begin{theorem}\label{thm:conv-semi-ode}
Assume $f$ is convex and $L$-smooth with $L>0$, $g$ is proper, closed, and convex, and $f+g$ is $\mu$-strongly convex with $\mu\ge0$. Let $\bs z_k=(x_k,y_k,\gamma_k)$ be generated by \eqref{eq:semi-ode}--\eqref{eq:gamak} with $x_0\in\dom g,\,\gamma_0>0$ and 
	\begin{equation}\label{eq:bk-ak}
		\beta_k
		=
		\frac{\alpha_k}{2(1-\sigma_k)\gamma_k\delta_k},
		\quad
		\alpha_k
		=
		\frac{2\gamma_k\sqrt{\delta_k(1-\sigma_k)(1-\delta_k)}}{\sqrt{L\gamma_k+\|K\|^2/\sigma_k}},
	\end{equation}
	where $\sigma_k,\delta_k\in(0,1)$ are arbitrary. Then
	\begin{equation}\label{eq:diff-Ek}
		E(\bs z_{k+1})-E(\bs z_k)
		\le
		-\alpha_k E(\bs z_{k+1}),
	\end{equation}
	and hence
	$$
	E(\bs z_k)
	\le
	E(\bs z_0)
	\prod_{i=0}^{k-1}\frac1{1+\alpha_i}.
	$$
	Moreover, if $\gamma_0\ge\mu$ and
	\begin{equation}\label{eq:bk-ak-opt}
		\beta_k
		=
		\frac{1}{\sqrt{\|K\|^2+L\gamma_k}},
		\quad
		\alpha_k
		=
		\frac{\gamma_k}{\|K\|+\sqrt{\|K\|^2+L\gamma_k}},
	\end{equation}
	then for $k\ge1$
	\begin{equation}\label{eq:opt-rate}
		E(\bs z_k)
		\le
		E(\bs z_0)
		\min\left\{
		(1+\alpha_{\min})^{-k},
		\frac{32L\theta^2}{(\sqrt{\gamma_0}k+4\sqrt{2L}\theta)^2}
		+
		\frac{4\theta\|K\|}{\gamma_0k+4\theta\|K\|}
		\right\},
	\end{equation}
	where
	$$
	\alpha_{\min}:=\frac{\mu}{\|K\|+\sqrt{L\mu+\|K\|^2}},
	\quad
	\alpha_{\max}:=\frac{\gamma_0}{\|K\|+\sqrt{L\gamma_0+\|K\|^2}},
	\quad
	\theta:=1+\alpha_{\max}.
	$$
\end{theorem}

\begin{proof}
For the proof, we write the discrete vector field in triangular form. For any $\xi\in\partial g(x)$, set
\begin{equation}
	\mathbb T_k
	:=
	\begin{pmatrix}
		I&0\\
		-K/\gamma_k&I/\gamma_k
	\end{pmatrix},
	\quad
	\mathbb G_k(x,y;\xi)
	:=
	\begin{pmatrix}
		y-x-\beta_k(\nabla f(x)+\xi+Ky)\\
		\mu(x-y)-\nabla f(x)-\xi-Kx
	\end{pmatrix}.
\end{equation}
Here $\nabla f(x)+\xi+Ky$ is the selected element of $\mathcal L(x)+K(y-x)$. The first two components of the discrete vector field are written as $\mathbb T_k\mathbb G_k(x,y;\xi)$.

Let
$$
q_{k+1}:=\nabla f(x_{k+1})+\xi_{k+1}+Ky_{k+1}.
$$
The scheme can be written as a perturbed implicit Euler step for the flow \eqref{eq:ode}:
$$
\begin{pmatrix}
	\frac{x_{k+1}-x_k}{\alpha_k}\\
	\frac{y_{k+1}-y_k}{\alpha_k}
\end{pmatrix}
=
\mathbb T_k
\left(
\mathbb G_k(x_{k+1},y_{k+1};\xi_{k+1})
+
\begin{pmatrix}
	\Delta_k\\
	0
\end{pmatrix}
\right),
$$
where
$$
\Delta_k
:=
y_k-y_{k+1}
+
\beta_k\bigl(\nabla f(x_{k+1})-\nabla f(x_k)\bigr)
+
\beta_kK(y_{k+1}-y_k).
$$	Combining this identity with the implicit update \eqref{eq:gamak}, the convexity estimate for $f+g$ (see~\cref{lem:dual-bregman-bounds}), and \eqref{eq:sL-ode}, we obtain
	$$
	\begin{aligned}
		E(\bs z_{k+1})-E(\bs z_k)
		&\le
		-\alpha_k E(\bs z_{k+1})
		-\alpha_k\beta_k\|q_{k+1}\|_*^2
		-\frac{\gamma_k}{2}\|y_{k+1}-y_k\|^2\\
		&\quad
		-\frac1{2L}\|\nabla f(x_{k+1})-\nabla f(x_k)\|_*^2
		+
		\alpha_k\dual{q_{k+1},\Delta_k}.
	\end{aligned}
	$$
	It remains to control the last term. By Young's inequality, for any $\sigma_k\in(0,1)$,
	$$
	\begin{aligned}
		\alpha_k\dual{q_{k+1},\Delta_k}
		&\le
		\frac1{2L}\|\nabla f(x_{k+1})-\nabla f(x_k)\|_*^2
		+
		\frac{\gamma_k}{2}\|y_{k+1}-y_k\|^2\\
		&\quad
		+
		\left(
		\frac{\alpha_k}{2(1-\sigma_k)\gamma_k\beta_k}
		+
		\frac{L\alpha_k\beta_k}{2}
		+
		\frac{\|K\|^2\alpha_k\beta_k}{2\sigma_k\gamma_k}
		\right)
		\alpha_k\beta_k\|q_{k+1}\|_*^2.
	\end{aligned}
	$$
	The choice \eqref{eq:bk-ak} makes the coefficient in parentheses equal to $1$. Therefore all consistency terms are absorbed, and \eqref{eq:diff-Ek} follows. For $K\ne0$, the optimized parameters \eqref{eq:bk-ak-opt} follow by taking $\delta_k=1/2$ and $\sigma_k=\|K\|/(\|K\|+\sqrt{\|K\|^2+L\gamma_k})$. For $K=0$, omit the skew term in Young's inequality. The remaining coefficient is $\alpha_k/(2\gamma_k\beta_k)+L\alpha_k\beta_k/2=1$ for $\alpha_k=\sqrt{\gamma_k/L}$ and $\beta_k=1/\sqrt{L\gamma_k}$, so \eqref{eq:diff-Ek} also holds.

	It remains to estimate the product. With \eqref{eq:bk-ak-opt}, define
	$$
	\Theta_0=1,\quad	\Theta_k:=\prod_{i=0}^{k-1}\frac1{1+\alpha_i}.
	$$
	Since $\gamma_0\ge\mu$, the implicit update \eqref{eq:gamak} gives $\gamma_k\ge\mu$, and hence $\alpha_k\ge\alpha_{\min}$. Thus we have $	\Theta_k\le(1+\alpha_{\min})^{-k}$.
	On the other hand, $\gamma_k\ge\gamma_0\Theta_k$, and therefore
	$$
	\alpha_k
	\ge
	\frac{\gamma_0\Theta_k}{\|K\|+\sqrt{\|K\|^2+L\gamma_0\Theta_k}}
	\ge
	\frac{\gamma_0\Theta_k}{\sqrt{4\|K\|^2+2L\gamma_0\Theta_k}}.
	$$
We use the interpolation argument of~\cite[Appendix C]{Luo2025j} to obtain the rate in \eqref{eq:opt-rate}. 

\end{proof}


\subsubsection{Examples}

For the unconstrained smooth problem \eqref{eq:minf}, the inclusion \eqref{eq:mi} reduces to $g=0$ and $K=0$. The flow \eqref{eq:ode} becomes
$$
x'=y-x-\beta\nabla f(x),
\qquad
\gamma y'=\mu(x-y)-\nabla f(x),
$$
which is the HNAG flow~\cite{chen2019orderoptimizationmethodsbased}; see also \eqref{eq:flow}. The scheme \eqref{eq:semi-ode} reduces to 
\begin{align*}
	x_{k+1} &= \frac{x_k + \alpha_k y_k - \alpha_k \beta_k \nabla f(x_k)}{1 + \alpha_k}, \quad
	y_{k+1}  = \frac{\gamma_k y_k + \mu \alpha_k x_{k+1} - \alpha_k \nabla f(x_{k+1})}{\gamma_k + \mu\alpha_k }.
\end{align*}
This recovers \cite[Eq.~(42)]{chen2019orderoptimizationmethodsbased}. For $\gamma_0\ge\mu$ with $\gamma_0>0$, Theorem~\ref{thm:conv-semi-ode} with $\alpha_k=\sqrt{\gamma_k/L}$ and $\beta_k=\alpha_k/\gamma_k$ gives, for $k\ge1$,
\begin{equation}\label{eq:opt-rate-hnag}
\begin{aligned}
	f(x_k)-f(x^\star)\leq {}&	\left[f(x_0)-f(x^\star)+\frac{\gamma_0}{2}\nm{y_0-x^\star}^2\right]\\{}&\quad \times \min\left\{
	\left(1+\sqrt{\frac{\mu}{L}}\right)^{-k},
	\frac{32(\sqrt L+\sqrt{\gamma_0})^2}{\gamma_0 k^2}
	\right\}.
\end{aligned}
\end{equation}
This matches the optimal complexity of first-order methods for smooth convex minimization in the standard first-order oracle model~\cite{nesterov2018lectures,NemirovskyYudin1983}.

For the composite problem \eqref{eq:min-f-g}, the flow \eqref{eq:ode} gives the nonsmooth HNAG inclusion
$$
x'\in y-x-\beta(\nabla f(x)+\partial g(x)),
\qquad
\gamma y'\in \mu(x-y)-\nabla f(x)-\partial g(x).
$$
The corresponding scheme \eqref{eq:semi-ode} is closely related to \cite[Eq.~(65)]{chen2019orderoptimizationmethodsbased}. The rate from Theorem~\ref{thm:conv-semi-ode} is similar to that in \cref{eq:opt-rate-hnag} and has the same optimal order in the standard first-order oracle model.

For the saddle-point problem \eqref{eq:saddleflow}, set
$$
\mathcal L(x)=\nabla f(x)+Kx,
\quad
f(x)=f^u(u)+f^p(p),
\quad
K=
\begin{pmatrix}
O&B^\intercal\\
-B&O
\end{pmatrix},
\quad
x=(u,p)^\intercal.
$$
Here $K$ is skew-symmetric. Applying the APGSS method \eqref{eq:semi-ode} gives an accelerated primal--dual scheme for \eqref{eq:saddleflow}. Since $\|K\|=\|B\|$, Theorem~\ref{thm:conv-semi-ode} gives, for $\gamma_0\ge\mu >0$,
\begin{equation*}
\begin{aligned}
D_f(x_k,x^\star)
&\le\left[D_f(x_0,x^\star)+\frac{\gamma_0}{2}\|y_0-x^\star\|^2\right] \times\left(1+\frac{\mu}{\|B\|+\sqrt{\|B\|^2+L\mu}}\right)^{-k}.
\end{aligned}
\end{equation*}
The mixed sublinear bound is \eqref{eq:opt-rate} with $\|K\|=\|B\|$. Its parameter dependence agrees with the lower-bound scaling in~\cite{Ouyang2021,zhang_lower_2022}.

For composite saddle-point problems with a convex, $L$-smooth part, $L>0$, and a proper, closed, convex nonsmooth part, the proximal $x$ update in \eqref{eq:semi-ode} gives the corresponding APGSS scheme. The energy estimate in \eqref{eq:opt-rate} applies under the assumptions of Theorem~\ref{thm:conv-semi-ode}.

Related accelerated rates are established by Chen, Lan, and Ouyang~\cite{ChenLanOuyang2014}, inertial corrected primal--dual splitting~\cite{AIValkonen2020}, and accelerated Condat--V{\~u}~\cite{DriggsEhrhardtSchonliebTang2024}. APGSS uses one proximal step and explicit gradient and skew evaluations, with the same subgradient selected in both updates. Its analysis uses the energy \eqref{eq:Lyap} without cross terms. This simplifies the proof; comparisons of rate constants also require the same error measure and oracle cost.
\begin{remark}[Acceleration for optimization with affine equality constraints]
    Consider smooth convex optimization with affine equality constraints:
    \begin{equation}\label{eq:contrinedopt}
        \min_{u \in V} f^u(u), \text{ s.t. } Bu = b,
    \end{equation}
    where $f^u$ is convex and $L$-smooth, and $B:V\to Q$ is a surjective linear operator. Introducing a Lagrange multiplier $p \in Q^*$, the KKT system of \eqref{eq:contrinedopt} can be rewritten in the form of \eqref{eq:mi} with 
    $$
    \mathcal L(x) =\nabla f(x)+Kx,\quad
f(x)=f^u(u)+\langle b,p \rangle,
\quad
K=
\begin{pmatrix}
	O&B^\intercal\\
	-B&O
\end{pmatrix},
\quad
x=(u,p)^\intercal.
    $$
In this case, APGSS yields the sublinear convergence rate $\mathcal O(1/k)$ since $\mu = 0$. We can combine the VOS framework with the transformed primal--dual (TPD) method~\cite{ChenWei2023TPD} to improve the convergence rate to $\mathcal O(1/k^2)$. When $f^u$ is $\mu$-strongly convex, an accelerated linear convergence rate of $\mathcal O((1+\sqrt{\mu/L})^{-k})$ can be achieved; see~\cite{ChenWei2026ATPD}. 
\end{remark}

\section{Conclusion}
\label{sec:conclu}
%

This survey presented a variable and operator splitting viewpoint for accelerated first-order methods. The main message is that acceleration can be organized through a common Lyapunov mechanism: variable splitting introduces the accelerated dynamics, operator splitting separates the structured parts of the optimality system, and stable discretizations transfer continuous-time decay to implementable algorithms.


Several questions remain open. One direction is to combine the VOS Lyapunov structure with automated PEP or IQC tools to obtain sharper parameters and possibly optimal certificates. Another is to extend the framework to stochastic settings, including stochastic gradients, variance reduction, and minibatch sampling. A third direction is to develop adaptive and preconditioned VOS methods for large-scale applications, and to extend the framework further to broader monotone inclusions and non-Euclidean geometries.



\bibliographystyle{siamplain}
\bibliography{references_clean}
\end{document}